\documentclass{article}
\usepackage{tikz}
\usepackage{amsfonts}
\usepackage{amsmath}
\usepackage{amssymb}
\usepackage{mathrsfs}
\usepackage{hyperref}
\hypersetup{
 colorlinks=true,
 citecolor=cyan,
 linkcolor=magenta,
 menucolor=magenta,
}
\usepackage{bm}
\usepackage{theorem}
\usepackage{enumitem}
\usepackage{stmaryrd}
\usetikzlibrary{cd}
\usepackage[left=42truemm,right=42truemm]{geometry}

\newtheorem{thm}{Theorem}
\newtheorem{prop}[thm]{Proposition}
\newtheorem{lem}[thm]{Lemma}
\newtheorem{cor}[thm]{Corollary}
{\theorembodyfont{\normalfont}

\newtheorem{defi}[thm]{Definition}
\newtheorem{rem}[thm]{Remark}
\newtheorem{con}[thm]{Construction}
}

\numberwithin{thm}{section}  

\makeatletter
\newenvironment{proof}[1][\Proofname]{\par
  \normalfont
  \topsep6\p@\@plus6\p@ \trivlist
  \item[\hskip\labelsep{\itshape #1}\@addpunct{\itshape.}]\ignorespaces
}{%
  \endtrivlist
}
\newcommand{\Proofname}{Proof}
\makeatother
\makeatletter
\def\BOXSYMBOL{\RIfM@\bgroup\else$\bgroup\aftergroup$\fi
  \vcenter{\hrule\hbox{\vrule height.85em\kern.6em\vrule}\hrule}\egroup}
\makeatother
\newcommand{\BOX}{%
  \ifmmode\else\leavevmode\unskip\penalty9999\hbox{}\nobreak\hfill\fi
  \quad\hbox{$\Box$}}
\newcommand\qed{\BOX}

\usepackage{relsize}
\usepackage[bbgreekl]{mathbbol}
\usepackage{amsfonts}

\DeclareSymbolFontAlphabet{\mathbb}{AMSb}
\DeclareSymbolFontAlphabet{\mathbbl}{bbold}
\newcommand{\prism}{{\mathlarger{\mathbbl{\Delta}}}}
\newcommand{\str}{\sigma}

\usepackage{titlesec}
\titleformat*{\section}{\large\bfseries}

\newcommand{\moda}[2]{\widetilde{\mathcal{M}}_{\prism}(#1, #2)}
\newcommand{\modaa}[1]{\widetilde{\mathcal{M}}_{\prism}(#1)}
\newcommand{\modb}[2]{\mathcal{M}_{\prism}(#1, #2)}
\newcommand{\modbb}[1]{\mathcal{M}_{\prism}(#1)}
\newcommand{\modc}[1]{\widetilde{\mathcal{M}}(#1)}
\newcommand{\modd}[1]{\mathcal{M}(#1)}
\newcommand{\modfp}[1]{\mathcal{M}^{\textrm{fp}}(#1)}
\newcommand{\modtors}[1]{\mathcal{M}^{\textrm{tors}}(#1)}
\newcommand{\modpntors}[1]{\mathcal{M}^{p^n\textrm{-tors}}(#1)}

\newcommand{\stra}[2]{\mathscr{S}_{#1, #2}}
\newcommand{\strat}[2]{\mathscr{S}^{\text{new}}_{#1, #2}}

\newcommand{\ol}[1]{\overline{#1}}

\title{Prismatic and $q$-crystalline sites of higher level}
\author{Kimihiko Li}
\date{}

\begin{document}

\maketitle

Address: The University of Tokyo, Tokyo, Japan.

kimihiko@g.ecc.u-tokyo.ac.jp 

\begin{abstract}
In this article, we define the $m$-prismatic site and the $m$-$q$-crystalline site, which are higher level analogs of the prismatic site and the $q$-crystalline site respectively.
We prove a certain equivalence between the category of crystals on the $m$-prismatic site (resp. the $m$-$q$-crystalline site) and that on the prismatic site (resp. the $q$-crystalline site),
 which can be regarded as the prismatic (resp. the $q$-crystalline) analog of the Frobenius descent due to Berthelot and the Cartier transform due to Ogus-Vologodsky, Oyama and Xu.
 We also prove the equivalence between the category of crystals on the $m$-prismatic site and that on the $(m-1)$-$q$-crystalline site.
\end{abstract} 

\vspace{0.1in}

\vspace{0.1in}

\tableofcontents

\phantomsection
\section*{Introduction}
\addcontentsline{toc}{section}{Introduction}

Fix a prime $p$. In \cite{BS19}, Bhatt and Scholze defined two new $p$-adic cohomology \linebreak theories generalizing crystalline cohomology, called prismatic cohomology and $q$-\linebreak crystalline cohomology. 
They are defined as the cohomology of the corresponding sites called the prismatic site and the $q$-crystalline site respectively. 
The notion of \linebreak prismatic site $(X/A)_\prism$ is defined for a $p$-adic formal scheme $X$ over $A/I$, where $(A,I)$ is a bounded prism. 
Also, the $q$-crystalline site $(X/D)_{q\text{-crys}}$ is defined for a $p$ \linebreak -completely smooth affine  formal scheme $X$ over $D/I$, where $(D,I)$ is a $q$-PD pair. 
(The assumptions imply that $A,D$ are endowed with a lift $\phi$ of Frobenius on $A/pA, D/pD$ respectively. Also, $D$ is a $\mathbf{Z}_p \llbracket q-1\rrbracket$-algebra such that $(D,[p]_qD)$, where $[p]_q = \frac{q^p-1}{q-1}$, is a bounded prism and $I \subseteq \phi^{-1}([p]_qD)$. 
Moreover, when $q = 1$ in $D$, the $q$-PD pair $(D,I)$ is a PD ring.) 
Among other things, they proved the following comparison theorems:

\begin{enumerate}[label = (\arabic*)]
\item If $(D,I)$ is a $q$-PD pair and $X$ is a  $p$-completely smooth affine  formal scheme over $D/I$, then the $q$-crystalline cohomology of $X$ over $D$ is isomorphic to the prismatic cohomology of its Frobenius pullback $X' := X \underset{\text{Spf}(D/I),\phi^*}{\widehat{\times}} \text{Spf}(D/[p]_qD)$ over $D$ (Theorem 16.18 of \cite{BS19}).
\item If $(D,I)$ is a $q$-PD pair with $q = 1$ in $D$ and $X$ is an affine $p$-completely smooth formal scheme over $D/I$, then the $q$-crystalline cohomology of $X$ over $D$ is \linebreak isomorphic to the crystalline cohomology of $X$ over $D$ (a special case of \linebreak Theorem 16.14 of \cite{BS19}).
\end{enumerate}

On the other hand, for an integer $m \ge 0$, Berthelot (\cite{Ber90}, see also the works of Le Stum-Quir\'{o}s \cite{SQ01} and Miyatani \cite{Miy15}) defined the notion of crystalline cohomology of level $m$ as the cohomology of the $m$-crystalline site. 
The notion of $m$-crystalline site $(X/D)_{m\text{-crys}}$ can be defined for a smooth scheme $X$ over $D/I$, where $(D,J,I,\gamma)$ is a $p$-torsion free $p$-complete $m$-PD ring or an $m$-PD ring in which $p$ is nilpotent such that the ideal $I$ contains $p$. When $m = 0$, it coincides with the usual crystalline site. 
He proved the following results which are called the Frobenius descent:

\begin{enumerate}[label = (\arabic*)]
\item[(3)] The category of crystals on the $m$-crystalline site of $X$ over $D$ is equivalent \linebreak to that on the crystalline site of $X'$ over $D$, where 
\[
X' := X \underset{\text{Spec}(D/I),(\phi^m)^*}{\times} \text{Spec}(D/I)
\]
 is the pullback of $X$ by the $m$-fold iteration $\phi^m$ of the Frobenius lift. (Corollaire 2.3.7, Th\'{e}or\`{e}me 4.1.3 of \cite{Ber00} in local situation). (See also the equivalence $\widetilde{\sigma}^*$ in Section 4.)
\item[(4)] The crystalline cohomology of level $m$ of $X$ over $D$ is isomorphic to the crystalline cohomology of $X'$ over $D$, where $X'$ is as above (Proposition 5.4 of \cite{Miy15}, Proposition 5.5 of \cite{SQ01}).
\end{enumerate}

The purpose of the present article is to introduce the notion of $m$-prismatic site and $m$-$q$-crystalline site which are the level $m$ version of the prismatic site and the $q$-crystalline site respectively, 
and prove a prismatic and a $q$-crystalline version of the equivalence of (3) in Theorems \ref{th:mcrys}, \ref{th:mqcrys}. 
The $q$-crystalline version of (3) is actually compatible with the equivalence of (3) when $q = 1$ and $p \in I$. 
We also prove the equivalences of categories of crystals corresponding to the higher level version of (1) and (2) in Theorem \ref{th:mm1} and Proposition  \ref{prtaucrys} respectively.

Let us explain the content of each section. 
In Section 1, for a bounded prism $(A,I)$ and a smooth and separated formal scheme $X$ over $A/J$ with $J := (\phi^m)^{-1}(I)$, we first define the $m$-prismatic site $(X/A)_{m-\prism}$ (Definition \ref{def:m-pris}). Next, we
prove that the category of crystals on the $m$-prismatic site $(X/A)_{m-\prism}$ is equivalent to that on the prismatic site $(X'/A)_{\prism}$, where $X' := X \underset{\text{Spf}(A/J),(\phi^m)^*}{\widehat{\times}} \text{Spf}(A/I)$ (Theorem \ref{th:mcrys}). 
We work with the category of crystals with respect to certain categories of modules with some technical conditions (Definition \ref{def:mod}) so that our argument works. 
Our proof of the equivalence is based on the idea developed by Oyama \cite{Oya17} and Xu \cite{Xu19} in their study of Cartier transform in positive characteristic or the case modulo $p^n$: 
we define a functor $\rho : (X/A)_{m-\prism} \to (X'/A)_{\prism}$ of sites and prove that it induces an equivalence of topoi $\widetilde{(X/A)}_{m-\prism} \to \widetilde{(X'/A)}_{\prism}$ (Theorem \ref{thpristop}) and moreover that this equivalence preserves the categories of crystals on both sides.

In Section 2, for a $q$-PD pair $(D,I)$ and a smooth and separated formal \linebreak scheme $X$ over $D/J$ with $J := (\phi^m)^{-1}(I)$,
we first define the $m$-$q$-crystalline site \linebreak $(X/D)_{m\text{-}q\text{-crys}}$ (Definition \ref{def:mqcrys}). Next, we prove that the category of crystals on \linebreak $(X/D)_{m\text{-}q\text{-crys}}$ is equivalent to that on the $q$-crystalline site $(X'/D)_{q\text{-crys}}$,
where \linebreak $X' := X \underset{\text{Spf}(D/J),(\phi^m)^*}{\widehat{\times}} \text{Spf}(D/I)$ (Theorem \ref{th:mqcrys}).
The proof of the equivalence is \linebreak parallel to the one in Section 1. We also introduce the category of stratifications which is equivalent to the category of crystals on $(X/D)_{m\text{-}q\text{-crys}}$ (Definition \ref{defstra}, Proposition \ref{pr:qstra}).

In section 3, we compare the category of crystals on the $m$-prismatic site with that on the $(m-1)$-$q$-crystalline site. 
By the comparison theorems in (1) and (4) above, it is natural to regard  the prismatic cohomology  as a kind of `level $(-1)$ $q$-crystalline cohomology'. So it would be natural to compare the $m$-prismatic site with the $(m-1)$-$q$-crystalline site. 
Based on this observation, for a $q$-PD pair $(D,I)$ with $J_q := (\phi^{m-1})^{-1}(I)$, $J_\prism := (\phi^{m})^{-1}([p]_qD)$ and a smooth and separated formal scheme $X$ over $D/J_q$, we prove that the category of crystals on the $m$-prismatic site $(\widetilde{X}/D)_{m-\prism}$ \linebreak is equivalent to that on the $(m-1)$-$q$-crystalline site $(X/D)_{(m-1)\text{-}q\text{-crys}}$, 
\nolinebreak where \linebreak $\widetilde{X} := X \underset{\text{Spf}(D/J_q)}{\widehat{\times}} \text{Spf}(D/J_\prism)$ (Theorem \ref{th:mm1}). The category of stratifications we \linebreak introduced in Section 2 plays a central role in the proof.

In Section 4, we relate our equivalence in Section 2 to the Frobenius descent in (3). 
Let $(D,J,I,\gamma)$ be a $p$-torsion free $p$-complete $m$-PD ring with $p \in I$, and let $\widetilde{X}$ be a smooth and separated scheme over $D/I$. First, we give an alternative site-theoretic proof (a proof based on the idea of Oyama \cite{Oya17} and Xu \cite{Xu19}) of the equivalence in (3), 
namely, the equivalence between the category of crystals on the $m$-crystalline site $(\widetilde{X}/D)_{m\text{-crys}}$ and that on the crystalline site $(X'/D)_{\text{crys}}$, where $X' := \widetilde{X} \underset{\text{Spec}(D/I),(\phi^m)^*}{\times} \text{Spec}(D/I)$ (Corollary \ref{cor:new}). 
Note that the original \linebreak definition of the ($m$-)crystalline site is not suitable enough to apply the site-theoretic \linebreak argument of Oyama and Xu. 
Our strategy is to introduce the variants $(X/D)_{m\text{-crys,new}}$ (where $X := \widetilde{X} \underset{\text{Spec}(D/I)}{\times} \text{Spec}(D/J)$) and $(X'/D)_{\text{crys,new}}$ of the ($m$-)crystalline site (Definition \ref{defmcrysnew}) which do not change the categories of crystals and for which the site-theoretic argument works. 
Next, assuming that  $(D,I)$ is a $q$-PD pair with $q=1$ in $D$ and $p \in I$, we prove that the equivalence of categories of crystals in Section 2 is compatible with the Frobenius descent by comparing the categories of crystals on the ($m$-)$q$-crystalline site and those on the variants of the ($m$-)crystalline site.

In Section 5, we establish relationships between our result and results in the works of Xu \cite{Xu19}, Gros-Le Stum-Quir\'{o}s \cite{GSQ20b} and Morrow-Tsuji \cite{MT20}. 
We will see that our equivalences between the category of crystals on the prismatic site, that on the 1-prismatic site and that on the $q$-crystalline site are compatible with the equivalence of Cartier transform in the case modulo $p^n$ by Xu,  and the equivalence between the category of twisted hyper-stratified modules of level $(-1)$ and that of level 0 by Gros-Le Stum-Quir\'{o}s.
The equivalence of categories of crystals in Section 1 also fits naturally into the diagram involving the category of generalized representations, the category of modules with flat $q$-Higgs field and the category of modules with flat $q$-connection that appeared in the work of Morrow-Tsuji.

The original motivation of introducing $m$-crystalline cohomology would be to \linebreak develop a $p$-adic cohomology theory over a ramified base, for example, when the base is a complete discrete valuation ring $V$ of mixed characteristic $(0,p)$ in which $p$ is not a uniformizer. 
However, there exists no $\delta$-ring structure on the ring $V$ above. So our definition of the $m$-prismatic site and the $m$-$q$-crystalline site is not enough for this purpose. We hope to generalize our results to the case of a possibly ramified base in the future.

Finally, the author would like to express his sincere gratitude to his supervisor \linebreak Atsushi Shiho who patiently answered many questions and had discussions on \linebreak  the  constructions of this paper. 
The author was partially supported by WINGS-\linebreak FMSP (World-leading Innovative Graduate Study for Frontiers of Mathematical \linebreak Sciences  and Physics) program at the Graduate School of Mathematical Sciences, \linebreak the  University of Tokyo. 
This work was supported by the Grant-in-Aid for Scientific \linebreak Research  (KAKENHI  No. 22J10387) and the Grant-in-Aid for JSPS DC2 fellowship. This work was also supported by the Research Institute for Mathematical Sciences, an International Joint Usage/Research Center located in Kyoto University.

\section{The \textit{m}-prismatic site}

In this section, we first define the $m$-prismatic site, which is a higher level analog of the prismatic site defined in \cite{BS19}  as well as a prismatic analog of the level $m$ crystalline site. Next, we prove that there exists an equivalence between the category of crystals on the $m$-prismatic site of a smooth formal scheme $X$ and that on the usual prismatic site of $X'$, where $X'$ is the pullback of $X$ by the $m$-fold iteration $\phi^m$ of the Frobenius lift $\phi$ on the base prism.

We fix a non-negative integer $m$ and a bounded prism $(A,I)$. Then $I$ is closed for the $p$-adic topology of $A$ and $\phi : A \to A$ is continuous for the $p$-adic topology. Thus $(\phi ^m)^{-1}(I)$ is closed for the $p$-adic topology of $A$.

\begin{defi}
\label{def:m-pris}
Let $J = (\phi ^m)^{-1}(I)$, and let $X$ be a $p$-adic formal scheme smooth and separated over $A/J$. We define the  $m$-\textit{prismatic site} $(X/A)_{m-\prism}$ of $X$ over $A$ as follows.  
Objects are maps $(A,I) \to (E,I_E)$ of bounded prisms together with a map Spf$(E/J_E) \to X$ over $A/J$ satisfying the following condition, where $J_E = (\phi ^m)^{-1}(I_E)$:
\begin{enumerate}[label = (\roman*)]
\item[(*)] Spf$(E/J_E) \to X$ factors through some affine open Spf$(R) \subseteq X$.
\end{enumerate}
We shall often denote such an object by
\[
(\text{Spf}(E) \gets \text{Spf}(E/J_E) \to X) \in (X/A)_{m-\prism}
\]
or $(E,I_E)$ if no confusion arises. A morphism 
\[
(\text{Spf}(E') \gets \text{Spf}(E'/J_{E'}) \to X) \to (\text{Spf}(E) \gets \text{Spf}(E/J_E) \to X)
\]
 is a map of bounded prisms $(E,I_E) \to (E',I_{E'})$ over $(A,I)$ such that the induced morphism
\[
\text{Spf}(E'/J_{E'}) \to \text{Spf}(E/J_E)
\]
 is compatible with the maps $\text{Spf}(E'/J_{E'}) \to X$, $\text{Spf}(E/J_E) \to X$. 
When we denote such an object by $(E,I_E)$, we shall write $(E,I_E) \to (E',I_{E'})$ (not $(E',I_{E'}) \to (E,I_E)$) for a morphism from $(E',I_{E'})$ to $(E,I_E)$. 
A map ($E,I_E$) $\to$ ($E',I_{E'}$) in  $(X/A)_{m-\prism}$ is a cover if it is a faithfully flat map of prisms, i.e., $E'$ is $(p,I_E)$-completely faithfully flat over $E$.
\end{defi}

\begin{rem}
By Lemma 3.5 of \cite{BS19}, for an object $(E,I_E)$ of $(X/A)_{m-\prism}$, the ideal $I_E$ is always equal to the ideal $IE$.
\end{rem}

We need to check that the category $(X/A)_{m-\prism}$ endowed with the topology as defined above forms a site. Thanks to Lemma 3.8 of \cite{MT20}, we have the following lemma:

\begin{lem}
\label{lm:pushout}
Let $(E_1,I_{E_1})\xleftarrow{f}(E,I_E)\xrightarrow{g}(E_2,I_{E_2})$ be maps in $(X/A)_{m-\prism}$ such that f is a cover. Let $E_3 := E_1 \widehat{\otimes}_E E_2$, where the completion is the classical $(p,I)$-completion. 
Then  $(E_3,I_{E_3})$ is the object that represents the coproduct $(E_1,I_{E_1})\sqcup_{(E,I_E)}(E_2,I_{E_2})$ in $(X/A)_{m-\prism}$, and the canonical map $(E_2,I_{E_2}) \to (E_3,I_{E_3})$ is a cover.
\end{lem}

By Lemma \ref{lm:pushout}, the set of covers in Definition \ref{def:m-pris} actually forms a pretopology on the category $(X/A)_{m-\prism}$.

\begin{rem}
\label{rem:aff}
When $m=0$, we denote the site $(X/A)_{m-\prism}$ simply by $(X/A)_{\prism}$ and call it the prismatic site. This is equal to the prismatic site defined in Definition 4.1 in \cite{BS19} except that the technical condition (*) in Definition \ref{def:m-pris}, which we imposed to facilitate some arguments below, is not assumed in \cite{BS19}. One easily checks that the topos is unchanged even if we do not impose the condition (*).
\end{rem}

In order to establish the equivalence between the category of crystals on the $m$-prismatic site and that on the usual prismatic site, we first construct a functor between these sites.

\begin{con}
\label{conpris}
Under the notation and assumption in Definition \ref{def:m-pris}, let $X'$ be $X$$\underset{\mathrm{Spf}(A/J),(\phi ^m)^*}{\widehat{\times}}$ Spf$(A/I)$. Then we have a diagram:
\[
\begin{tikzcd}
\mathrm{Spf}(A)\arrow[r,hookleftarrow] & \mathrm{Spf}(A/J)\arrow[dr,phantom, "\Box"] & X\arrow[l]\\
\mathrm{Spf}(A)\arrow[r,hookleftarrow]\arrow[u,"(\phi ^m)^*"] & \mathrm{Spf}(A/I)\arrow[u,"(\phi ^m)^*"] & X'.\arrow[l]\arrow[u]
\end{tikzcd}
\]

We define a functor $\rho$ from the $m$-prismatic site of $X$ over $A$ to the usual prismatic site of $X'$ over $A$ in the following way: For an object (Spf$(E) \gets$ Spf$(E/J_E) \to X)$ of $(X/A)_{m-\prism}$, 
we define the object $\rho$(Spf$(E) \gets$ Spf$(E/J_E) \to X)$ of $(X'/A)_{\prism}$ by 
\[
(\text{Spf}(E) \gets \text{Spf}(E/I_E) \xrightarrow{f} X'),
\]
where the right map $f$ is defined as follows:
\[
\mathrm{Spf}(E/I_E) \xrightarrow{g} \mathrm{Spf}(E/J_E) \underset{\mathrm{Spf}(A/J),(\phi ^m)^*}{\widehat{\times}} \mathrm{Spf}(A/I) \to X \underset{\mathrm{Spf}(A/J),(\phi ^m)^*}{\widehat{\times}} \mathrm{Spf}(A/I)=X'.
\]
Here the first map $g$ is induced by the map of rings
\[
E/J_E \underset{A/J,\phi^m}{\widehat{\otimes}} A/I \to E/I_E ; \ \ \ \ \  e \otimes a \mapsto \phi^m(e)a.
\]
This defines the functor $\rho:(X/A)_{m-\prism} \to (X'/A)_{\prism}$.
\end{con}

Next, we want to show that $\rho$ induces an equivalence of topoi. We follow the proof of Theorem 9.2 of \cite{Xu19}, where the author proved the analogous result for the Oyama topoi. The main properties of the functor $\rho$ that we need are summarized in the following propositions.

\begin{prop}
\label{pr:1}
{\rm (cf. \cite{Xu19} 9.3 (i))} The functor $\rho$ is fully faithful.
\end{prop}
\begin{proof}
The functor $\rho$ is clearly faithful. To prove fullness, suppose that 
\[
\alpha : \rho(\text{Spf}(E') \gets \text{Spf}(E'/J_{E'}) \to X) \to \rho(\text{Spf}(E) \gets \text{Spf}(E/J_E) \to X)
\]
 is a map in  $(X'/A)_{\prism}$. We consider the diagram
\[
\begin{tikzcd}
\text{Spf}(E'/I_{E'}) \arrow[rrr,"(\phi^m)^*"]\arrow[dr]\arrow[dd,"\overline{\alpha}"] &&&\text{Spf}(E'/J_{E'})\arrow[dl]\arrow[dd,"\overline{\alpha}"] \\
&X'\arrow[r]&X& \\
\text{Spf}(E/I_{E}) \arrow[rrr,"(\phi^m)^*"]\arrow[ur]&&&\text{Spf}(E/J_{E}), \arrow[ul]
\end{tikzcd}
\]
where the morphisms $\overline{\alpha}$ are induced by $\alpha$ and the morphisms $(\phi^m)^*$ are induced by $\phi^m$ on $E'$ and $E$. 
It is enough to check that the right triangle is commutative. Since the left triangle, the outer square and the two trapezoids are commutative, it suffices to prove the following claim:
\bigskip 

\noindent
{\bf claim.} \, Let $f_i$ : $\text{Spf}(E'/J_{E'}) \to X (i = 1,2)$ be maps which factor through some affine opens $\text{Spf}(R_i) \subseteq X (i = 1,2)$ respectively, and satisfies 
\[
f_1 \circ (\phi^m)^* = f_2 \circ (\phi^m)^* : \text{Spf}(E'/I_{E'}) \to X.
\]
 Then $f_1 = f_2$.
\bigskip 

We prove the claim. Let $\text{Spf}(R_3) = \text{Spf}(R_1) \cap \text{Spf}(R_2)$ (note that, since $X$ is separated by assumption, the intersection on the right hand side is affine).
Then the map $f_1 \circ (\phi^m)^* = f_2 \circ (\phi^m)^*$ factors through  $\text{Spf}(R_3)$. Since $(\phi^m)^*$ is  homeomorphic as a map of topological spaces, this implies that both $f_1$ and $f_2$ factor through $\text{Spf}(R_3)$. 
If $f_i^*$ denote the maps $R_3 \to E'/J_{E'}$ corresponding to $f_i (i = 1,2)$, then we have the equality of the map of rings $\phi^m \circ f_1^* = \phi^m \circ f_2^* : R_3 \to E'/I_{E'}$. 
Since the map $\phi^m : E'/J_{E'} \to E'/I_{E'}$ is injective, we conclude that $f_1^* = f_2^*$, and hence $f_1 = f_2$.
\qed
\end{proof}

\begin{prop}
\label{pr:2}
{\rm (cf. \cite{Xu19} 9.3 (ii))} The functor $\rho$ is continuous.
\end{prop}
\begin{proof}
By definition, $\alpha$ : $(E,I_E)$ $\to$ $(E',I_{E'})$ is a cover in $(X/A)_{m-\prism}$ if and only if \linebreak $\rho(\alpha)$ : $\rho$$(E,I_E)$ $\to$ $\rho(E',I_{E'})$ is a cover in $(X'/A)_{\prism}$ (note that the definition of covers in the prismatic site is the $m=0$ case of Definition \ref{def:m-pris}). 
On the other hand, if
\[
(E_1,IE_1)\xleftarrow{f}(E,IE)\xrightarrow{g}(E_2,IE_2)
\]
 are maps in $(X/A)_{m-\prism}$ such that the map $f$ is a cover, then by Lemma \ref{lm:pushout}, we see that $\rho(E_1\sqcup_{E}E_2) \to \rho(E_1)\sqcup_{\rho(E)}\rho(E_2)$ is an isomorphism. These imply that $\rho$ is continuous, as desired.
\qed
\end{proof}

\begin{prop}
\label{pr:3}
{\rm (cf. \cite{Xu19} 9.3 (ii))} The functor $\rho$ is cocontinuous.
\end{prop}
\begin{proof}
Suppose that 
\[
\alpha' : (\text{Spf}(E') \gets \text{Spf}(E'/I_{E'}) \to X') \to \rho(\text{Spf}(E) \gets \text{Spf}(E/J_E) \allowbreak \to X)
\]
 is a cover in $(X'/A)_{\prism}$. The condition (*) in Definition \ref{def:m-pris} ensures that the map Spf($E/J_E) \to X$ factors through some affine open $\text{Spf}(R) \subseteq X$. Then $\alpha'$ induces the commutative diagram:
\[
\begin{tikzcd}
R' = R  \underset{A/J,\phi^m}{\widehat{\otimes}} A/I \arrow[r,"g \otimes \text{id}"]\arrow[drr,bend right=10,"g'"] & E/J_E \underset{A/J,\phi^m}{\widehat{\otimes}} A/I \arrow[r] &E/I_E \arrow[d,"\overline{\alpha}'"] & E\arrow[l, two heads]\arrow[d,"\alpha'"]\\
&&E'/I_{E'} & E',\arrow[l, two heads]
\end{tikzcd}
\]
where $g$ (resp. $g'$) is the map induced by $\text{Spf}(E/J_E) \allowbreak \to X$ (resp. $\text{Spf}(E'/I_{E'}) \to X'$). 
If we define $g''$ as the composite $R \to R' \xrightarrow{g'} E'/I_{E'}$, 
the commutativity of the above diagram implies that the solid arrows in
\[
\begin{tikzcd}
R  \arrow[r,"g"]\arrow[dr,dashed,bend right=20]\arrow[drr,bend right=5,near end,"g''"] & E/J_E \arrow[d,"\overline{\alpha}'",crossing over] \arrow[r,hookrightarrow,"\phi ^m"] &E/I_E \arrow[d,"\overline{\alpha}'"] & E\arrow[l, two heads]\arrow[d,"\alpha'"]\\
&E'/J_{E'} \arrow[r,hookrightarrow,"\phi ^m"']&E'/I_{E'} & E'\arrow[l, two heads]
\end{tikzcd}
\]
commute. Then we can define the dotted arrow as the composite $\overline{\alpha}' \circ g$ so that the \linebreak diagram  commutes and that it defines 
an object (Spf$(E') \gets$ Spf($E'/J_{E'}) \to X$) and a  \linebreak morphism 
\[
(\text{Spf}(E') \gets \text{Spf}(E'/J_{E'}) \to X) \to (\text{Spf}(E) \gets \text{Spf}(E/J_{E}) \to X)
\]
 in $(X/A)_{m-\prism}$. 
If we denote it by $\alpha$, then $\rho (\alpha) = \alpha'$. The map $\alpha$ is also a cover by assumption. These imply that $\rho$ is cocontinuous, as desired.
\qed
\end{proof}

In order to prove another important proposition to prove the equivalence of topoi, we prepare several lemmas.

\begin{lem}
\label{lemlift}
Let $A$ be an $I$-adically complete ring where $I$ is a finitely generated ideal. 
Let $J \subseteq A$ be a closed ideal such that the image of $J$ in $A/I$ is a nil ideal, and let $\overline{f} : \overline{U} \to \mathrm{Spf}(A/J)$ be a smooth morphism $($resp. an open immersion$)$ with $\overline{U}$ affine. 
Then there is a unique smooth morphism $($resp. open immersion$)$ $f : U \to \mathrm{Spf}(A)$ which lifts $\overline{f}$. Moreover, $U$ is affine.
\end{lem}
\begin{proof}
We follow the proof of Lemma 2.3.14 of \cite{NS08} where the authors proved the analogous results for log schemes. We may replace $A$ by $A/I^n$ and Spf by Spec to prove the lemma. Then $J$ is a nil ideal of $A$ and the morphism $\overline{f} : \overline{U} \to \text{Spec}(A/J)$ is a smooth morphism (resp. an open immersion) of schemes. 

We first treat the case that $\overline{f}$ is smooth.
Express $J$ as the inductive limit of the inductive system $(J_\lambda)_\lambda$ of finitely generated nilpotent ideals of $A$: $J = \varinjlim_\lambda J_\lambda$. Since \[\mathrm{Spec}(A/J) = \varprojlim_\lambda \mathrm{Spec}(A/J_\lambda),\] by Th\'{e}or\`{e}me 8.8.2 (ii) and Proposition 17.7.8 of \cite{EGA4}, there exists a smooth scheme $\overline{U_\lambda}$ over $\mathrm{Spec}(A/J_\lambda)$ for some $\lambda$ such that $\overline{U} = \overline{U_\lambda} \underset{\mathrm{Spec}(A/J_\lambda)}{\times} \mathrm{Spec}(A/J)$. By Th\'{e}or\`{e}me 8.10.5 (viii) of \cite{EGA4}, we may assume $\overline{U_\lambda}$ is affine. By replacing $\overline{U} \to \text{Spec}(A/J)$ by $\overline{U_\lambda} \to \mathrm{Spec}(A/J_\lambda)$, we may assume that $J$ is a finitely generated nilpotent ideal. Next, fix a generator $x_1, \dots , x_n$ of $J$ and write $A$ as the inductive limit of the inductive system $(A_\lambda)_\lambda$ of Noetherian subrings containing $x_1, \dots , x_n$. If we write $J_\lambda$ for the ideal of $A_\lambda$ generated by $x_1, \dots , x_n$, we have 
\[
\text{Spec}(A/J) = \varprojlim_\lambda \mathrm{Spec}(A_\lambda/J_\lambda).
\]
Then, by a similar argument as above, we may replace $J \subseteq A$ by $J_\lambda \subseteq A_\lambda$ for some $\lambda$ and so we may assume furthermore that $A$ is Noetherian. Then the existence and the uniqueness of $U$ follows from I\hspace{-1pt}I\hspace{-1pt}I, Corollaire 6.8 of \cite{SGA1}.

It remains to check that the scheme $U$ is affine. We can reduce to the case where $J$ is nilpotent by the argument above, and then we may assume that $J^2 = 0$.  Let $\mathscr{I}$ be a coherent ideal sheaf of $\mathscr{O}_U$. By the proof of Theorem 3.7 of \cite{Har77}, for a scheme $X$, the following conditions on $X$ are equivalent:
\begin{enumerate}
\item $X$ is affine.
\item $H^1(X,\mathscr{I}) = 0$ for all coherent  ideal sheaves $\mathscr{I}$.
\end{enumerate}
So we need to prove that $H^1(U,\mathscr{I}) = 0$. We consider the exact sequence
\[
 0 \to J \mathscr{I} \to \mathscr{I} \to \mathscr{I}/J \mathscr{I} \to 0. 
\]
As $\overline{U}$ is affine, we see that $H^1(U,\mathscr{I}/J \mathscr{I}) = H^1(\overline{U},\mathscr{I}/J \mathscr{I}) = 0$. Similarly, one can show that $H^1(U, J \mathscr{I}) = 0$. Hence $H^1(X,\mathscr{I}) = 0$.

Next we  consider the case that $\overline{f}$ is an open immersion.
Then the existence and the uniqueness of $f$ follow as $\text{Spec}(A/J)$ is homeomorphic to $\text{Spec}(A)$. 
To prove that $U$ is affine, we may assume that $A$ is Noetherian and $J$ is nilpotent, and then the above argument works. So we are done.
\qed
\end{proof}

\begin{lem}
\label{prislift}
Let $(E,I_E)$ be a bounded prism. Let $\overline{f}$ be an open immersion 
\[
\overline{f} : \mathrm{Spf}(\overline{E_i}) \to \mathrm{Spf}(E/I_E)
\]
 or an open immersion 
\[
\overline{f} : \mathrm{Spf}(\overline{E_i}) \to \mathrm{Spf}(E/J_E).
\]
 Then there is a unique  open immersion  $f : \mathrm{Spf}(E_i) \to \mathrm{Spf}(E)$ which lifts $\overline{f}$. Moreover, the corresponding map of rings induces a map $(E,I_E) \to (E_i,I_EE_i)$ of bounded prisms.
\end{lem}
\begin{proof}
We first check that the images of $I_E, J_E$ in $E/(p,I_E)$ are nil ideals. The assertion is trivial for $I_E$. For $J_E$, it follows immediately from the definitions of $J_E$ and the Frobenius lift $\phi$. Then the existence, the uniqueness and the affinity of the lifting follow from Lemma \ref{lemlift}. We denote the lifting by $f : \mathrm{Spf}(E_i) \to \mathrm{Spf}(E)$.

As $f$ is an open immersion, there exists an open cover 
\[
\mathrm{Spf}(E_i) = \bigcup_{j=1}^N \mathrm{Spf}(\widehat{E}_{g_j}),
\] 
where $\widehat{E}_{g_j}$ is the classical $(p,I_E)$-completion of the localization $E_{g_j}$ of $E$ by $g_j \in E$. As $E_{g_j}$ is an \'{e}tale $E$-algebra, the derived  $(p,I_E)$-completion $E_{g_j}^D$ of  $E_{g_j}$ is $(p,I_E)$-completely \'{e}tale and it admits a unique $\delta$-structure compatible with the one on $E$ by Lemma 2.18 of \cite{BS19}. Then, $(E_{g_j}^D, I_E E_{g_j}^D)$ is a bounded prism by Lemma 3.7 (2),(3) of \cite{BS19}, and so $\widehat{E}_{g_j}$ is equal to $E_{g_j}^D$ by Lemma 3.7 (1) of \cite{BS19}. We see also that $\widehat{E}_{g_{j_1} \dots g_{j_n}}$ is equal to $E_{g_{j_1} \dots g_{j_n}}^D$ by the same reason.

By the sheaf property and the vanishing of higher cohomologies of the structure sheaf of $\text{Spf}(E_i)$, we have a quasi-isomorphism 
\[
E_i \xrightarrow{\simeq} [ \underset{j}{\prod} \widehat{E}_{g_j} \to \underset{j_1<j_2}{\prod} \widehat{E}_{g_{j_1}g_{j_2}} \to \dots \to \widehat{E}_{g_{1} \dots g_{N}} ].
\]
We write $K^\bullet$ for the complex on the right hand side. As each $\widehat{E}_{g_{j_1} \dots g_{j_n}} = E_{g_{j_1} \dots g_{j_n}}^D$ is $(p,I_E)$-completely flat over $E$, $K^\bullet \otimes^L_E N$ is concentrated in degree $\geq 0$ for any  $(p,I_E)$-torsion $E$-module $N$. On the other hand, $E_i \otimes^L_E N$ is concentrated in degree $\leq 0$ by definition. Hence $E_i \otimes^L_E N$ is concentrated in degree 0. Then, since \[E_i \otimes^L_E E/(p,I_E) = E_i/(p,I_E)\] is \'{e}tale over $E/(p,I_E)$, $E_i$ is  $(p,I_E)$-completely \'{e}tale over $E$. Thus it admits a unique $\delta$-structure compatible with the one on $E$ by Lemma 2.18 of \cite{BS19} and then $(E_i, I_EE_i)$ is a bounded prism by Lemma 3.7 (2),(3) of \cite{BS19}.
\qed
\end{proof}

The next result will be also used to prove the equivalence of topoi.

\begin{prop}
\label{pr:4}
{\rm (cf. \cite{Xu19} 9.8 (i))} Let $(E',I_{E'})$ be an object in $(X'/A)_{\prism}$. Then there exists an object $(E,I_{E})$ in $(X/A)_{m-\prism}$ and a cover of the form 
\[(E',I_{E'}) \to \rho (E,I_{E}).
\]
\end{prop}
\begin{proof}
Let $X = \bigcup_i \text{Spf}(R_i)$ be a finite affine open cover of $X$. By the definition of $X'$ in Construction \ref{conpris}, we see that $\text{Spf}(R_i) \underset{X}{\widehat{\times}} X' = \text{Spf}(R_i \underset{A/J,\phi^m}{\widehat{\otimes}} A/I)$. We denote this formal scheme by $\text{Spf}(R'_i)$.
On the other hand, the condition (*) in Definition \ref{def:m-pris} ensures that the map $\text{Spf}(E'/I_{E'}) \to X'$ factors through some affine open $\text{Spf}(R'') \subseteq X'$. As $X'$ is separated by assumption, we see that  $\text{Spf}(R_i') \cap \text{Spf}(R'') = \text{Spf}(R^*_i)$ for some $R^*_i$. Then the formal scheme $\text{Spf}(R'_i) \widehat{\times}_{X'} \text{Spf}(E'/I_{E'})$ is identified with $\text{Spf}(R_i' \widehat{\otimes}_{R^*_i} E'/I_{E'})$. In particular, it is affine. We denote this affine open subscheme of $\text{Spf}(E'/I_{E'})$ by $\text{Spf}(\overline{E'_i})$. 
By Lemma \ref{prislift}, there is a unique bounded prism $(E'_i, I_{E'_i}) \in (X'/A)_\prism$ for which the corresponding formal scheme $\text{Spf}(E'_i)$ is an affine open formal subscheme of $\text{Spf}(E')$ and lifts $\text{Spf}(\overline{E'_i})$, namely $E'_i/I_{E'_i} = \overline{E'_i}$. By construction, we see that 
\[
(E',I_{E'}) \to \prod_i(E'_i,I_{E'_i})
\]
 is a cover in $(X'/A)_\prism$. To prove the proposition, 
we may replace $(E',I_{E'})$ by $(E'_i,I_{E'_i})$ to assume that the structure morphism $\text{Spf}(E'/I_{E'}) \to X'$ factors through an affine open $\text{Spf}(R') \subseteq X'$ such that $R'$ is of the form $R \underset{A/J,\phi^m}{\widehat{\otimes}} A/I$ for some affine open $\text{Spf}(R) \subseteq X$.

As the map $A/J \to R$ is smooth, Tag 00TA of \cite{Sta} tells us that there exists an open cover of $\text{Spf}(R)$ by standard opens $\text{Spf}(\widehat{R_g})$ such that each $\widehat{R_g}/p$ is standard smooth over $A/(p,J)$.
After refining the given cover of $X$ in the previous paragraph we may assume that $R/p$ is standard smooth over $A/(p,J)$.
By Tag 00T7 of \cite{Sta},  there exists a surjection $A[x_1, \dots, x_n]^\wedge \to R$ whose kernel is the ideal ($J, y_1, \dots, y_r$) such that 
\[
\overline{y_1}, \dots, \overline{y_r} \in A/(p,J)[x_1, \dots, x_n]
\]
 form a regular sequence. 
On the other hand, there is a natural map 
\[
A[x_1, \dots, x_n]^\wedge \to A\{x_1, \dots, x_n\}^\wedge,
\]
 where the symbol \{\} denotes the adjoining of elements in the theory of $\delta$-rings (see Notation 2.8 of \cite{BS19}). By Corollary \ref{cor:reg seq} below about the regular sequence, we see that the sequence 
\[
\phi^m(y_1), \dots, \phi^m(y_r) \in A\{x_1, \dots, x_n\}^\wedge
\]
 is ($p,I$)-completely regular relative to $A$. 
Then we can construct a map 
\[
A\{x_1, \dots, x_n\}^\wedge \to S \stackrel{\mathrm{def}}{=} ((A\{x_1, \dots, x_n\})^\wedge\{ \tfrac{K}{I}\})^\wedge
\]
 to the prismatic envelope by Proposition 3.13 of \cite{BS19}, where $K$ denotes the ideal ($I,\phi^m(y_1), \dots, \phi^m(y_r)) \subseteq A\{x_1, \dots, x_n\}^\wedge$. 
By the construction of the prismatic \linebreak envelope, $\phi^m(y_i) \in I_S$ and so $y_i \in J_S = (\phi^m)^{-1}(I_S)$ for all $i$. This gives a map $R \to S/J_S$.

Next, let $A[x_1, \dots, x_n]^\wedge \to R \underset{A,\phi^m}{\widehat{\otimes}} A$ be the base change of $A[x_1, \dots, x_n]^\wedge \to R$ in \linebreak the previous paragraph along $\phi^m : A \to A$. Passing to the quotient then induces $A[x_1, \dots, x_n]^\wedge \to R'$.
The kernel of this map is the ideal ($I, y'_1, \dots, y'_r$), where the $y_{i}'$'s are the images of the $y_i$'s under the map $A[x_1, \dots, x_n]^\wedge \to A[x_1, \dots, x_n]^\wedge$ sending \linebreak $\sum_{\underline{j}} \alpha_{\underline{j}} x^{\underline{j}}$ to $\sum_{\underline{j}} \phi^m(\alpha_{\underline{j}}) x^{\underline{j}}$ (here we write $\underline{j} := (j_1, \dots, j_n)$ for the multi-index). 
By the definition of $y_i$'s, the sequence $\overline{y_1'}, \dots, \overline{y_r'} \in A/(p,I)[x_1, \dots, x_n]$ is a regular  sequence. \linebreak By Corollary \ref{cor:reg seq} below, the sequence $y_1', \dots, y_r' \in A\{x_1, \dots, x_n\}^\wedge$ is ($p,I$)-completely regular relative to $A$. Then we can construct a map 
\[
A\{x_1, \dots, x_n\}^\wedge \to S' \stackrel{\mathrm{def}}{=} ((A\{x_1, \dots, x_n\})^\wedge\{ \tfrac{K'}{I}\})^\wedge
\]
to the prismatic envelope, where $K'$ denotes the ideal $(I,y'_1, \dots, y'_r) \subseteq A\{x_1, \dots, x_n\}^\wedge$. Then we have a diagram:
\[
\begin{tikzcd}[column sep=0.4cm]
A/I  \arrow[r] & R'&&\\
A \arrow[u, two heads]\arrow[r] &A[x_1, \dots, x_n]^\wedge \arrow[u, two heads]\arrow[r]& A\{x_1, \dots, x_n\}^\wedge \arrow[r] & S' \stackrel{\mathrm{def}}{=} ((A\{x_1, \dots, x_n\})^\wedge \{ \tfrac{K'}{I}\})^\wedge.
\end{tikzcd}
\]
By the definitions of $y_i$ and $y_i'$, we have a map of $\delta$-$A$-algebras $(\phi^m)' : S' \to S$ sending $x_i$ to $\phi^m(x_i)$. By Proposition \ref{prop:mprismff} below, we see that $(\phi^m)'$ is $(p,I)$-completely faithfully flat.

Let ($E',I_{E'}) \in (X'/A)_{\prism}$ be as above. We have a map $f_1 : A[x_1, \dots, x_n]^\wedge \to E'/I_{E'}$ determined as the composition of the map $A[x_1, \dots, x_n]^\wedge \to R'$ in the above diagram with $R' \to E'/I_{E'}$. As $A[x_1, \dots, x_n]^\wedge$ is the completion of a polynomial ring, one can choose a map $f_2 : A[x_1, \dots, x_n]^\wedge \to E'$ lifting $f_1$.
As $E'$ is a $\delta$-$A$-algebra, $f_2$ extends uniquely to a $\delta$-$A$-algebra map $f_3 : A\{x_1, \dots, x_n\}^\wedge \to E'$.
By construction, this extension carries $K'$ into $I_{E'}$. By the universal property of $S'$, $f_3$ extends uniquely to a $\delta$-$A$-algebra map $g : S' \to E'$. 
If we set $h : E' \to E' \widehat{\otimes}_{S'} S$ to be the base change of $(\phi^m)' : S' \to S$ along $g$, then by $(p,I)$-complete faithful flatness of $(\phi^m)'$, we see that the same holds true for $h$.

It remains to check that the map $h$ defines a morphism
\[
(E',I_{E'}) \to \rho(E' \widehat{\otimes}_{S'} S, I_{E' \widehat{\otimes}_{S'} S})
\]
in $(X'/A)_{\prism}$. To see this, it is enough to check that the bottom right square in the following diagram is commutative, in which all the other squares are commutative.
{\fontsize{9pt}{10pt}\selectfont
\[
\begin{tikzcd}[column sep=0.35cm]
&A[x_1, \dots, x_n]^\wedge  \arrow[r]\arrow[d] & S\arrow[r,"i_2"]\arrow[d]&    E' \widehat{\otimes}_{S'} S \arrow[d] \\
A[x_1, \dots, x_n]^\wedge \arrow[r, two heads] \arrow[d,"f_2"] & R' = R  \underset{A/J,\phi^m}{\widehat{\otimes}} A/I \arrow[r]\arrow[d]& S/J_S  \underset{A/J,\phi^m}{\widehat{\otimes}} A/I \arrow[r,"\overline{i_2}"] & E' \widehat{\otimes}_{S'} S/J_{E' \widehat{\otimes}_{S'} S}  \underset{A/J,\phi^m}{\widehat{\otimes}} A/I \arrow[d] \\
E'\arrow[r, two heads]&E'/I_{E'}\arrow[rr,"\overline{h}"]&&E' \widehat{\otimes}_{S'} S/I_{E' \widehat{\otimes}_{S'} S}
\end{tikzcd}
\]
}

\noindent
Here $i_2$ is the map to the second component of the coproduct in Lemma \ref{lm:pushout}. We can check the commutativity of the bottom right square by tracing the elements $x_i$'s:
\[
\begin{tikzcd}
&x_i \arrow[r, mapsto]\arrow[d, mapsto] & x_i\arrow[r, mapsto]&1 \otimes x_i \arrow[d, mapsto] \\
x_i \arrow[r, mapsto] \arrow[d, mapsto] & \overline{x_i} \otimes 1 \arrow[rr, mapsto]&& \overline{1 \otimes x_i} \otimes 1 \arrow[d, mapsto] \\
f_2(x_i)\arrow[r, mapsto]&\overline{f_2(x_i)}\arrow[rr, mapsto]&&\overline{1 \otimes \phi^m(x_i)} = \overline{f_2(x_i) \otimes 1},
\end{tikzcd}
\]
so the proposition follows.
\qed
\end{proof}

We prove the claims which were used in the proof above. We first review several notions that appeared in Section 2.6 in \cite{BS19}. 

A map $A \to B$ of simplicial commutative rings is called flat if $\pi_0(A) \to \pi_0(B)$ is flat and $\pi_i(A) \otimes_{\pi_0(A)} \pi_0(B) \to \pi_i(B)$ is an isomorphism for all $i$. 
This can be checked after the derived base change along $A \to \pi_0(A)$. For a commutative ring $A$, a finitely generated ideal $I = (f_1, \dots,  f_n) \subseteq A$ and an $I$-completely flat commutative $A$-algebra $B$, a sequence $x_1, \dots,  x_r \in B$ is $I$-completely regular relative to $A$ if the map of simplicial commutative rings
\[
A \otimes^L_{\mathbb{Z}[f_1, \dots,  f_n]} \mathbb{Z} \to B \otimes^L_{\mathbb{Z}[f_1, \dots,  f_n, x_1, \dots,  x_r]} \mathbb{Z}
\]
is flat. Here the maps $\mathbb{Z}[f_1, \dots,  f_n] \to \mathbb{Z}$, $\mathbb{Z}[f_1, \dots,  f_n, x_1, \dots,  x_r] \to \mathbb{Z}$ are defined by sending each $f_i, x_i$ to 0.

\begin{lem}
\label{lem:reg seq}
Let $A, I = (f_1, \dots,  f_n) \subseteq A$ and $B$ be as above and let $x_1, \dots,  x_r \in B$ be a sequence of elements satisfying the following conditions:
\begin{enumerate}
\item The images $\overline{x_1}, \dots,  \overline{x_r}$ of $x_1, \dots,  x_r$ in $B/IB$ form a regular sequence.
\item $B/(I, x_1, \dots,  x_r)$ is flat over $A/I$.
\end{enumerate}
Then $x_1, \dots,  x_r$ is $I$-completely regular relative to $A$.
\end{lem}
\begin{proof}
We prove the flatness of the map
\[
A \otimes^L_{\mathbb{Z}[f_1, \dots,  f_n]} \mathbb{Z} \to B \otimes^L_{\mathbb{Z}[f_1, \dots,  f_n, x_1, \dots,  x_r]} \mathbb{Z}.
\]
This can be checked after the derived base change along
\[
A \otimes^L_{\mathbb{Z}[f_1, \dots,  f_n]} \mathbb{Z} \to \pi_0(A \otimes^L_{\mathbb{Z}[f_1, \dots,  f_n]} \mathbb{Z}) = A/I,
\]
which is the composition of the maps
\[
A \otimes^L_{\mathbb{Z}[f_1, \dots,  f_n]} \mathbb{Z} \xrightarrow{f_1} A/I \otimes^L_{\mathbb{Z}[f_1, \dots,  f_n]} \mathbb{Z}  \xrightarrow{f_2} A/I,
\]
where $f_1$ is the map induced by the projection $A \to A/I$ and $f_2$ is the map taking the degree 0 part of the Koszul complex. By taking the base change along $f_1$, we obtain the map
\begin{align*}
A/I \otimes^L_{\mathbb{Z}[f_1, \dots,  f_n]} \mathbb{Z} &\to (A/I \otimes^L_{A} B) \otimes^L_{\mathbb{Z}[f_1, \dots,  f_n, x_1, \dots,  x_r]} \mathbb{Z} \\
&\cong B/IB \otimes^L_{\mathbb{Z}[f_1, \dots,  f_n, x_1, \dots,  x_r]} \mathbb{Z} \\
&\cong (B/IB \otimes^L_{\mathbb{Z}[f_1, \dots,  f_n, x_1, \dots,  x_r]} \mathbb{Z}[x_1, \dots,  x_r]) \otimes^L_{\mathbb{Z}[x_1, \dots,  x_r]}  \mathbb{Z},
\end{align*}
where the first isomorphism follows from the $I$-complete flatness of $B$ over $A$. Then, taking the base change along $f_2$, we obtain the map
\[
A/I \to B/IB \otimes^L_{\mathbb{Z}[x_1, \dots,  x_r]} \mathbb{Z}.
\]
Note that the complex $B/IB \otimes^L_{\mathbb{Z}[x_1, \dots,  x_r]} \mathbb{Z}$ is the Koszul complex corresponding to $B/IB$ and $\overline{x_1}, \dots,  \overline{x_r}$.
As $\overline{x_1}, \dots,  \overline{x_r}$ form a regular sequence in $B/IB$, we see that the \linebreak complex $B/IB \otimes^L_{\mathbb{Z}[x_1, \dots,  x_r]} \mathbb{Z}$ is identified with $B/(I, x_1, \dots,  x_r)$, which is flat over $A/I$ by assumption, as desired.
\qed
\end{proof}

\begin{cor}
\label{cor:reg seq}
With notation as in the proof of Proposition \ref{pr:4}, the sequence \linebreak $\phi^m(y_1), \dots,  \phi^m(y_r) \in A\{x_1, \dots, x_n\}^\wedge$ and the sequence $y_1', \dots,  y_r' \in A\{x_1, \dots, x_n\}^\wedge$ are $(p,I)$-completely regular relative to $A$.
\end{cor}
\begin{proof}
We first treat the case of the sequence $y_1', \dots,  y_r'$. Note that the map 
\[
A/(p,I)[x_1, \dots, x_n] \to A/(p,I)\{x_1, \dots, x_n\}
\]
 is flat. Recall that $\overline{y_1'}, \dots,  \overline{y_r'} \in A/(p,I)[x_1, \dots, x_n]$
 form a regular sequence and the ring $A[x_1, \dots, x_n]/(p,I,y_1', \dots,  y_r') = R'/pR'$ is smooth (and hence flat) over $A/(p,I)$, so the claim follows from Lemma \ref{lem:reg seq}.

Next we prove the claim for the sequence $\phi^m(y_1), \dots,  \phi^m(y_r)$. Recall that the sequence
\[
\overline{y_1}, \dots,  \overline{y_r} \in A/(p,J)[x_1, \dots, x_n]
\]
is a regular sequence and $A[x_1, \dots, x_n]/(p,J,y_1, \dots,y_r) = R/p$ is smooth over $A/(p,J)$. 
If $J^{(p^m)}$ denotes the ideal generated by $x^{p^m}$ for all elements $x$ of $J$, then we have $(p,J^{(p^m)}) \subseteq (p,I)$.
 Since $(p,J)$ is a nil ideal in  $A/(p,J^{(p^m)})[x_1, \dots, x_n]$,  Jacobian criterion for smoothness implies that \[\overline{y_1}, \dots,  \overline{y_r} \in A/(p,J^{(p^m)})[x_1, \dots, x_n]\] form a  regular  sequence and that the quotient
$A[x_1, \dots, x_n]/(p,J^{(p^m)},y_1, \dots,y_r)$
 is smooth over $A/(p,J^{(p^m)})$. 
Then it implies that $\overline{y_1}, \dots,  \overline{y_r} \in A/(p,I)[x_1, \dots, x_n]$ form a  regular sequence and that the quotient $A[x_1, \dots, x_n]/(p,I,y_1, \dots,y_r)$ is smooth over $A/(p,I)$. 
As $\overline{\phi^m(y_i)} = \overline{y_i}^{p^m}$ in $A/(p,I)\{x_1, \dots, x_n\}$, we see that 
\[
\overline{\phi^m(y_1)}, \dots,  \overline{\phi^m(y_r)} \in A/(p,I)\{x_1, \dots, x_n\}
\]
 form a regular sequence and that $A\{x_1, \dots, x_n\}/(p,I,\phi^m(y_1), \dots,  \phi^m(y_r))$
 is flat over $A/(p,I)$. By Lemma \ref{lem:reg seq}, we conclude that the sequence 
\[
\phi^m(y_1), \dots,  \phi^m(y_r) \in A\{x_1, \dots, x_n\}^\wedge
\]
 is $(p,I)$-completely regular relative to $A$.
\qed
\end{proof}

\begin{prop}
\label{prop:mprismff}
With notation as in the proof of Proposition \ref{pr:4}, the map 
\[
(\phi^m)' : S' \to S
\]
 is $(p,I)$-completely faithfully flat.
\end{prop}
\begin{proof}
Since the map $(\phi^m)'$ is the derived $(p,I)$-completion of the base change of the $\delta$-$A$-algebra map \[
(\phi^m)'' : A\{x_1, \dots, x_n\}^\wedge \to A\{x_1, \dots, x_n\}^\wedge, \ \ x_i \mapsto \phi^m(x_i),
\]
it suffices to prove that $(\phi^m)''$ is $(p,I)$-completely faithfully flat. By Lemma 2.11 of \cite{BS19}, the $\delta$-$A$-algebra map
\[
(\phi^m)''' : A\{x_1, \dots, x_n\} \to A\{x_1, \dots, x_n\}, \ \ x_i \mapsto \phi^m(x_i)
\]
is faithfully flat. As the map $(\phi^m)''$ is the derived $(p,I)$-completion of the base change of $(\phi^m)'''$ along $A\{x_1, \dots, x_n\} \to A\{x_1, \dots, x_n\}^\wedge$, we conclude that $(\phi^m)''$ is $(p,I)$-completely faithfully flat.
\qed
\end{proof}

We are now prepared to prove the equivalence of topoi. First, we note that the functor of sites $\rho:(X/A)_{m-\prism} \to (X'/A)_{\prism}$ in Construction \ref{conpris} induces a functor between the categories of presheaves of sets
\[
\widehat{\rho}^* : \widehat{(X'/A)}_{\prism} \to \widehat{(X/A)}_{m-\prism} \ \ \ \ \ \ \mathscr{G} \mapsto \mathscr{G} \circ \rho,
\] 
and it admits a right adjoint
\[
\widehat{\rho}_* :  \widehat{(X/A)}_{m-\prism} \to \widehat{(X'/A)}_{\prism}.
\] 
By Proposition \ref{pr:2} and \ref{pr:3}, we obtain a morphism of topoi
\[
\mathrm{C} : \widetilde{(X/A)}_{m-\prism} \to \widetilde{(X'/A)}_{\prism} \ \ \ \ \ \ \mathrm{C}^* = \widehat{\rho}^*, \mathrm{C}_* = \widehat{\rho}_*.
\] 

\begin{thm}
\label{thpristop}
The morphism $\mathrm{C} : \widetilde{(X/A)}_{m-\prism} \to \widetilde{(X'/A)}_{\prism}$ is an equivalence of topoi.
\end{thm}
\begin{proof}
By Proposition 4.2.1 of \cite{Oya17} (see also Proposition 9.10 of \cite{Xu19}), we are reduced to check the conditions which were proved in Propositions \ref{pr:1}, \ref{pr:2}, \ref{pr:3}, \ref{pr:4}.
\qed
\end{proof}


We want to prove that C induces an equivalence between the categories of crystals. 
First, we define a suitable category of crystals with some technical conditions so that our argument works.

\begin{defi}
\label{def:mod}
Let $(E,I_E)$ be a bounded prism.
\begin{enumerate}
\item 
Let $\moda{E}{I_E}$ be the category of $E$-modules $M$ such that, for any map of bounded prisms 
$
(E,I_E) \to (E_0,I_{E_0})
$
 and  any faithfully flat map  of bounded prisms
$
(E_0,I_{E_0}) \to (E'_0,I_{E'_0}),
$  the sequence
\[ 0 \to M \widehat{\otimes}_E E_0 \to M \widehat{\otimes}_E E'_0 \to M \widehat{\otimes}_E (E'_0\widehat{\otimes}_{E_0}E'_0) \]
is exact, where the completion is the classical $(p,I_{E_0})$-completion.
\item
Let $\{ \modb{E}{I_E} \subseteq \moda{E}{I_E} \}_{(E,I_E)}$ be the largest family of full subcategories such that, for any $M \in \modb{E}{I_E}$, any map $(E,I_E) \to (E'_0,I_{E'_0})$ of bounded prisms   and any faithfully flat map  of bounded prisms $(E_0,I_{E_0}) \to (E'_0,I_{E'_0})$, any descent datum $\epsilon$ on    $M \widehat{\otimes}_E E'_0$ (i.e., an isomorphism 
\[ (E'_0\widehat{\otimes}_{E_0}E'_0) \widehat{\otimes}_{E'_0}(M\widehat{\otimes}_E E'_0) \cong (M\widehat{\otimes}_E E'_0) \widehat{\otimes}_{E'_0}(E'_0\widehat{\otimes}_{E_0} E'_0) \]
satisfying the cocycle condition on $E'_0\widehat{\otimes}_{E_0}E'_0\widehat{\otimes}_{E_0}E'_0$)  descends uniquely to an $E_0$-module $M_0 \in \modb{E_0}{I_{E_0}}$.
\end{enumerate}
\end{defi} 

We will need the following variant of Definition \ref{def:mod}; in this variant, the ring $E$ does not necessarily admit a $\delta$-structure, and we assume that the ideal $I_E$ above is equal to $(p)$.

\begin{defi}
Let $E$ be a $p$-torsion free $p$-complete ring.
\begin{enumerate}
\item 
Let $\modc{E}$ be the category of $E$-modules $M$ such that, for any map $E \to E_0$ of $p$-torsion free $p$-complete rings and any $p$-completely faithfully flat map $E_0 \to E'_0$ of $p$-torsion free $p$-complete rings, the sequence
\[ 0 \to M \widehat{\otimes}_E E_0 \to M \widehat{\otimes}_E E'_0 \to M \widehat{\otimes}_E (E'_0\widehat{\otimes}_{E_0}E'_0) \]
is exact, where the completion is the classical $p$-completion.
\item
Let $\{ \modd{E} \subseteq \modc{E} \}_E$ be the largest family of full subcategories such that, for any $M \in \modd{E}$, any map $E \to E'_0$ of $p$-torsion free $p$-complete rings and any $p$-completely faithfully flat map $E_0 \to E'_0$ of $p$-torsion free $p$-complete rings, any descent datum $\epsilon$ on $M \widehat{\otimes}_E E'_0$ (i.e., an isomorphism
\[ (E'_0\widehat{\otimes}_{E_0}E'_0) \widehat{\otimes}_{E'_0}(M\widehat{\otimes}_E E'_0) \cong (M\widehat{\otimes}_E E'_0) \widehat{\otimes}_{E'_0}(E'_0\widehat{\otimes}_{E_0} E'_0) \]
satisfying the cocycle condition on $E'_0\widehat{\otimes}_{E_0}E'_0\widehat{\otimes}_{E_0}E'_0$)  descends uniquely to an $E_0$-module $M_0 \in \modd{E_0}$.
\end{enumerate}
\end{defi}

The categories introduced above have the following properties.
\begin{prop}
\label{pr:tensor}
\begin{enumerate}
\item For any $M \in \moda{E}{I_E}$ and any map $(E,I_E) \to (E',I_{E'})$ of bounded prisms, $M \widehat{\otimes}_E E'$ belongs to $\moda{E'}{I_{E'}}$. The analogous property holds true for the other categories $\modb{E}{I_E}$, $\modc{E}$ and $\modd{E}$.
\item Let $f$ : $M \to M'$ be a morphism in $\moda{E}{I_E}$ and let $(E,I_E) \to (E',I_{E'})$ be a faithfully flat map of bounded prisms. Assume that $f$ induces an isomorphism \[
M \widehat{\otimes}_E E' \simeq M' \widehat{\otimes}_E E'.
\] Then $f$ is also an isomorphism. The analogous property holds ture for $\modc{E}$.
\end{enumerate}
\end{prop}
\begin{proof}
Part (1) follows immediately from the definition of $\moda{E}{I_E}$. For part (2), we consider the map of exact sequences:
\[
\begin{tikzcd}
0 \arrow[r] & M\arrow[r]\arrow[d,"f"] & M \widehat{\otimes}_E E'\arrow[d,"f \widehat{\otimes} \text{id}"]\arrow[r,shift left]\arrow[r,shift right] & M \widehat{\otimes}_E (E' \widehat{\otimes}_E E')\arrow[d,"f \widehat{\otimes} (\text{id}\widehat{\otimes}\text{id})"]\\
0 \arrow[r] &M'\arrow[r] & M' \widehat{\otimes}_E E'\arrow[r,shift left]\arrow[r,shift right] & M' \widehat{\otimes}_E (E' \widehat{\otimes}_E E').
\end{tikzcd}
\]
The map $f \widehat{\otimes} \text{id}$ is an isomorphism by assumption. It follows that $f \widehat{\otimes} (\text{id}\widehat{\otimes}\text{id})$ is also an isomorphism. These imply that $f$ is an isomorphism.
\qed
\end{proof}

\begin{rem}
\begin{enumerate}
\item Let $\modfp{E}$ be  the category of finite projective $E$-modules. Also, let $\modpntors{E}$ be the category of
\[
\begin{cases}
(p,I_E)^n\text{-torsion $E$-modules}  & (\text{when }(E,I_E)\text{ is a bounded prism})  \\
p^n\text{-torsion $E$-modules}  & (\text{when }E\text{ is a $p$-torsion free $p$-complete ring}) 
\end{cases}
\]
and let $\modtors{E}$ be $\bigcup_{n} \modpntors{E}$.
Then we have an inclusion
\[
\modfp{E} \subseteq \modb{E}{I_E}, \ \modd{E}
\]
 by Proposition A.12 of \cite{AB19}. We also have an inclusion
\[
\modtors{E} \subseteq \modb{E}{I_E}, \  \modd{E}
\]
by the usual descent argument.
\item  If $E$ is a $p$-torsion free $p$-complete $\delta$-ring, then we have an inclusion 
\[ 
\modc{E} \subseteq \moda{E}{pE}
\]
by the definition. The same holds true for $\mathcal{M}(E)$ and $\mathcal{M}_{\prism}(E,pE)$. 
\end{enumerate}
\end{rem}

\noindent
We shall simply write $\modaa{E}$ (resp. $\modbb{E}$)  for $\moda{E}{I_E}$ (resp. $\modb{E}{I_E}$).

Next, we define the category of crystals with respect to the categories of modules $\modbb{E}$ in Definition \ref{def:mod}.

\begin{defi}
\label{mcrys}
Let $\mathscr{C}_{\prism}((X/A)_{m-\prism})$ (resp. $\mathscr{C}^{\mathrm{fp}}((X/A)_{m-\prism}),
\mathscr{C}^{\mathrm{tors}}((X/A)_{m-\prism})$) be the category of abelian presheaves $\mathscr{F}$ on $(X/A)_{m-\prism}$ such that, for any object \linebreak $(E,I_E)$  in $(X/A)_{m-\prism}$, $\mathscr{F}(E,I_E) \in \modbb{E}$ (resp. $\modfp{E},\modtors{E}$), and for any \linebreak morphism $(E,I_E) \to (E_1,I_{E_1})$ in $(X/A)_{m-\prism}$, the map $\mathscr{F}(E,I_E)  \to \mathscr{F}(E_1,I_{E_1})$ is compatible with the module structures in the usual sense and
the canonical map $\mathscr{F}(E,I_E) \widehat{\otimes}_E E_1 \to \mathscr{F}(E_1,I_{E_1})$ is an isomorphism of $E_1$-modules. 
This condition means that $\mathscr{C}_{\prism}((X/A)_{m-\prism})$ (resp. $\mathscr{C}^{\mathrm{fp}}((X/A)_{m-\prism}), \mathscr{C}^{\mathrm{tors}}((X/A)_{m-\prism})$) is the category of crystals with respect to $\mathcal{M}_{\prism}$ (resp. $\mathcal{M}^{\rm fp},\mathcal{M}^{\rm tors}$).
\end{defi}

\begin{rem}
\label{remmcrys}
\begin{enumerate}
\item Presheaves $\mathscr{F}$ in the above definition are automatically sheaves by Definition \ref{def:mod}.1.
\item One can prove that the category $\mathscr{C}_{\prism}((X/A)_{m-\prism})$ is unchanged even if we do not impose the condition (*) in Definition \ref{def:m-pris}.
\end{enumerate}
\end{rem}

In order to prove that the morphism C induces an equivalence between the \linebreak categories of crystals, we will use the following propositions. We follow the proof of Theorem 9.12 of \cite{Xu19}.

\begin{prop}
\label{pr:5}
{\rm (cf. \cite{Xu19} 9.5)} Let $(E,I_E)$ be an object of $(X/A)_{m-\prism}$ and let \linebreak $g : \rho(E,I_E) \to (E',I_{E'})$ be a morphism in $(X'/A)_{\prism}$. Then there exist an object $(E_1,I_{E_1})$ of $(X/A)_{m-\prism}$ and a morphism $f : (E,I_E) \to (E_1,I_{E_1})$ in $(X/A)_{m-\prism}$ such that  $g = \rho(f)$ $($So $(E_1,I_{E_1}) = (E',I_{E'})$ as a bounded prism$)$. 

If $g$ is a cover, then so is $f$.
\end{prop}
\begin{proof}
We have shown this result in the proof of Proposition \ref{pr:3}.
\qed
\end{proof}

\begin{prop}
\label{pr:6}
{\rm (cf. \cite{Xu19} 9.8(ii))} Let $g : (E',I_{E'}) \to (E'_1,I_{E'_1})$ be a morphism in $(X'/A)_{\prism}$. Then there exist a morphism $h : (E,I_{E}) \to (E_1,I_{E_1})$ in $(X/A)_{m-\prism}$ and covers $f : (E',I_{E'}) \to \rho(E,I_{E})$, $f_1 : (E'_1,I_{E'_1}) \to \rho(E_1,I_{E_1})$ in $(X'/A)_{\prism}$ such that the following diagram is a pushout diagram:
\[
\begin{tikzcd}
(E'_1,I_{E'_1}) \arrow[r,"f_1"] & \rho(E_1,I_{E_1}) \arrow[dl,phantom, "\urcorner", very near start] \\
(E',I_{E'}) \arrow[u,"g"]\arrow[r,"f"]&\rho(E,I_{E})\arrow[u,"\rho(h)"].
\end{tikzcd}
\]
\end{prop}
\begin{proof}
Let $f$ be the cover constructed in Proposition \ref{pr:4}, and let $(E'',I_{E''})$ be the pushout of the diagram $(E'_1,I_{E'_1})\xleftarrow{g}(E',I_{E'})\xrightarrow{f}\rho(E,I_{E})$. Applying Proposition \ref{pr:5} to the map $\rho(E,I_{E}) \to (E'',I_{E''})$, we obtain the desired diagram.
\qed
\end{proof}

\begin{prop}
\label{pr:7}
{\rm (cf. \cite{Xu19} 9.9)}  Let $(E',I_{E'})$ be an object of $(X'/A)_{\prism}$, let $(E,I_E)$ be an object of $(X/A)_{m-\prism}$ and let $(E',I_{E'}) \to \rho(E,I_{E})$ be a cover. 
Then there exist an object  $(E_2,I_{E_2})$ of $(X/A)_{m-\prism}$ and two morphisms $p_1,p_2 : (E,I_{E}) \to (E_2,I_{E_2})$ in $(X/A)_{m-\prism}$ such that $\rho(E_2,I_{E_2}) = \rho(E,I_E) \underset{(E',I_{E'})}{\widehat{\otimes}} \rho(E,I_E)$, and $\rho(p_1)$ {\rm (}resp. $\rho(p_2)${\rm )} is the map $\rho(E,I_E) \to \rho(E,I_E) \underset{(E',I_{E'})}{\widehat{\otimes}} \rho(E,I_E) $ to the first {\rm (}resp. second{\rm )} component.
\end{prop}
\begin{proof}
Applying Proposition \ref{pr:5} to the coprojection 
\[
\rho(E,I_E) \to \rho(E,I_E) \underset{(E',I_{E'})}{\widehat{\otimes}} \rho(E,I_E) 
\]
 into the first component, we obtain the map $p_1 : (E,I_{E}) \to (E_2,I_{E_2})$ satisfying the conditions. The existence of $p_2$ follows from the fullness of $\rho$.
\qed
\end{proof}

We are now prepared to prove the equivalence of categories of crystals.

\begin{thm}
\label{th:mcrys}
The functors $\mathrm{C}_*,\mathrm{C}^*$ induce equivalences of categories
\[
\mathscr{C}_{\prism}((X/A)_{m-\prism}) \rightleftarrows \mathscr{C}_{\prism}((X'/A)_{\prism})
\]
quasi-inverse to each other.
\end{thm}
\begin{proof}
We want to prove that the functors $\mathrm{C}_*,\mathrm{C}^*$ preserve crystals. For $\mathrm{C}^*$, it follows easily from the equality
\[
\mathrm{C}^*(\mathscr{F})(E,I_E) = \mathscr{F}(\rho(E,I_E)).
\]

We show the claim for  $\mathrm{C}_*$. Let $\mathscr{F}$ be an object in $\mathscr{C}_{\prism}((X/A)_{m-\prism})$ and let \linebreak $g : (E',I_{E'}) \to (E'_1,I_{E'_1})$ be a morphism in $(X'/A)_{\prism}$. We want to prove that \linebreak $\mathrm{C}_*(\mathscr{F})(E',I_{E'}) \in \modbb{E'}, \mathrm{C}_*(\mathscr{F})(E'_1,I_{E'_1}) \in \modbb{E'_1}$ and that the map
\[
\mathrm{C}_*(\mathscr{F})(E',I_{E'}) \underset{E'}{\widehat{\otimes}}  E'_1 \to \mathrm{C}_*(\mathscr{F})(E'_1,I_{E'_1})
\]
is an isomorphism.  By Proposition \ref{pr:6}, we have a pushout diagram
\[
\begin{tikzcd}
(E'_1,I_{E'_1}) \arrow[r,"f_1"] & \rho(E_1,I_{E_1}) \arrow[dl,phantom, "\urcorner", very near start] \\
(E',I_{E'}) \arrow[u,"g"]\arrow[r,"f"]&\rho(E,I_{E})\arrow[u,"\rho(h)"]
\end{tikzcd}
\]
where $f$ and $f_1$ are covers. Then we have the following commutative diagram :
\[
\begin{tikzcd}[column sep=0.4cm]
\mathrm{C}_*(\mathscr{F})(E',I_{E'}) \underset{E'}{\widehat{\otimes}} E'_1 \underset{E'_1}{\widehat{\otimes}}  E_1 \arrow[r,"\simeq"]\arrow[d] & \mathrm{C}_*(\mathscr{F})(E',I_{E'}) \underset{E'}{\widehat{\otimes}} E \underset{E}{\widehat{\otimes}}  E_1 \arrow[r]& \mathrm{C}_*(\mathscr{F})(\rho(E,I_{E})) \underset{E}{\widehat{\otimes}}  E_1 \arrow[d] \\
\mathrm{C}_*(\mathscr{F})(E'_1,I_{E'_1}) \underset{E'_1}{\widehat{\otimes}} E_1 \arrow[rr]&&\mathrm{C}_*(\mathscr{F})(\rho(E_1,I_{E_1})).
\end{tikzcd}
\]
Using Proposition \ref{pr:tensor}, it is enough to show that $\mathrm{C}_*(\mathscr{F})(E',I_{E'}) \in \modbb{E'},$ \linebreak $\mathrm{C}_*(\mathscr{F})(E'_1,I_{E'_1}) \in \modbb{E'_1}$ and that the left vertical arrow is an isomorphism. 
To show that the left vertical arrow is an isomorphism, it suffices to check that the other arrows are all isomorphisms.
As $\mathrm{C}^* \mathrm{C}_* \simeq$ id, we see that for each $(E,I_E) \in (X/A)_{m-\prism}$,
\begin{equation}
\label{eq:c*f}
\mathscr{F}(E,I_E) = \mathrm{C}^* \mathrm{C}_*(\mathscr{F})(E,I_E) = \mathrm{C}_*(\mathscr{F})(\rho(E,I_E)).
\end{equation}

As $\mathscr{F}$ is a crystal, the equality (\ref{eq:c*f}) implies that the right vertical arrow is an isomorphism.
So it is enough to show that for any cover $f : (E',I_{E'}) \to \rho(E,I_{E})$, one has $\mathrm{C}_*(\mathscr{F})(E',I_{E'}) \in \modbb{E'}$ and the map
\[
\mathrm{C}_*(\mathscr{F})(E',I_{E'}) \underset{E'}{\widehat{\otimes}} E \to  \mathrm{C}_*(\mathscr{F})(\rho(E,I_{E}))
\]
is an isomorphism.

In the following, we simply write $E,E'$ for $(E,I_E),(E',I_{E'})$ respectively in order to lighten notation.
 We consider the diagram
\[
\begin{tikzcd}
E' \arrow[r,"f"] & \rho(E) \arrow[r,shift left]\arrow[r,shift right] &  \rho(E) \underset{E'}{\widehat{\otimes}} \rho(E) \arrow[r,shift left =1.2ex]\arrow[r,shift right =1.2ex]\arrow[r] & \rho(E) \underset{E'}{\widehat{\otimes}} \rho(E) \underset{E'}{\widehat{\otimes}} \rho(E),
\end{tikzcd}
\]
where all the arrows except $f$ are the maps constructed by the coprojections. By Proposition \ref{pr:7} and the definition of $\rho$, the above diagram may be rewritten as
\[
\begin{tikzcd}
E' \arrow[r,"f"] & \rho(E) \arrow[r,shift left]\arrow[r,shift right] &  \rho(E_2) \arrow[r,shift left =1.2ex]\arrow[r,shift right =1.2ex]\arrow[r] & \rho(E_3),
\end{tikzcd}
\]
where $E_2,E_3 \in (X/A)_{m-\prism}$ such that  $E_2= E \underset{E'}{\widehat{\otimes}} E, E_3 = E \underset{E'}{\widehat{\otimes}} E \underset{E'}{\widehat{\otimes}} E$ as prisms and all the arrows except $f$ come from the morphisms in $(X/A)_{m-\prism}$. Then we have a diagram
\[
\begin{tikzcd}
\mathrm{C}_*(\mathscr{F})(\rho(E)) \arrow[r,shift left]\arrow[r,shift right] &  \mathrm{C}_*(\mathscr{F})(\rho(E_2)) \arrow[r,shift left =1.2ex]\arrow[r,shift right =1.2ex]\arrow[r] & \mathrm{C}_*(\mathscr{F})(\rho(E_3))
\end{tikzcd}
\]
which is identified with the following diagram by using the equality (\ref{eq:c*f}):
\[
\begin{tikzcd}
\mathscr{F}(E) \arrow[r,shift left]\arrow[r,shift right] &  \mathscr{F}(E_2) \arrow[r,shift left =1.2ex]\arrow[r,shift right =1.2ex]\arrow[r] & \mathscr{F}(E_3).
\end{tikzcd}
\]
As $\mathscr{F}$ is a crystal, the above diagram defines a descent datum on $\mathscr{F}(E)$ relative to the faithfully flat map of bounded prisms $(E',I_{E'}) \to (E,I_E)$. So it descends uniquely to an object $M$ in $\modbb{E'}$. In particular, we have an isomorphism of $E$-modules $M \widehat{\otimes}_{E'} E \cong \mathscr{F}(E)$ and an exact sequence
\[
\begin{tikzcd}
0 \arrow[r] & M \arrow[r] & \mathscr{F}(E) \arrow[r,shift left]\arrow[r,shift right] &  \mathscr{F}(E_2).
\end{tikzcd}
\]
On the other hand, since $\mathrm{C}_*(\mathscr{F})$ is a sheaf, we have an exact sequence
\[
\begin{tikzcd}
0 \arrow[r] & \mathrm{C}_*(\mathscr{F})(E') \arrow[r] & \mathrm{C}_*(\mathscr{F})(\rho(E)) \arrow[r,shift left]\arrow[r,shift right] &  \mathrm{C}_*(\mathscr{F})(\rho(E_2)).
\end{tikzcd}
\]
By these two exact sequences and the equality (\ref{eq:c*f}), we see that 
\[
\mathrm{C}_*(\mathscr{F})(E') \cong M \in \modbb{E'}.
\]
 Thus,
\[
\mathrm{C}_*(\mathscr{F})(E',I_{E'}) \underset{E'}{\widehat{\otimes}} E \to  \mathrm{C}_*(\mathscr{F})(\rho(E,I_{E}))
\]
is an isomorphism, as desired.
\qed
\end{proof}

The following is proved in exactly the same way as Theorem \ref{th:mcrys}.

\begin{cor}
The functors $\mathrm{C}_*,\mathrm{C}^*$ induce equivalences of categories
\begin{align*}
\mathscr{C}^{\mathrm{fp}}((X/A)_{m-\prism}) &\rightleftarrows \mathscr{C}^{\mathrm{fp}}((X'/A)_{\prism}), \\
\mathscr{C}^{\mathrm{tors}}((X/A)_{m-\prism}) &\rightleftarrows \mathscr{C}^{\mathrm{tors}}((X'/A)_{\prism})
\end{align*}
quasi-inverse to each other.
\end{cor}

\section{The \textit{m}-\textit{q}-crystalline site}

In this section, we prove a $q$-crystalline version of the theorems in the previous section, 
namely, we define the $m$-$q$-crystalline site which is a higher level analog of the $q$-crystalline site defined in Section 16 of \cite{BS19}, and prove the equivalence between the category of crystals on the $m$-$q$-crystalline site of a smooth formal scheme $X$ and that on the usual $q$-crystalline site of $X'$, where $X'$ is the pullback of $X$  by the $m$-fold iteration $\phi^m$ of the Frobenius lift $\phi$ on the base $q$-PD pair.

First, we give the definition of  $q$-PD pairs we  consider. Set $A = \mathbf{Z}_p \llbracket q-1\rrbracket$ with $\delta$-structure given by $\delta(q) = 0$.

\begin{defi}
\label{def:qpd}
A $q$-PD pair is given by a derived $(p,[p]_q)$-complete $\delta$-pair $(D,I)$ over $(A,(q-1))$ satisfying the following conditions:
\begin{enumerate}
\item For any $f \in I$, $\phi (f) - [p]_q \delta(f) \in [p]_qI$.
\item The pair $(D,([p]_q))$ is a bounded prism over $(A,([p]_q))$, i.e., $D$ is $[p]_q$-torsion free and $D/([p]_q)$ has bounded $p^\infty$-torsion.
\item The ring $D/(q-1)$ is $p$-torsion free with finite $(p,[p]_q)$-complete Tor-amplitude over $D$.
\item $D/I$ is classically $p$-complete.
\end{enumerate}
\end{defi}

\begin{rem}
\label{rqpd}
\begin{enumerate}
\item This definition of $q$-PD pairs follows that in Section 7.1 of \cite{Kos20}. 
The condition 4. enables us to consider the affine $p$-adic formal scheme Spf$(D/I)$, but this is not imposed in Definition 16.2 of \cite{BS19}. Our notion of $q$-PD pair differs from that in Definition 3.1 of \cite{GSQ20a}; 
the conditions 3., 4. in Definition \ref{def:qpd}, the condition that $D/([p]_q)$ has bounded $p^\infty$-torsion and the condition of $(p,[p]_q)$-completeness of $D$ are not imposed in Definition 3.1 of \cite{GSQ20a} (but the latter two conditions are imposed in the definition of the $q$-crystalline site in Section 7 of \cite{GSQ20a}).
\item By the conditions 1. and 2., we see that for any $x \in I$, the element 
\[
\gamma(x) \stackrel{\mathrm{def}}{=} \phi(x)/[p]_q - \delta(x) \in I
\]
 is well-defined.
\item Let $J = (\phi^m)^{-1}(I)$. As $\phi$ is continuous, the ideal $J \subseteq D$ is closed for the $p$-adic topology of $D$. So $D/J$ is also classically $p$-complete.
\item The condition 4. in Definition \ref{def:qpd} implies that $I$ is closed for the $(p,[p]_q)$-adic topology of $D$ because \[
I + (p,[p]_q)^n \subseteq I + (p, q-1)^n \subseteq I+(p)^n.
\]
\end{enumerate}
\end{rem}

Next, we define the $m$-$q$-crystalline site, which is a higher level analog of the $q$-crystalline site as well as a $q$-analog of the level $m$ crystalline site. We fix a non-negative integer $m$ and a $q$-PD pair $(D,I)$.

\begin{defi}
\label{def:mqcrys}
Let $J = (\phi ^m)^{-1}(I)$ and let $X$ be a $p$-adic formal scheme smooth and separated over $D/J$. We define the  $m$-$q$-\textit{crystalline site} $(X/D)_{m\text{-}q\text{-crys}}$ of $X$ over $D$ as follows.  
Objects are maps $(D,I) \to (E,I_E)$ of $q$-PD pairs together with a map Spf$(E/J_E) \to X$ over $D/J$ satisfying the following condition, where $J_E = (\phi ^m)^{-1}(I_E)$:
\begin{enumerate}[label = (\roman*)]
\item[(*)] Spf$(E/J_E) \to X$ factors through some affine open Spf$(R) \subseteq X$.
\end{enumerate}
We shall often denote such an object by
\[
(\text{Spf}(E) \gets \text{Spf}(E/J_E) \to X) \in (X/D)_{m\text{-}q\text{-crys}}
\]
or $(E,I_E)$ if no confusion arises. A morphism \[
(\text{Spf}(E') \gets \text{Spf}(E'/J_{E'}) \to X) \to (\text{Spf}(E) \gets \text{Spf}(E/J_E) \to X)
\]
 is a map of $q$-PD pairs $(E,I_E) \to (E',I_{E'})$ over $(D,I)$ such that the induced morphism
\[
\text{Spf}(E'/J_{E'}) \to \text{Spf}(E/J_E)
\]
 is compatible with the maps $\text{Spf}(E'/J_{E'}) \to X$, $\text{Spf}(E/J_E) \to X$. 
When we denote such an object by $(E,I_E)$, we shall write $(E,I_E) \to (E',I_{E'})$ (not $(E',I_{E'}) \to (E,I_E))$ for a morphism from $(E',I_{E'})$ to $(E,I_E)$. A map ($E,I_E$) $\to$ ($E',I_{E'}$) in  $(X/D)_{m\text{-}q\text{-crys}}$ is a cover if it is a $(p,[p]_q)$-completely faithfully flat map and satisfies
\begin{equation}
\label{eq:rigid}
I_{E'} = \ol{I_E E'},
\end{equation}
where the right hand side means the $(p,[p]_q)$-complete ideal of $E'$ generated by $I_EE'$.
\end{defi}

\begin{rem}
Note that the equality (\ref{eq:rigid}) does not imply the equality $J_{E'} = \ol{J_E E'}$ in general. Indeed, in the case of the inclusion
\[
(E,I_E) := (\mathbf{Z}_p \llbracket q-1\rrbracket, (q-1)) \hookrightarrow (\mathbf{Z}_p \llbracket q^{1/p}-1\rrbracket, (q-1)) =: (E', I_{E'})
\]
where the $\delta$-ring structure on $E'$ is defined by $\delta(q^{1/p}) = 0$, we have $I_{E'} = \ol{I_E E'}$ but $\ol{J_E E'} = (q-1) \neq (q^{1/p}-1) = J_{E'}$.
\end{rem}

We need to check that the category $(X/D)_{m\text{-}q\text{-crys}}$ endowed with the topology as defined above forms a site. We have the following lemma as in the case of the $m$-prismatic site: 

\begin{lem}
\label{lm:pushout2}
Let $(E_1,I_{E_1})\xleftarrow{f}(E,I_E)\xrightarrow{g}(E_2,I_{E_2})$ be maps in $(X/D)_{m\text{-}q\text{-}\mathrm{crys}}$ such that f is a cover. 
Let $E_3 := E_1 \widehat{\otimes}_E E_2$, where the completion is the classical $(p,[p]_q)$-completion. Then $(E_3,\ol{I_{E_2}E_3})$ is the object that represents the  coproduct 
\[
(E_1,I_{E_1})\sqcup_{(E,I_E)}(E_2,I_{E_2})
\]
 in $(X/D)_{m\text{-}q\text{-}\mathrm{crys}}$, and the canonical map $(E_2,I_{E_2}) \to (E_3,I_{E_3})$ is a cover.
\end{lem}

\begin{proof}
By Lemma \ref{lm:pushout}, $(E_3, ([p]_q))$ represents the object $(E_1,([p]_q))\sqcup_{(E,([p]_q))}(E_2,([p]_q))$ in the category of bounded prisms. Also, by Lemma 16.5 (5) of \cite{BS19}, $(E_3,\ol{I_{E_2}E_3})$ is a $q$-PD pair. We can check that it represents the object 
\[
(E_1,I_{E_1})\sqcup_{(E,I_E)}(E_2,I_{E_2})
\]
 in $(X/D)_{m\text{-}q\text{-}\mathrm{crys}}$ by using Remark \ref{rqpd}.4.
\qed
\end{proof}

By Lemma \ref{lm:pushout2}, the set of covers in Definition \ref{def:mqcrys} actually forms a pretopology on the category  $(X/D)_{m\text{-}q\text{-crys}}$. 

\begin{rem}
\label{remq}
When $m=0$, we denote the site $(X/D)_{m\text{-}q\text{-}\mathrm{crys}}$ simply by $(X/D)_{q\text{-crys}}$ and call it the $q$-crystalline site. This site is similar to the one in \cite{BS19} and that in \cite{Kos20} but there are slight differences: Firstly, the topology considered here (flat topology) is different from the one in \cite{BS19} (indiscrete topology) and the one in \cite{Kos20} (\'{e}tale topology). Secondly, we imposed the technical condition (*) which are not assumed in \cite{BS19} and \cite{Kos20}. (However, we note that our topos is unchanged if we do not impose the condition (*).) Thirdly, $X$ is assumed to be affine in \cite{BS19} but we do not assume that $X$ is affine.
\end{rem}

In order to establish the equivalence between the category of crystals on the $m$-$q$-crystalline site and that on the usual $q$-crystalline site, we first construct a functor between these sites, as in the prismatic case.

\begin{con}
\label{conqcrys}
Under the notation and assumption in Definition \ref{def:mqcrys}, let $X'$ be $X$$\underset{\mathrm{Spf}(D/J),(\phi ^m)^*}{\widehat{\times}}$ Spf$(D/I)$. Then we have a diagram:
\[
\begin{tikzcd}
\mathrm{Spf}(D)\arrow[r,hookleftarrow] & \mathrm{Spf}(D/J)\arrow[dr,phantom, "\Box"] & X\arrow[l]\\
\mathrm{Spf}(D)\arrow[r,hookleftarrow]\arrow[u,"(\phi ^m)^*"] & \mathrm{Spf}(D/I)\arrow[u,"(\phi ^m)^*"] & X'.\arrow[l]\arrow[u]
\end{tikzcd} 
\]
We define a functor $\rho$ from the $m$-$q$-crystalline site of $X$ over $D$ to the usual $q$-crystalline site of $X'$ over $D$ as in the case of $m$-prismatic site: For an object \[(\text{Spf}(E) \gets \text{Spf}(E/J_E) \to X)\] of $(X/D)_{m\text{-}q\text{-crys}}$, 
 we define the object $\rho$(Spf$(E) \gets$ Spf$(E/J_E) \to X)$ of $(X'/D)_{q\text{-crys}}$ by 
\[
(\text{Spf}(E)  \gets \text{Spf}(E/I_E) \xrightarrow{f} X'),
\]
 where the right map $f$ is defined as follows:
\[ \mathrm{Spf}(E/I_E) \xrightarrow{g} \mathrm{Spf}(E/J_E) \underset{\mathrm{Spf}(D/J),(\phi ^m)^*}{\widehat{\times}} \mathrm{Spf}(D/I) \to X \underset{\mathrm{Spf}(D/J),(\phi ^m)^*}{\widehat{\times}} \mathrm{Spf}(D/I) = X'. \]
Here the first map $g$ is induced by the map of rings
\[
E/J_E \underset{D/J,\phi^m}{\widehat{\otimes}} D/I \to E/I_E ; \ \ \ \ \  e \otimes d \mapsto \phi^m(e)d.
\]
This defines the functor $\rho:(X/D)_{m\text{-}q\text{-crys}} \to (X'/D)_{q\text{-crys}}$.
\end{con}

Next, we want to show that $\rho$ induces an equivalence of topoi. 
Since the $m$-$q$-crystalline site is defined in a very similar way to the $m$-prismatic site, almost all propositions can be proved in exactly the same way as in the case of $m$-prismatic site by replacing the bounded prism $(A,I)$ with the $q$-PD pair $(D,I)$ and so we omit the details. In particular, we have the following proposition.

\begin{prop}
\label{pr:1q}
The functor $\rho$ is fully faithful, continuous and cocontinuous.
\end{prop}

As the construction of the $q$-PD envelope is different from that of the prismatic envelope that appeared in the proof of Proposition \ref{pr:4}, we give a proof of the $q$-analog of Proposition \ref{pr:4}. We will need the following $q$-analog of Lemma \ref{prislift}:

\begin{lem}
\label{qpdlift}
Let $(E,I_E)$ be a $q$-PD pair. Let $\overline{f}$ be an open immersion 
\[
\overline{f} :\mathrm{Spf}(\overline{E_i}) \to \mathrm{Spf}(E/I_E)
\]
 or an open immersion 
\[
\overline{f} :\mathrm{Spf}(\overline{E_i}) \to \mathrm{Spf}(E/J_E).
\]
 Then there is a  unique open immersion  $f : \mathrm{Spf}(E_i) \to \mathrm{Spf}(E)$ which lifts $\overline{f}$. Moreover,  the corresponding map of rings induces a map $(E,I_E) \to (E_i,\ol{I_EE_i})$ of $q$-PD pairs.
\end{lem}
\begin{proof}
We first check that the images of $I_E, J_E$ in $E/(p,[p]_q)$ are nil ideals. For $I_E$, it follows immediately from the condition 1. in Definition \ref{def:qpd} and the definition of the Frobenius lift $\phi$. For $J_E$, we note that the image of $J_E$ in $E/I_E$ is a nil ideal by the definitions of $J_E$ and $\phi$, so the claim follows. Then the existence, the uniqueness and the affinity of the lifting follow from Lemma \ref{lemlift}. We denote the lifting by \linebreak $f : \mathrm{Spf}(E_i) \to \mathrm{Spf}(E)$. We can give a $\delta$-structure on $E_i$ such that the corresponding map $E \to E_i$ is a map of $\delta$-rings as in the proof of Lemma  \ref{prislift} by replacing the ideal $I_E$ with the ideal $([p]_q)$.

It remains to show that $(E_i,\ol{I_EE_i})$ is a $q$-PD pair. The condition 4. in Definition \ref{def:qpd} follows as the ideal $\ol{I_EE_i}$ is closed. To prove the other conditions, by Lemma 16.5 (5) of \cite{BS19}, it is enough to prove that the map of $\delta$-rings $E \to E_i$ is $(p,[p]_q)$-completely flat. Applying the same argument as in the proof of Lemma \ref{prislift} by replacing the ideal $I_E$ with the ideal $([p]_q)$, we see that $E \to E_i$ is $(p,[p]_q)$-completely flat (actually $(p,[p]_q)$-completely \'{e}tale), as desired.
\qed
\end{proof}

The next result will be used to prove the equivalence of topoi.

\begin{prop}
\label{pr:4q}
Let $(E',I_{E'})$ be an object in $(X'/D)_{q\text{-}\mathrm{crys}}$. Then there exists an object $(E,I_{E})$ in $(X/D)_{m\text{-}q\text{-}\mathrm{crys}}$ and a cover of the form $(E',I_{E'})$ $\to$ $\rho$$(E,I_{E})$.
\end{prop}
\begin{proof}
Let $X = \bigcup_i \text{Spf}(R_i)$ be a finite affine open cover of $X$. By the definition of $X'$ in Construction \ref{conqcrys}, we see that $\text{Spf}(R_i) \underset{X}{\widehat{\times}} X' = \text{Spf}(R_i \underset{D/J,\phi^m}{\widehat{\otimes}} D/I)$. We denote this formal scheme by $\text{Spf}(R'_i)$.
On the other hand, the condition (*) in Definition \ref{def:qpd} ensures that the map $\text{Spf}(E'/I_{E'}) \to X'$ factors through some affine open $\text{Spf}(R'') \subseteq X'$. As $X'$ is separated by assumption, we see that  $\text{Spf}(R_i') \cap \text{Spf}(R'') = \text{Spf}(R^*_i)$ for some $R^*_i$. Then the formal scheme $\text{Spf}(R'_i) \widehat{\times}_{X'} \text{Spf}(E'/I_{E'})$ is identified with $\text{Spf}(R_i' \widehat{\otimes}_{R^*_i} E'/I_{E'})$. In particular, it is affine. We denote this affine open subscheme of $\text{Spf}(E'/I_{E'})$ by $\text{Spf}(\overline{E'_i})$. 
By Lemma \ref{qpdlift}, there is a unique $q$-PD pair $(E'_i, \ol{I_{E'}E'_i}) \in (X'/D)_{q\text{-}\mathrm{crys}}$ for which the corresponding formal scheme $\text{Spf}(E'_i)$ is an affine open formal subscheme of $\text{Spf}(E')$ and lifts $\text{Spf}(\overline{E'_i})$, namely $E'_i/\ol{I_{E'}E'_i} = \overline{E'_i}$. By construction, we see that \[
(E',I_{E'}) \to \prod_i(E'_i,\ol{I_{E'}E'_i})
\]
 is a cover in $(X'/D)_{q\text{-}\mathrm{crys}}$.
To prove the proposition, we may replace $(E',I_{E'})$ by $(E'_i,\ol{I_{E'}E'_i})$ to assume that the structure morphism $\text{Spf}(E'/I_{E'}) \to X'$ factors through an affine open $\text{Spf}(R') \subseteq X'$ such that $R'$ is of the form $R \underset{D/J,\phi^m}{\widehat{\otimes}} D/I$ for some affine open $\text{Spf}(R) \subseteq X$.

As the map $D/J \to R$ is smooth, Tag 00TA of \cite{Sta} tells us that there exists an open cover of $\text{Spf}(R)$ by standard opens $\text{Spf}(\widehat{R_g})$ such that each $\widehat{R_g}/p$ is standard smooth over $D/(p,J)$.
After refining the given cover of $X$ in the previous paragraph we may assume that $R/p$ is standard smooth over $D/(p,J)$.
By Tag 00T7 of \cite{Sta},  there exists a surjection $D[x_1, \dots, x_n]^\wedge \to R$ whose kernel is the ideal ($J, y_1, \dots, y_r$) such that 
\[
\overline{y_1}, \dots, \overline{y_r} \in D/(p,J)[x_1, \dots, x_n]
\]
 form a regular sequence. 
On the other hand, there is a natural map 
\[
D[x_1, \dots, x_n]^\wedge \to D\{x_1, \dots, x_n\}^\wedge,
\]
 where the symbol \{\} denotes the adjoining of elements in the theory of $\delta$-rings. By a similar argument to that in the proof of Corollary \ref{cor:reg seq}, we see that the sequence 
\[
\phi^m(y_1), \dots, \phi^m(y_r) \in D\{x_1, \dots, x_n\}^\wedge
\]
 is ($p,[p]_q$)-completely regular relative to $D$. 
Then we can construct a map 
\[
D\{x_1, \dots, x_n\}^\wedge \to S \stackrel{\mathrm{def}}{=} (D\{x_1, \dots, x_n, \frac{\phi^{m+1}(y_1)}{[p]_q}, \dots, \frac{\phi^{m+1}(y_r)}{[p]_q}\})^\wedge
\]
 to the $q$-PD envelope of $(D\{x_1, \dots, x_n\}^\wedge,K)$ by Lemma 16.10 of \cite{BS19}, where $K$ denotes the ideal $(I,\phi^{m}(y_1), \dots, \phi^{m}(y_r)) \subseteq D\{x_1, \dots, x_n\}^\wedge$. 
By the construction of the $q$-PD envelope, $\phi^m(y_i) \in I_S$ and so $y_i \in J_S = (\phi^m)^{-1}(I_S)$ for all $i$. This gives a map $R \to S/J_S$.

Next, let $D[x_1, \dots, x_n]^\wedge \to R \underset{D,\phi^m}{\widehat{\otimes}} D$ be the base change of $D[x_1, \dots, x_n]^\wedge \to R$ \linebreak in the previous paragraph along $\phi^m : D \to D$. Passing to the quotient then \linebreak induces $D[x_1, \dots, x_n]^\wedge \to R'$.
The kernel of this map is the ideal ($I, y'_1, \dots, y'_r$), where the $y_{i}'$'s are the images of the $y_i$'s under the map $D[x_1, \dots, x_n]^\wedge \to D[x_1, \dots, x_n]^\wedge$ send\nolinebreak ing $\sum_{\underline{j}} \alpha_{\underline{j}} x^{\underline{j}}$ to $\sum_{\underline{j}} \phi^m(\alpha_{\underline{j}}) x^{\underline{j}}$ (here we write $\underline{j} := (j_1, \dots, j_n)$ for the multi-index). 
By the definition of $y_i$'s, the sequence $\overline{y_1'}, \dots, \overline{y_r'} \in D/(p,I)[x_1, \dots, x_n]$ is a regular \linebreak sequence. By a similar argument to that in the proof of Corollary \ref{cor:reg seq}, the sequence $y_1', \dots, y_r' \in D\{x_1, \dots, x_n\}^\wedge$ is ($p,[p]_q$)-completely regular relative to $D$. Then we can construct a map 
\[
D\{x_1, \dots, x_n\}^\wedge \to S' \stackrel{\mathrm{def}}{=} (D\{x_1, \dots, x_n, \frac{\phi(y_1')}{[p]_q}, \dots, \frac{\phi(y_r')}{[p]_q}\})^\wedge
\]
 to the $q$-PD envelope of $(D\{x_1, \dots, x_n\}^\wedge,K')$, where $K'$ denotes the ideal 
\[
(I,y'_1, \dots, y'_r) \subseteq D\{x_1, \dots, x_n\}^\wedge.
\]
 Then we have a diagram:
\[
\begin{tikzcd}[column sep=0.19cm]
D/I  \arrow[r] & R'&&\\
D \arrow[u, two heads]\arrow[r] &D[x_1, \dots, x_n]^\wedge \arrow[u, two heads]\arrow[r]& D\{x_1, \dots, x_n\}^\wedge \arrow[r] & S' \stackrel{\mathrm{def}}{=} (D\{x_1, \dots, x_n, \frac{\phi(y_1')}{[p]_q}, \dots, \frac{\phi(y_r')}{[p]_q}\})^\wedge.
\end{tikzcd}
\]
By the definitions of $y_i$ and $y_i'$, we have a map of $\delta$-$D$-algebras $(\phi^m)' : S' \to S$ sending $x_i$ to $\phi^m(x_i)$. By Proposition \ref{prop:mqff} below, we see that $(\phi^m)'$ is $(p,[p]_q)$-completely faithfully flat.

Let ($E',I_{E'}) \in (X'/D)_{q\text{-crys}}$ be as above. We have a map 
\[
f_1 : D[x_1, \dots, x_n]^\wedge \to E'/I_{E'}
\]
 determined as the composition of the map $D[x_1, \dots, x_n]^\wedge \to R'$ in the above diagram with $R' \to E'/I_{E'}$. As $D[x_1, \dots, x_n]^\wedge$ is the completion of a polynomial ring, one can choose a map $f_2 : D[x_1, \dots, x_n]^\wedge \to E'$ lifting $f_1$.
As $E'$ is a $\delta$-$D$-algebra, $f_2$ extends uniquely to a $\delta$-$D$-algebra map $f_3 : D\{x_1, \dots, x_n\}^\wedge \to E'$.
By construction, this extension carries $K'$ into $I_{E'}$. By the universal property of $S'$, $f_3$ extends uniquely to a $\delta$-$D$-algebra map $g : S' \to E'$. 
If we set $h : E' \to E' \widehat{\otimes}_{S'} S$ to be the base change of $(\phi^m)' : S' \to S$ along $g$, then by $(p,[p]_q)$-complete faithful flatness of $(\phi^m)'$, we see that the same holds true for $h$.

It remains to check that the map $h$ defines a morphism
\[
(E',I_{E'}) \to \rho(E' \widehat{\otimes}_{S'} S, I_{E' \widehat{\otimes}_{S'} S})
\]
in $(X'/D)_{q\text{-crys}}$. To see this, it is enough to check that the bottom right square in the following diagram is commutative, in which all the other squares are commutative.
{\fontsize{9pt}{10pt}\selectfont
\[
\begin{tikzcd}[column sep=0.35cm]
&D[x_1, \dots, x_n]^\wedge  \arrow[r]\arrow[d] & S\arrow[r,"i_2"]\arrow[d]&    E' \widehat{\otimes}_{S'} S \arrow[d] \\
D[x_1, \dots, x_n]^\wedge \arrow[r, two heads] \arrow[d,"f_2"] & R' = R  \underset{D/J,\phi^m}{\widehat{\otimes}} D/I \arrow[r]\arrow[d]& S/J_S  \underset{D/J,\phi^m}{\widehat{\otimes}} D/I \arrow[r,"\overline{i_2}"] & E' \widehat{\otimes}_{S'} S/J_{E' \widehat{\otimes}_{S'} S}  \underset{D/J,\phi^m}{\widehat{\otimes}} D/I \arrow[d] \\
E'\arrow[r, two heads]&E'/I_{E'}\arrow[rr,"\overline{h}"]&&E' \widehat{\otimes}_{S'} S/I_{E' \widehat{\otimes}_{S'} S}
\end{tikzcd}
\]}

\noindent
Here $i_2$ is the map to the second component of the  coproduct in Lemma \ref{lm:pushout2}. We can check the commutativity of the bottom right square by tracing the elements $x_i$'s:
\[
\begin{tikzcd}
&x_i \arrow[r, mapsto]\arrow[d, mapsto] & x_i\arrow[r, mapsto]&1 \otimes x_i \arrow[d, mapsto] \\
x_i \arrow[r, mapsto] \arrow[d, mapsto] & \overline{x_i} \otimes 1 \arrow[rr, mapsto]&& \overline{1 \otimes x_i} \otimes 1 \arrow[d, mapsto] \\
f_2(x_i)\arrow[r, mapsto]&\overline{f_2(x_i)}\arrow[rr, mapsto]&&\overline{1 \otimes \phi^m(x_i)} = \overline{f_2(x_i) \otimes 1},
\end{tikzcd}
\]
so the proposition follows.
\qed
\end{proof}

We prove the claim used in the proof above.

\begin{prop}
\label{prop:mqff}
With notation as in the proof of Proposition \ref{pr:4q}, the map 
\[
(\phi^m)' : S' \to S
\]
 is $(p,[p]_q)$-completely faithfully flat.
\end{prop}
\begin{proof}
Since the map $(\phi^m)'$ is the derived $(p,[p]_q)$-completion of the base change of the $\delta$-$D$-algebra map \[
(\phi^m)'' : D\{x_1, \dots, x_n\}^\wedge \to D\{x_1, \dots, x_n\}^\wedge, \ \ x_i \mapsto \phi^m(x_i),
\]
it suffices to prove that $(\phi^m)''$ is $(p,[p]_q)$-completely faithfully flat. By Lemma 2.11 of \cite{BS19}, the $\delta$-$D$-algebra map
\[
(\phi^m)''' : D\{x_1, \dots, x_n\} \to D\{x_1, \dots, x_n\}, \ \ x_i \mapsto \phi^m(x_i)
\]
is faithfully flat. As the map $(\phi^m)''$ is the derived $(p,[p]_q)$-completion of the base change of $(\phi^m)'''$ along $D\{x_1, \dots, x_n\} \to D\{x_1, \dots, x_n\}^\wedge$, we conclude that $(\phi^m)''$ is $(p,[p]_q)$-completely faithfully flat.
\qed
\end{proof}

We are now prepared to prove the equivalence of topoi. First, we note that the functor of sites $\rho:(X/D)_{m\text{-}q\text{-crys}} \to (X'/D)_{q\text{-crys}}$ in Construction \ref{conqcrys} induces a functor between the categories of presheaves of sets
\[
\widehat{\rho}^* : \widehat{(X'/D)}_{q\text{-crys}} \to \widehat{(X/D)}_{m\text{-}q\text{-crys}} \ \ \ \ \ \ \mathscr{G} \mapsto \mathscr{G} \circ \rho,
\] 
and it admits a right adjoint
\[
\widehat{\rho}_* :  \widehat{(X/D)}_{m\text{-}q\text{-crys}} \to \widehat{(X'/D)}_{q\text{-crys}}.
\] 
By Proposition \ref{pr:1q}, we obtain a morphism of topoi
\[
\mathrm{C} : \widetilde{(X/D)}_{m\text{-}q\text{-crys}} \to \widetilde{(X'/D)}_{q\text{-crys}} \ \ \ \ \ \ \mathrm{C}^* = \widehat{\rho}^*, \mathrm{C}_* = \widehat{\rho}_*.
\]

\begin{thm}
The morphism $\mathrm{C} : \widetilde{(X/D)}_{m\text{-}q\text{-}\mathrm{crys}} \to \widetilde{(X'/D)}_{q\text{-}\mathrm{crys}}$ is an equivalence of topoi.
\end{thm}
\begin{proof}
By Proposition 4.2.1 of \cite{Oya17} (see also Proposition 9.10 of \cite{Xu19}), we are reduced to check the conditions which were proved in Propositions \ref{pr:1q} and \ref{pr:4q}.
\qed
\end{proof}

Next, we define the category of crystals with respect to the categories of modules $\modbb{E}$ and $\modd{E}$ in Definition \ref{def:mod}.

\begin{defi}
\begin{enumerate}
\item Let 
\[
\mathscr{C}_{\prism}((X/D)_{m\text{-}q\text{-crys}}) \, (\text{resp.} \mathscr{C}^{\mathrm{fp}}((X/D)_{m\text{-}q\text{-crys}}),
\mathscr{C}^{\mathrm{tors}}((X/D)_{m\text{-}q\text{-crys}}))
\]
 be the category of abelian presheaves $\mathscr{F}$ on $(X/D)_{m\text{-}q\text{-crys}}$ such that, for any object  $(E,I_E)$  in $(X/D)_{m\text{-}q\text{-crys}}$, $\mathscr{F}(E,I_E) \in \modbb{E}$ (resp. $\modfp{E},\modtors{E}$), and for any morphism $(E,I_E) \to (E_1,I_{E_1})$ in $(X/D)_{m\text{-}q\text{-crys}}$, the map 
\[
\mathscr{F}(E,I_E)  \to \mathscr{F}(E_1,I_{E_1})
\]
 is compatible with the module structures in the usual sense and
the  canonical  \linebreak map $\mathscr{F}(E,I_E) \widehat{\otimes}_E E_1 \to \mathscr{F}(E_1,I_{E_1})$ is an isomorphism of $E_1$-modules. 
\item Assume that $(D, I)$ is a $q$-PD pair with $q = 1$ in $D$. Then we define the cate\nolinebreak gory
$
\mathscr{C}((X/D)_{m\text{-}q\text{-crys}})
$
 as the category of abelian presheaves $\mathscr{F}$ on $(X/D)_{m\text{-}q\text{-crys}}$ such that, for any object  $(E,I_E)$  in $(X/D)_{m\text{-}q\text{-crys}}$, $\mathscr{F}(E,I_E) \in \modd{E}$, and for any morphism $(E,I_E) \to (E_1,I_{E_1})$ in $(X/D)_{m\text{-}q\text{-crys}}$, the map 
\[
\mathscr{F}(E,I_E)  \to \mathscr{F}(E_1,I_{E_1})
\]
 is compatible with the module structures in the usual sense and
the canonical \linebreak map $\mathscr{F}(E,I_E) \widehat{\otimes}_E E_1 \to \mathscr{F}(E_1,I_{E_1})$ is an isomorphism of $E_1$-modules. 
\end{enumerate}
\end{defi}

\begin{rem}
As in Remark \ref{remmcrys}, one can prove that presheaves $\mathscr{F}$ in the above definition are automatically sheaves and the category $\mathscr{C}_{\prism}((X/D)_{m\text{-}q\text{-crys}})$ is unchanged even if we do not impose the condition (*) in Definition \ref{def:mqcrys}.
\end{rem}

We are now prepared to prove the equivalence of categories of crystals. The fol\nolinebreak low\nolinebreak ing results can be proved in exactly the same way as in the case of $m$-prismatic site, so we omit the details.

\begin{prop}
\label{pr:5q}
Let $(E,I_E)$ be an object of $(X/D)_{m\text{-}q\text{-}\mathrm{crys}}$ and let \[g : \rho(E,I_E) \to (E',I_{E'})\] be a morphism in $(X'/D)_{q\text{-}\mathrm{crys}}$. Then there exist an object $(E_1,I_{E_1})$ of $(X/D)_{m\text{-}q\text{-}\mathrm{crys}}$ and a morphism $f : (E,I_E) \to (E_1,I_{E_1})$ in $(X/D)_{m\text{-}q\text{-}\mathrm{crys}}$ such that  $g = \rho(f)$. If $g$ is a cover, then so is $f$.
\end{prop}

\begin{prop}
\label{pr:6q}
Let $g : (E',I_{E'}) \to (E'_1,I_{E'_1})$ be a morphism in $(X'/D)_{q\text{-}\mathrm{crys}}$. Then there exist a morphism $h : (E,I_{E}) \to (E_1,I_{E_1})$ in $(X/D)_{m\text{-}q\text{-}\mathrm{crys}}$ and covers $f : (E',I_{E'}) \to \rho(E,I_{E})$, $f_1 : (E'_1,I_{E'_1}) \to \rho(E_1,I_{E_1})$ in $(X'/D)_{q\text{-}\mathrm{crys}}$ such that the following diagram is a pushout diagram:
\[
\begin{tikzcd}
(E'_1,I_{E'_1}) \arrow[r,"f_1"] & \rho(E_1,I_{E_1}) \arrow[dl,phantom, "\urcorner", very near start] \\
(E',I_{E'}) \arrow[u,"g"]\arrow[r,"f"]&\rho(E,I_{E})\arrow[u,"\rho(h)"].
\end{tikzcd}
\]
\end{prop}

\begin{prop}
\label{pr:7q}
Let $(E',I_{E'})$ be an object of $(X'/D)_{q\text{-}\mathrm{crys}}$, let $(E,I_E)$ be an \linebreak object of $(X/D)_{m\text{-}q\text{-}\mathrm{crys}}$ and let $(E',I_{E'}) \to \rho(E,I_{E})$ be a cover. 
Then there exist an object  $(E_2,I_{E_2})$ of $(X/D)_{m\text{-}q\text{-}\mathrm{crys}}$ and two morphisms $p_1,p_2 : (E,I_{E}) \to (E_2,I_{E_2})$ in $(X/D)_{m\text{-}q\text{-}\mathrm{crys}}$ such that $\rho(E_2,I_{E_2}) = \rho(E,I_E) \underset{(E',I_{E'})}{\widehat{\otimes}} \rho(E,I_E)$, and $\rho(p_1)$ {\rm (}resp. $\rho(p_2)${\rm )} is the map $\rho(E,I_E) \to \rho(E,I_E) \underset{(E',I_{E'})}{\widehat{\otimes}} \rho(E,I_E) $ to the first {\rm (}resp. second{\rm )} component.
\end{prop}

\begin{thm}
\label{th:mqcrys}
The functors $\mathrm{C}_*,\mathrm{C}^*$ induce  equivalences of categories
\[
\mathscr{C}_{\prism}((X/D)_{m\text{-}q\text{-}\mathrm{crys}}) \rightleftarrows \mathscr{C}_{\prism}((X'/D)_{q\text{-}\mathrm{crys}})
\]
quasi-inverse to each other.
\end{thm}

\begin{cor}
The functors $\mathrm{C}_*,\mathrm{C}^*$ induce  equivalences of categories
\begin{align*}
\mathscr{C}^{\mathrm{fp}}((X/D)_{m\text{-}q\text{-}\mathrm{crys}}) &\rightleftarrows \mathscr{C}^{\mathrm{fp}}((X'/D)_{q\text{-}\mathrm{crys}}), \\
\mathscr{C}^{\mathrm{tors}}((X/D)_{m\text{-}q\text{-}\mathrm{crys}}) &\rightleftarrows \mathscr{C}^{\mathrm{tors}}((X'/D)_{q\text{-}\mathrm{crys}})
\end{align*}
quasi-inverse to each other.
\end{cor}

\begin{cor}
Assume that $(D, I)$ is a $q$-PD pair with $q = 1$ in $D$. Then the functors $\mathrm{C}_*,\mathrm{C}^*$ induce  equivalences of categories
\[
\mathscr{C}((X/D)_{m\text{-}q\text{-}\mathrm{crys}}) \rightleftarrows \mathscr{C}((X'/D)_{q\text{-}\mathrm{crys}}).
\]
quasi-inverse to each other.
\end{cor}

Finally, we establish a relationship between the category of crystals on the $m$-$q$-crystalline site and the category of certain stratifications.

\begin{defi}
\label{defstra}
\begin{enumerate}
\item Let $(D,I)$ be a $q$-PD pair, let $J = (\phi ^m)^{-1}(I)$ and let $X$ be a $p$-adic formal scheme smooth and separated over $D/J$. 
We define the category $\stra{m}{(D,I)}$ as follows.  Objects are maps $\str : T \to R$ where $T$ is a finite set and Spf$(R) \subseteq X$ is an affine open formal subscheme such that the induced map $D/(p,J) \to R/p$ is standard smooth in the sense of Tag 00T6 of \cite{Sta}; moreover, $\str$ is a map satisfying the following conditions:
\begin{enumerate}
\item $f_\str : D_\str := D[x_t]^{\wedge}_{t \in T} \to R$ \ \ \ is surjective. \\
$\ \ \ \ \ \ \ \ \ \ \ \ \ \ \ \ \ \ \   x_t \ \ \ \mapsto \str(t)$
\item There exists a sequence $y_1, \dots , y_r \in D[x_t]^{\wedge}_{t \in T}$ such that the kernel of $f_\str$ can be described as the ideal $(J,(y_w)_{w \in W})$ (where $W = \{ 1, \dots , r \}$), and that $\overline{y_1}, \dots , \overline{y_r} \in D/(p,J)[x_t]_{t \in T}$ form a regular sequence. 
\end{enumerate}
A morphism from $\str : T \to R$ to $\str' : T' \to R'$ in $\stra{m}{(D,I)}$ is a pair $(f,g)$ where $f$ is a map of sets $f : T \to T'$ and $g$ is a map of $p$-complete rings $g : R \to R'$, such that the diagram
\[
\begin{tikzcd}
T\arrow[r,"\str"]\arrow[d,"f"] & R\arrow[d,"g"] \\
T'\arrow[r,"\str'"] & R'.
\end{tikzcd}
\] 
commute and that the map $g^* : \text{Spf}(R') \to \text{Spf}(R)$ corresponding to $g$ is \linebreak compatible with the open immersions $\text{Spf}(R') \subseteq X, \text{Spf}(R) \subseteq X$.

Let $\str : T \to R$ and $\str' : T' \to R'$ be objects of $\stra{m}{(D,I)}$. Their sum $\str \sqcup \str'$ is given by the map
\[
T \sqcup T' \to R \widehat{\otimes}_{D/J} R' \to R'',
\]
where the first map is given by $t \mapsto \str(t) \otimes 1 (t \in T), t' \mapsto 1 \otimes \str'(t') (t' \in T')$, $R''$ is defined by the equality Spf$(R'')$ = Spf$(R)$ $\cap$ Spf$(R') \subseteq X$ (it is well-defined as $X$ is separated) 
and the second map is induced by the open immersions $\text{Spf}(R'') \subseteq \text{Spf}(R), \text{Spf}(R'') \subseteq \text{Spf}(R')$. We shall simply write $\stra{m}{D}$ instead of $\stra{m}{(D,I)}$ if no confusion arises.

\item For $\str : T \to R \in \stra{m}{D}$, we define the ring  $S_\str$ in the same way as in the proof of Proposition \ref{pr:4q} : 
namely, we define $S_\str = D\{x_t, \frac{\phi^{m+1}(y_w)}{[p]_q}\}^\wedge_{t \in T, w \in W}$ to be the $q$-PD envelope of $(D\{x_t\}^\wedge_{t \in T},K)$, where the ideal $K = (I,(\phi^m(y_w))_{w \in W})$.
Note that it is independent of the choice of the elements $y_w (w \in W)$. Indeed, the ideal $(J, (y_w)_{w \in W})$ is independent of the choice when regarded as an ideal of $D\{x_t\}^\wedge_{t \in T}$, and so is the ideal $(I,(\phi^m(y_w))_{w \in W}) = (I, \phi^m((J, (y_w)_{w \in W})))$. 
Since the $q$-PD envelope only depends on the ideal $(I,(\phi^m(y_w))_{w \in W})$ by Lemma 16.10 of \cite{BS19}, this gives the desired independence. In particular, we see that the construction of $S_\str$ is functorial in $\str \in \stra{m}{D}$.

\item A \textit{stratification} with respect to $\stra{m}{D}$ and  $\mathcal{M}_{\prism}$ is a pair 
\[
((M_\str)_{\str \in \stra{m}{D}}, (\varphi_{\str \str'})_{\str \to \str'}),
\]
 where $M_\str \in \modbb{S_\str}$ and $\varphi_{\str \str'} : M_\str \widehat{\otimes}_{S_\str} S_{\str'}  \xrightarrow{\simeq} M_{\str'}$ is an isomorphism of $S_{\str'}$-modules satisfying the cocycle condition. 
If $(D, I)$ is a $q$-PD pair with $q = 1$ in $D$, then we can also define a \textit{stratification} with respect to $\stra{m}{D}$ and $\mathcal{M}$  as above by replacing $\mathcal{M}_{\prism}$ with $\mathcal{M}$. 
We denote the category of stratifications with respect to $\stra{m}{D}$ and  $\mathcal{M}_{\prism}$ (resp. $\mathcal{M}$) by $\text{Str}_\prism(\stra{m}{D})$ (resp. $\text{Str}(\stra{m}{D})$).
\end{enumerate}
\end{defi}

\begin{prop}
\label{pr:qstra}
There is an equivalence of categories
\[
\mathscr{C}_{\prism}((X/D)_{m\text{-}q\text{-}\mathrm{crys}}) \xrightarrow{\simeq} \mathrm{Str}_\prism(\stra{m}{D}).
\]
\end{prop}

\begin{proof}
The functor $\mathscr{C}_{\prism}((X/D)_{m\text{-}q\text{-}\mathrm{crys}}) \to \text{Str}_\prism(\stra{m}{D})$ is given by
\[
\mathscr{F} \mapsto ((\mathscr{F}(S_\str))_{\str \in \stra{m}{D}}, (\mathscr{F}(S_\str) \widehat{\otimes}_{S_\str} S_{\str'} \xrightarrow{\simeq} \mathscr{F}(S_{\str'}))_{\str \to \str'}).
\]

We can define the functor $\text{Str}_\prism(\stra{m}{D}) \to \mathscr{C}_{\prism}((X/D)_{m\text{-}q\text{-}\mathrm{crys}})$ as follows.
Given $((M_\str)_{\str \in \stra{m}{D}}, (\varphi_{\str \str'})_{\str \to \str'}) \in \text{Str}_\prism(\stra{m}{D})$ and $(E,I_E) \in (X/D)_{m\text{-}q\text{-}\mathrm{crys}}$, we choose an affine open Spf$(R) \subseteq X$ such that Spf$(E/J_E) \to X$ factors through Spf$(R)$. 
 There exists an open cover of $\text{Spf}(R)$ by standard opens $\text{Spf}(\widehat{R_{g_i}})$ such that each $\widehat{R_{g_i}}/p$ is standard smooth over $D/(p,J)$ as in the proof of Proposition \ref{pr:4q}.
Then the fiber product $\text{Spf}(\widehat{R_{g_i}} \widehat{\otimes}_{R} E/J_E)$ is an affine open of $\text{Spf}(E/J_E)$. We denote the formal \linebreak scheme $\text{Spf}(\widehat{R_{g_i}} \widehat{\otimes}_{R} E/J_E)$ by $\text{Spf}(\overline{E_i})$. By Lemma \ref{qpdlift}, there is a unique $q$-PD pair $(E_i, \widehat{I_EE_i}) \in (X/D)_{m\text{-}q\text{-}\mathrm{crys}}$ for which the corresponding formal scheme $\text{Spf}(E_i)$ is an affine open formal subscheme of $\text{Spf}(E)$, and lifts $\text{Spf}(\overline{E_i})$.

First, we define $\mathscr{F}(E_i)$ for such $E_i$. By construction, Spf$(E_i/J_{E_i}) \to X$ fac\nolinebreak tors through some affine open Spf$(\widetilde{R})$ which satisfies the condition in Definition \ref{defstra}.1, so there exists an object in $\stra{m}{D}$ of the form $\str : T \to \widetilde{R}$. Then we have a map \linebreak $f_1 : D[x_t]^\wedge_{t \in T} \to E_i/J_{E_i}$ determined as  the composition of the map $f_\str : D[x_t]^\wedge_{t \in T} \to \widetilde{R}$ and $\widetilde{R} \to E_i/J_{E_i}$. 
As $D[x_t]^\wedge_{t \in T}$ is the completion of a polynomial ring, one can choose a map $f_2 : D[x_t]^\wedge_{t \in T} \to E_i$ lifting $f_1$. 
As $E_i$ is a $\delta$-$D$-algebra, $f_2$ extends uniquely to a $\delta$-$D$-algebra map $f_3 : D\{x_t\}^\wedge_{t \in T} \to E_i$. Since the image of $y_w (w \in W)$ under $f_3$ belongs to $J_{E_i}$, the image of $\phi^m(y_w)$ belongs to $I_{E_i}$. 
As $f_3$ carries $K := (I,(\phi^m(y_w))_{w \in W})$ into $I_{E_i}$ and $S_\str$ is the $q$-PD envelope of $(D\{x_t\}^\wedge_{t \in T},K)$, $f_3$ extends uniquely to a $\delta$-$D$-algebra map $g : S_\str \to E_i$ in $(X/D)_{m\text{-}q\text{-}\mathrm{crys}}$. We define
\[
\mathscr{F}(E_i) := M_\str \widehat{\otimes}_{S_\str} E_i \in \modbb{E_i}.
\]
We must show that this is well-defined. If we choose another affine open Spf$(\widetilde{R}')$, $\str' : T' \to \widetilde{R}'$ which is an object in $\stra{m}{D}$ and a morphism $S_{\str'} \to E_i$, then we have a diagram
\[
\begin{tikzcd}
S_\str \arrow[dr,bend left=10]\arrow[d] &\\
S_{\str \sqcup \str'} \arrow[r] & E_i. \\
S_{\str'} \arrow[ur,bend right=10]\arrow[u] &
\end{tikzcd}
\]
So we have the isomorphism
\begin{align*}
M_\str \widehat{\otimes}_{S_\str} E_i &= (M_\str \widehat{\otimes}_{S_\str} S_{\str \sqcup \str'}) \widehat{\otimes}_{S_{\str \sqcup \str'}} E_i \\
&\xrightarrow[\simeq]{\varphi_{\str,\str \sqcup \str'}}  M_{\str \sqcup \str'} \widehat{\otimes}_{S_{\str \sqcup \str'}} E_i \\
&\xleftarrow[\simeq]{\varphi_{\str',\str \sqcup \str'}}  (M_{\str'} \widehat{\otimes}_{S_{\str'}} S_{\str \sqcup \str'}) \widehat{\otimes}_{S_{\str \sqcup \str'}} E_i \\
&= M_{\str'} \widehat{\otimes}_{S_{\str'}} E_i.
\end{align*}
Hence, $\mathscr{F}(E_i)$ is independent of the choice. Moreover, for a morphism \[(E_i, I_{E_i}) \to (E',I_{E'})\] in $(X/D)_{m\text{-}q\text{-}\mathrm{crys}}$, $\mathscr{F}(E')$ is defined as in the case of $\mathscr{F}(E_i)$, and there exists a natural isomorphism
\[
\mathscr{F}(E_i) \widehat{\otimes}_{E_i} E' \xrightarrow{\simeq} \mathscr{F}(E')
\]
by definition.

Next, we define $\mathscr{F}(E)$ for general $(E,I_E) \in (X/D)_{m\text{-}q\text{-}\mathrm{crys}}$. Take $(E_i)_i$ for $E$ as in the first paragraph. Then there exists a cover $(E,I_E) \to \prod_i(E_i,\ol{I_EE_i})$ in $(X/D)_{m\text{-}q\text{-}\mathrm{crys}}$. Then we define $\mathscr{F}(E)$ as the kernel of $\underset{i}{\prod} \mathscr{F}(E_i) \rightrightarrows \underset{i,j}{\prod} \mathscr{F}(E_i \widehat{\otimes}_{E} E_j)$: As
\[
\mathscr{F}(E_i) \widehat{\otimes}_{E_i} (E_i \widehat{\otimes}_{E} E_j) \cong \mathscr{F}(E_i \widehat{\otimes}_{E} E_j) \cong \mathscr{F}(E_j) \widehat{\otimes}_{E_j} (E_i \widehat{\otimes}_{E} E_j)
\]
and each $\mathscr{F}(E_i)$ belongs to $\modbb{E_i}$, we see that $\mathscr{F}(E) \in \modbb{E}$ and that \[\mathscr{F}(E) \widehat{\otimes}_{E} E_i \cong \mathscr{F}(E_i).\]

For a morphism $(E,I_E) \to (E', I_{E'})$ in $(X/D)_{m\text{-}q\text{-}\mathrm{crys}}$, we have the isomorphisms
\begin{align*}
\mathscr{F}(E_i) \widehat{\otimes}_{E_i} (E_i \widehat{\otimes}_{E} E') &\xrightarrow{\simeq} \mathscr{F}(E_i \widehat{\otimes}_{E} E'), \\
\mathscr{F}(E_i \widehat{\otimes}_{E} E_j) \widehat{\otimes}_{E_i \widehat{\otimes}_{E} E_j} (E_i \widehat{\otimes}_{E} E_j \widehat{\otimes}_{E} E') &\xrightarrow{\simeq} \mathscr{F}(E_i \widehat{\otimes}_{E} E_j \widehat{\otimes}_{E} E')
\end{align*}
and these induce the isomorphism
\[
\mathscr{F}(E) \widehat{\otimes}_{E} E' \xrightarrow{\simeq} \mathscr{F}(E').
\]
 Thus, $\mathscr{F} \in \mathscr{C}_{\prism}((X/D)_{m\text{-}q\text{-}\mathrm{crys}})$. So the functor $\text{Str}_\prism(\stra{m}{D}) \to \mathscr{C}_{\prism}((X/D)_{m\text{-}q\text{-}\mathrm{crys}})$ can be defined by 
\[
((M_\str)_{\str \in \stra{m}{D}}, (\varphi_{\str \str'})_{\str \to \str'}) \mapsto \mathscr{F}.
\] The two functors we constructed are quasi-inverse to each other.
\qed
\end{proof}

Assume that $(D, I)$ is a $q$-PD pair with $q = 1$ in $D$. Then we can prove the result for $\mathscr{C}((X/D)_{m\text{-}q\text{-}\mathrm{crys}})$ in the same way as Proposition \ref{pr:qstra}.

\begin{cor}
Assume that $(D, I)$ is a $q$-PD pair with $q = 1$ in $D$. Then there is an equivalence of categories
\[
\mathscr{C}((X/D)_{m\text{-}q\text{-}\mathrm{crys}}) \xrightarrow{\simeq} \mathrm{Str}(\stra{m}{D}).
\]
\end{cor}

\section{The \textit{m}-prismatic site and the (\textit{m}$-$1)-\textit{q}-crystalline site}

In \cite{BS19}, Bhatt and Scholze showed that $q$-crystalline cohomology of a smooth formal scheme $X$ is isomorphic to the prismatic cohomology of the pullback of $X$ by the Frobenius lift $\phi$ on the base ring.
 As the latter is isomorphic to the level 1-prismatic cohomology of $X$ by our result, it is then natural to wonder if the $m$-prismatic site can be compared with the $(m-1)$-$q$-crystalline site. 
In this section, we prove that there exists an equivalence between the category of crystals on the $m$-prismatic site and that on the $(m-1)$-$q$-crystalline site.

\begin{con}
Let $(D,I)$ be a $q$-PD pair, let $J_q = (\phi^{m-1})^{-1}(I)$ and let \linebreak $J_\prism = (\phi^{m})^{-1}([p]_qD)$. 
Note that  $I \subseteq \phi^{-1}([p]_qD)$ by Corollary 16.8 of \cite{BS19}. In particular, we have $J_q \subseteq J_\prism$. 
Let $X$ be a $p$-adic formal scheme smooth and separated over $D/J_q$, and let $\widetilde{X}$ be  $X \underset{\mathrm{Spf}(D/J_q)}{\widehat{\times}}$ Spf$(D/J_\prism)$, where the completion is the classical $p$-completion. Then we have a  diagram:
\[
\begin{tikzcd}
\mathrm{Spf}(D)\arrow[r,hookleftarrow] & \mathrm{Spf}(D/J_q)\arrow[dr,phantom, "\Box"] & X\arrow[l]\\
\mathrm{Spf}(D)\arrow[r,hookleftarrow]\arrow[u,"\simeq"] & \mathrm{Spf}(D/J_\prism)\arrow[u,hookrightarrow] & \widetilde{X}.\arrow[l]\arrow[u,hookrightarrow]
\end{tikzcd} 
\]
We can define a functor $\alpha$ from the $(m-1)$-$q$-crystalline site of $X$ over $(D,I)$ to the $m$-prismatic site of $\widetilde{X}$ over $(D,[p]_qD)$ in the following way: 
For an object \[(\text{Spf}(E) \gets \text{Spf}(E/J_{E,q}) \to X)\] of $(X/(D,I))_{(m-1)\text{-}q\text{-crys}}$, we define the object $\alpha(\text{Spf}(E) \gets \text{Spf}(E/J_{E,q}) \to X)$ of $(\widetilde{X}/(D,[p]_qD))_{m-\prism}$ by \[(\text{Spf}(E) \gets \text{Spf}(E/J_{E,\,\prism}) \xrightarrow{f} \widetilde{X}),\] where the right map $f$ is defined as follows:
\[
\text{Spf}(E/J_{E,\,\prism}) \xrightarrow{g} \text{Spf}(E/J_{E,q}) \underset{\mathrm{Spf}(D/J_q)}{\widehat{\times}}\text{Spf}(D/J_\prism) \to X \underset{\mathrm{Spf}(D/J_q)}{\widehat{\times}}\text{Spf}(D/J_\prism) = \widetilde{X}.
\]
Here the first map is induced by the natural surjection of the corresponding rings. This defines the functor \[\alpha: (X/(D,I))_{(m-1)\text{-}q\text{-crys}} \to (\widetilde{X}/(D,[p]_qD))_{m-\prism}.\] One can check that $\alpha$ is continuous. So $\alpha$ gives a morphism of topoi
\[
\widehat{\alpha}: (\widetilde{X}/(D,[p]_qD))^{\sim}_{m-\prism} \to (X/(D,I))^{\sim}_{(m-1)\text{-}q\text{-crys}}.
\]
\end{con}

\begin{thm}
\label{th:mm1}
$\widehat{\alpha}$ induces an equivalence of categories of crystals
\[
\widehat{\alpha}_* : \mathscr{C}_{\prism}((\widetilde{X}/(D,[p]_qD))_{m-\prism}) \xrightarrow{\simeq} \mathscr{C}_{\prism}((X/(D,I))_{(m-1)\text{-}q\text{-}\mathrm{crys}}).
\]
The same holds true for $\mathscr{C}^{\mathrm{fp}}$ and $\mathscr{C}^{\mathrm{tors}}$.
\end{thm}
\begin{proof}
By Proposition \ref{pr:qstra}, we  see  that
\[
\mathscr{C}_{\prism}((X/(D,I))_{(m-1)\text{-}q\text{-}\mathrm{crys}}) \xrightarrow{\simeq} \mathrm{Str}_\prism(\stra{m-1}{D,X}),
\]
where $\stra{m-1}{D,X}$ is the category $\stra{m-1}{D}$ for $X$ constructed in Definition \ref{defstra}. It remains to show that there is a natural equivalence of categories
\[
\mathscr{C}_{\prism}((\widetilde{X}/(D,[p]_qD))_{m-\prism}) \xrightarrow{\simeq}
\mathrm{Str}_\prism(\stra{m-1}{D,X}).
\] 

First, we define the functor $\mathscr{C}_{\prism}((\widetilde{X}/(D,[p]_qD))_{m-\prism}) \to \text{Str}_\prism(\stra{m-1}{D,X})$. 
For any object $(\str : T \to R) \in \stra{m-1}{D,X}$, we define $\widetilde{R}' = R \widehat{\otimes}_{D/J_q} D/J_{\prism}$. 
If we describe the kernel of the surjection $f_\str : D[x_t]^{\wedge}_{t \in T} \to R$ as the ideal $(J_q,(y_w)_{w \in W})$ as in Definition \ref{defstra}, then the kernel of the surjection $D[x_t]^{\wedge}_{t \in T} \to R \to \widetilde{R}'$ is the ideal $(J_{\prism},(y_w)_{w \in W})$.  We can construct the $q$-PD envelope $S_\str$ as in Definition \ref{defstra}. Then $S_\str$ can be regarded as an object of $(X/(D,I))_{(m-1)\text{-}q\text{-}\mathrm{crys}}$ as in the proof of  Proposition \ref{pr:4q}. 
As $\phi^{m-1}(y_w)$ belongs to $I_{S_\str,q}$, we see that $y_w \in J_{S_\str,q}$. By the containment $J_{S_\str,q} \subseteq J_{S_\str,\prism}$, it follows that $y_w \in J_{S_\str,\prism}$. This gives a map $\widetilde{R}' \to S_{\str}/J_{S_\str,\prism}$. So we may also regard $S_\str$  as an object of $(\widetilde{X}/(D,[p]_qD))_{m-\prism}$. Thus, we can define the functor 
\[
\mathscr{C}_{\prism}((\widetilde{X}/(D,[p]_qD))_{m-\prism}) \to \text{Str}_\prism(\stra{m-1}{D,X})
\]
by
\[
\mathscr{F} \mapsto ((\mathscr{F}(S_\str))_{\str \in \stra{m-1}{D,X}}, (\mathscr{F}(S_\str) \widehat{\otimes}_{S_\str} S_{\str'} \xrightarrow{\simeq} \mathscr{F}(S_{\str'}))_{\str \to \str'}).
\]

Next we define the functor 
\[
\text{Str}_\prism(\stra{m-1}{D,X}) \to \mathscr{C}_{\prism}((\widetilde{X}/(D,[p]_qD))_{m-\prism})
\]
 as follows. Given $((M_\str)_{\str \in \stra{m-1}{D,X}}, (\varphi_{\str \str'})_{\str \to \str'}) \in \text{Str}_\prism(\stra{m-1}{D,X})$ and \[(\text{Spf}(E) \gets \text{Spf}(E/J_E) \to \widetilde{X}) \in (\widetilde{X}/(D,[p]_qD))_{m-\prism},\] we choose an affine open formal subscheme Spf$(\widetilde{R}) \subseteq \widetilde{X}$ such that Spf$(E/J_E) \to \widetilde{X}$ factors through Spf$(\widetilde{R})$. 
Since the map $\widetilde{X} \to X$ is a closed immersion defined by the ideal $J_{\prism}$ which is a nil ideal in $D/J_q$, there is an open formal subscheme $U \subseteq X$ which lifts $\text{Spf}(\widetilde{R})$. 
Then the map $U \to \text{Spf}(D/J_q)$ is a  smooth morphism which lifts the morphism $\text{Spf}(\widetilde{R}) \to \text{Spf}(D/J_{\prism})$. By Lemma \ref{lemlift}, we see that $U$ is affine. We denote the formal scheme $U$ by $ \text{Spf}(R)$. Then we have a diagram:
\[
\begin{tikzcd}
& \text{Spf}(R) \arrow[r,hookrightarrow] &X \arrow[r]& \text{Spf}(D/J_q)\\
\text{Spf}(E/J_E) \arrow[r]& \text{Spf}(\widetilde{R})\arrow[r,hookrightarrow]\arrow[u,hookrightarrow] & \widetilde{X} \arrow[r]\arrow[u]& \text{Spf}(D/J_{\prism}).\arrow[u,hookrightarrow]
\end{tikzcd}
\]
There exists an open cover $\text{Spf}(R) = \bigcup_j \text{Spf}(\widehat{R_{g_i}})$ such that each $\widehat{R_{g_i}}/p$ is standard smooth over $D/(p,J_q)$ as in the proof of Proposition \ref{pr:qstra}.  Then the fiber product $\text{Spf}(\widehat{R_{g_i}} \widehat{\otimes}_{R} \widetilde{R})$ is an affine open of $\text{Spf}(\widetilde{R})$. We denote the formal scheme $\text{Spf}(\widehat{R_{g_i}} \widehat{\otimes}_{R} \widetilde{R})$ by $\text{Spf}(\widetilde{R}_i)$. We also have the fiber product $\text{Spf}(\widetilde{R}_i \widehat{\otimes}_{\widetilde{R}} E/J_E)$ which is an affine open of $\text{Spf}(E/J_E)$. We denote the formal scheme $\text{Spf}(\widetilde{R}_i \widehat{\otimes}_{\widetilde{R}} E/J_E)$ by $\text{Spf}(\overline{E_i})$. By Lemma \ref{prislift}, there is a unique  bounded prism $(E_i, [p]_qE_i) \in (\widetilde{X}/(D,[p]_qD))_{m-\prism}$ for which the corresponding formal scheme $\text{Spf}(E_i)$ is an affine open formal subscheme of $\text{Spf}(E)$, and lifts $\text{Spf}(\overline{E_i})$.

First, we define $\mathscr{F}(E_i)$ for such $E_i$. By construction, there exists an object in $\stra{m-1}{D,X}$ of the form $\str : T \to \widehat{R_{g_i}}$. 
Let $S_\str$ be the $q$-PD envelope constructed from $\str$ in Definition \ref{defstra}.2. Explicitly, the ring $S_\str$ is given by $D\{x_t, \frac{\phi^{m}(y_w)}{[p]_q}\}^\wedge_{t \in T, w \in W}$, where $(y_w)_{w \in W}$ is a sequence of elements in $D[x_t]^\wedge_{t \in T}$ in Definition \ref{defstra}.1. Then, noting \linebreak that $(\phi^m(y_w))_{w \in W}$ is $(p,[p]_q)$-completely regular relative to $D$ by an argument \linebreak similar to that in Corollary \ref{cor:reg seq}, we see that $S_\str$ is the prismatic envelope of \[(D\{x_t\}^\wedge_{t \in T}, ([p]_q, (\phi^m(y_w))_{w \in W})).\] In particular, we can regard $S_\str$ as an object of $(\widetilde{X}/(D,[p]_qD))_{m-\prism}$, and then we can construct a morphism $S_\str \to E_i$ in $(\widetilde{X}/(D,[p]_qD))_{m-\prism}$ as in the proof of Proposition \ref{pr:4}. We define
\[
\mathscr{F}(E_i) := M_\str \widehat{\otimes}_{S_\str} E_i \in \modbb{E_i}.
\]
We can also define $\mathscr{F}(E)$ for general $(E,[p]_qE) \in (\widetilde{X}/(D,[p]_qD))_{m-\prism}$  as in the proof of Proposition \ref{pr:qstra}.

We can prove that the presheaf $\mathscr{F}$ is well-defined and that it defines an object of $\mathscr{C}_{\prism}((\widetilde{X}/(D,[p]_qD))_{m-\prism})$ as in the proof of Proposition \ref{pr:qstra}. So the functor \[\text{Str}_\prism(\stra{m-1}{D,X}) \to \mathscr{C}_{\prism}((\widetilde{X}/(D,[p]_qD))_{m-\prism})\] is defined by
\[
((M_\str)_{\str \in \stra{m-1}{D,X}}, (\varphi_{\str \str'})_{\str \to \str'}) \mapsto \mathscr{F}.
\]
The two functors we constructed are quasi-inverse to each other. Hence, the category $\mathscr{C}_{\prism}((\widetilde{X}/(D,[p]_qD))_{m-\prism})$ is equivalent to the category $\text{Str}_\prism(\stra{m-1}{D,X})$, as desired.
\qed
\end{proof}

\section{Relation to the Frobenius descent}

The equivalences of categories of crystals proved in the previous sections are modeled on the Frobenius descent, which is due to Berthelot: there exists an equivalence between the category of crystals on the $m$-crystalline site and that on the usual crystalline site. 
However, the Frobenius descent was proved by identifying the  crystals with the stratifications. It did not follow from a certain equivalence of topoi because the $m$-crystalline site was not suitable enough to apply the site-theoretic argument. 

In this section, we first give an alternative, site-theoretic proof of the Frobenius descent in a certain setting. Our strategy is to suitably  modify the definition of the $m$-crystalline site  without changing the category of crystals. Then we  apply the site-theoretic argument in the previous sections to this modified site. 
Next, assuming that $(D,I)$ is a $q$-PD pair with $q=1$ in $D$ and $p \in I$, we use the modified version of the $m$-crystalline site to  prove that, the equivalence between the category of crystals on the $m$-$q$-crystalline site and that on the usual $q$-crystalline site in Section 2 is compatible with the Frobenius descent.

First, we recall the definition of the $m$-PD ring.

\begin{defi}
\label{mpd}
For a $\mathbf{Z}_{(p)}$-ring $D$ and an ideal $J$ of $D$, an $m$-PD structure on $J$ is a PD ideal $(I,\gamma)$ of $D$  satisfying the following conditions:
\begin{enumerate}
\item $J^{(p^m)} + pJ \subseteq I \subseteq J$, where $J^{(p^m)}$ is the ideal generated by $x^{p^m}$ for all elements $x$ of $J$.
\item The PD structure $\gamma$ is compatible with the unique one on $p\mathbf{Z}_{(p)}$.
\end{enumerate}
We call the triple $(J,I,\gamma)$ an $m$-PD ideal of $D$ and the quadruple $(D,J,I,\gamma)$ an $m$-PD ring.
\end{defi}

\begin{rem}
\begin{enumerate}
\item We warn the reader that the use of the letters $I,J$ in Definition \ref{mpd} is reversed, compared to the one in \cite{Ber90}. We prefer to use the notation in Definition \ref{mpd} because it is compatible with the one in Section 2.
\item When $m = 0$, the notion of a $m$-PD structure on an ideal $J$ is nothing but a PD structure on $J$ compatible with the unique one on $p\mathbf{Z}_{(p)}$.
\end{enumerate}
\end{rem}

Next, we recall the definition of the $m$-crystalline site. Note that in this paper, we consider the site which is `affine', `big', with respect to the `flat topology', and possibly over the `$p$-adic base'. So it is not exactly the same as the original definition.

\begin{defi}
\label{mcryssite}
Let $(D,J,I,\gamma)$ be a $p$-torsion free $p$-complete $m$-PD ring or an $m$-PD ring in which $p$ is nilpotent, and
 suppose that the ideal $I$ contains $p$. Let $X$ be a scheme smooth and separated over $D/J$. 
We define the {\it $m$-crystalline site} $(X/D)_{m\text{-crys}}$ of $X$ over $D$ as follows. Objects are maps $(D,J,I,\gamma) \to (E,J_E,I_E,\gamma_E)$ of $m$-PD rings together with a map Spec$(E/J_E) \to X$ over $D/J$ satisfying the following conditions:
\begin{enumerate}
\item There exists some $n \ge 0$ such that $p^nE=0$.
\item Spec$(E/J_E) \to X$ factors through some affine open Spec$(R) \subseteq X$.
\end{enumerate}
We shall often denote such an object by
\[
(\text{Spec}(E) \gets \text{Spec}(E/J_E) \to X) \in (X/D)_{m\text{-crys}}
\]
or $(E,J_E,I_E,\gamma_E)$ if no confusion arises. 
A morphism 
\[
(\text{Spec}(E') \gets \text{Spec}(E'/J_{E'}) \to X) \to (\text{Spec}(E) \gets \text{Spec}(E/J_E) \to X)
\]
 is a map of $m$-PD rings $(E,J_E,I_E,\gamma_E) \to (E',J_{E'},I_{E'},\gamma_{E'})$ over $(D,J,I,\gamma)$ such that the induced morphism
$\text{Spec}(E'/J_{E'}) \to \text{Spec}(E/J_E)$
 is compatible with the maps $\text{Spec}(E'/J_{E'}) \to \nolinebreak X$, $\text{Spec}(E/J_E) \to X$. When we denote such an object by $(E,J_E,I_E,\gamma_E)$, we shall write \[(E,J_E,I_E,\gamma_E) \to (E',J_{E'},I_{E'},\gamma_{E'})\] (not $(E',J_{E'},I_{E'},\gamma_{E'}) \to (E,J_E,I_E,\gamma_E)$)
 for a morphism from $(E',J_{E'},I_{E'},\gamma_{E'})$ to $(E,J_E,I_E,\gamma_E)$. 

A map $(E,J_E,I_E,\gamma_E) \to (E',J_{E'},I_{E'},\gamma_{E'})$  in  $(X/D)_{m\text{-crys}}$ is a cover if it is a  faithfully flat map and satisfies $J_{E'} = J_EE', I_{E'} = I_EE'$.
\end{defi}

We define the categories of crystals on $(X/D)_{m\text{-crys}}$ as follows.

\begin{defi}
\begin{enumerate}
\item Let $\mathscr{C}^{\text{qcoh}}((X/D)_{m\text{-crys}})$  be the category of abelian sheaves $\mathscr{F}$ on the $m$-crystalline site $(X/D)_{m\text{-crys}}$ such that, for any object $(E,J_E,I_E,\gamma_E)$  in $(X/D)_{m\text{-crys}}$,  $\mathscr{F}(E,J_E,I_E,\gamma_E)$ is an $E$-module, and for any  morphism \[(E,J_E,I_E,\gamma_E) \to (E_1,J_{E_1},I_{E_1},\gamma_{E_1})\] in $(X/D)_{m\text{-crys}}$, $\mathscr{F}(E,J_E,I_E,\gamma_E)  \to \mathscr{F}(E_1,J_{E_1},I_{E_1},\gamma_{E_1})$ is compat\nolinebreak ible with the module structures in the usual sense and 
the canonical map \[\mathscr{F}(E,J_E,I_E,\gamma_E) \otimes_E E_1 \to \mathscr{F}(E_1,J_{E_1},I_{E_1},\gamma_{E_1})\] is an isomorphism of $E_1$-modules. 
\item Let $\mathscr{C}^{\text{tors}}((X/D)_{m\text{-crys}})$ be the category of abelian sheaves defined by
\[
\mathscr{C}^{\text{tors}}((X/D)_{m\text{-crys}}) = \displaystyle\bigcup_{n}\mathscr{C}^{p^n\text{-tors}}((X/D)_{m\text{-crys}}),
\]
where $\mathscr{C}^{p^n\text{-tors}}((X/D)_{m\text{-crys}})$ is the full subcategory of $\mathscr{C}^{\text{qcoh}}((X/D)_{m\text{-crys}})$ \linebreak consisting of $p^n$-torsion objects. We see that $\mathscr{C}^{\text{tors}}((X/D)_{m\text{-crys}})$ is also a full subcategory of $\mathscr{C}^{\text{qcoh}}((X/D)_{m\text{-crys}})$.
\item Let $\mathscr{C}^{\text{fp}}((X/D)_{m\text{-crys}})$ be the full subcategory of $\mathscr{C}^{\text{qcoh}}((X/D)_{m\text{-crys}})$ consisting  of objects $\mathscr{F}$ such that $\mathscr{F}(E)$ is a finite projective $E$-module for any $(E,J_E,I_E,\gamma_E)$ in $(X/D)_{m\text{-crys}}$. 
\end{enumerate}
\end{defi}

\begin{rem}
\label{crysstab}
Our definition of $m$-crystalline site differs from the original one in the following sense:
\begin{enumerate}
\item The objects in our definition is `affine' in the sense that the left two schemes appearing in the diagram
\[
(\text{Spec}(E) \gets \text{Spec}(E/J_E) \to X)
\]
are always affine, while in the usual definition, one considers the objects of the form
\[
(T\gets U \to X)
\]
(with the conditions in our definition) where the schemes $T,U$ are not necessarily affine. We also assume that the morphism $\text{Spec}(E/J_E) \to X$ in the diagram above  factors through some affine open $\text{Spec}(R) \subseteq X$, but this condition is not imposed in the usual definition.
\item Our site is `big' in the sense that, in the definition of an object
\[
(\text{Spec}(E) \gets \text{Spec}(E/J_E) \xrightarrow{f} X),
\]
the map $f$ may be any map of schemes, while in the usual definition, the morphism $U \to X$ in the diagram in 1. is assumed to be an open immersion.
\item The topology in our definition is defined by flat covers, while the topology in the usual definition is defined by the Zariski covers.
\item The base ring $D$ can be a $p$-torsion free $p$-complete $m$-PD ring in our definition, while the base ring $D$ is only allowed to be an $m$-PD ring in which $p$ is nilpotent in the usual definition.
\end{enumerate}
By a standard argument, we see that the difference in 1. does not change the associated topos and the category of crystals $\mathscr{C}^{\text{qcoh}}((X/D)_{m\text{-crys}})$, and that the differences in 2. and 3. do not change the category of crystals $\mathscr{C}^{\text{qcoh}}((X/D)_{m\text{-crys}})$ 
(hence the categories $\mathscr{C}^{p^n\text{-tors}}((X/D)_{m\text{-crys}})$, $\mathscr{C}^{\text{tors}}((X/D)_{m\text{-crys}})$ and $\mathscr{C}^{\text{fp}}((X/D)_{m\text{-crys}})$ are also unchanged). For the difference in 4., we note that if $D$ is a $p$-torsion free $p$-complete $m$-PD ring, then there are equivalences of categories
\begin{align*}
&\mathscr{C}^{\text{qcoh}}((X/(D/p^nD))_{m\text{-crys}}) \simeq \mathscr{C}^{p^n\text{-tors}}((X/D)_{m\text{-crys}}),  \\
&\mathscr{C}^{\text{qcoh}}((X/D)_{m\text{-crys}}) \simeq \varprojlim_n \mathscr{C}^{p^n\text{-tors}}((X/D)_{m\text{-crys}}) \\
&\phantom{\mathscr{C}^{\text{qcoh}}((X/D)_{m\text{-crys}})} \simeq \varprojlim_n \mathscr{C}^{\text{qcoh}}((X/(D/p^nD))_{m\text{-crys}}).
\end{align*}
So one can recover the categories of crystals on $(X/(D/p^nD))_{m\text{-crys}}$'s from that on $(X/D)_{m\text{-crys}}$.
\end{rem}

Next, we review the Frobenius descent functor from the $m$-crystalline site to the usual crystalline site. 

\begin{con}
Let $(D,I,\gamma)$ be a $p$-torsion free $p$-complete PD ring or a PD ring in which $p$  is nilpotent. Suppose that the ideal $I$ contains $p$.  Assume that there exists a Cartesian diagram
\[
\begin{tikzcd}
\widetilde{X}\arrow[r,"f"]\arrow[dr,phantom, "\Box"] & \text{Spec}(D/I) \\
X'\arrow[r]\arrow[u] & \text{Spec}(D/I),\arrow[u,"(\phi ^m)^*"]
\end{tikzcd} 
\]
where $f$ is a smooth and separated map, and $\phi$ is the Frobenius on $D/I$. Note that in this situation, $(D,I,I,\gamma)$ is an $m$-PD ring. So we can define the $m$-crystalline site $(\widetilde{X}/(D,I,I,\gamma))_{m\text{-crys}}$.

We define a functor $\widetilde{\sigma}$ from $(\widetilde{X}/(D,I,I,\gamma))_{m\text{-crys}}$ to $(X'/(D,I,\gamma))_{\text{crys}}$ in the \linebreak  following  way: 
For an object $(\text{Spec}(E) \gets \text{Spec}(E/J_E) \to \widetilde{X})$ of $(\widetilde{X}/(D,I,I,\gamma))_{m\text{-crys}}$, we define the  object $\widetilde{\sigma}(\text{Spec}(E) \gets \text{Spec}(E/J_E) \to \widetilde{X})$ of $(X'/(D,I,\gamma))_{\text{crys}}$ by
\[(\text{Spec}(E) \gets \text{Spec}(E/I_E) \xrightarrow{f} X'),\] where the right map $f$ is defined as follows:
\[ \mathrm{Spec}(E/I_E) \xrightarrow{g} \mathrm{Spec}(E/J_E) \underset{\mathrm{Spec}(D/I),(\phi ^m)^*}{\times} \mathrm{Spec}(D/I) \to \widetilde{X} \underset{\mathrm{Spec}(D/I),(\phi ^m)^*}{\times} \mathrm{Spec}(D/I). \]
Here, the first map $g$ is induced by the map of rings
\[
E/J_E \underset{D/I,\phi^m}{\otimes} D/I \to E/I_E ; \ \ \ \ \  e \otimes d \mapsto \phi^m(e)d.
\]
This defines the functor $\widetilde{\sigma} : (\widetilde{X}/(D,I,I,\gamma))_{m\text{-crys}} \to (X'/(D,I,\gamma))_{\text{crys}}$.
\end{con}

\begin{prop}
\label{sigmacoco}
The functor $\widetilde{\sigma}$ is cocontinuous.
\end{prop}
\begin{proof}
We can prove the result in the same way as Proposition \ref{pr:3}.
\qed
\end{proof}

By Proposition \ref{sigmacoco}, we obtain a morphism of topoi \[(\widetilde{X}/(D,I,I,\gamma))^{\sim}_{m\text{-crys}} \to (X'/(D,I,\gamma))^{\sim}_{\text{crys}}.\] By abuse of notation, we will denote it by $\widetilde{\sigma}$. The morphism of topoi $\widetilde{\sigma}$ induces a pullback functor of categories of crystals
\[
\widetilde{\sigma}^* : \mathscr{C}^{\mathrm{qcoh}}((X'/(D,I,\gamma))_{\text{crys}}) \to \mathscr{C}^{\mathrm{qcoh}}((\widetilde{X}/(D,I,I,\gamma))_{m\text{-crys}}),
\]
which we call the Frobenius descent functor. 
The Frobenius descent of Berthelot implies that $\widetilde{\sigma}^*$ is an equivalence: in fact, in the case where $p$ is nilpotent in $D$, it follows by gluing its local version in Corollaire 2.3.7 of \cite{Ber00}. In the case where $D$ is a $p$-torsion free $p$-complete ring, it follows from the last equivalence of Remark \ref{crysstab} (see also Th\'{e}or\`{e}me 4.1.3 of \cite{Ber00}).

In this section, we first give an alternative, site-theoretic proof of the equivalence $\widetilde{\sigma}^*$ in a certain setting, which uses techniques similar to those in the previous sections. 
We will be particularly interested in the case where $D$ is a $p$-torsion free $p$-complete ring.

We need to suitably modify  the $m$-crystalline site for our purposes. First, we recall the infinitesimal invariance of the category of crystals on the $m$-crystalline site. 

\begin{con}
Let $(D,J,I,\gamma)$ be a $p$-torsion free $p$-complete $m$-PD ring or an $m$-PD ring in which $p$ is nilpotent. Suppose that the ideal $I$ contains $p$. Let $(J_1,I_1,\gamma_1)$ be an $m$-PD subideal of $(J,I,\gamma)$. Assume that  there exists a Cartesian diagram
\[
\begin{tikzcd}
X_1\arrow[r,"f"]\arrow[dr,phantom, "\Box"] & \text{Spec}(D/J_1) \\
X\arrow[r]\arrow[u,hookrightarrow] & \text{Spec}(D/J),\arrow[u,hookrightarrow]
\end{tikzcd}
\]
where $f$ is a smooth and separated map. Then we have a cocontinuous functor of sites $\pi:(X_1/(D,J_1,I_1,\gamma_1))_{m\text{-crys}} \to (X/(D,J,I,\gamma))_{m\text{-crys}}$ defined by
\[
(\text{Spec}(E) \gets \text{Spec}(E/J_E) \to X_1) \mapsto (\text{Spec}(E) \gets X \times_{X_1} \text{Spec}(E/J_E) \to X).
\]
This is well-defined by the proof of Proposition 2.11 in \cite{Miy15}. So we obtain a \linebreak morphism of topoi $ (X_1/(D,J_1,I_1,\gamma_1))_{m\text{-crys}}^{\sim} \to (X/(D,J,I,\gamma))_{m\text{-crys}}^{\sim}$. By abuse of notation, we will denote it by $\pi$.
 Note that the inverse image functor $\pi^*$ is equal to the functor $i^{(m)}_{\text{cris} \, *}$ in Proposition 2.11 of \cite{Miy15}, except that we work with the big site with respect to the flat topology while Miyatani worked with the small site with respect to the Zariski topology.
\end{con}

Then the following holds true, which is sometimes called the infinitesimal invariance of the category of crystals on the $m$-crystalline site:

\begin{prop}
$\pi^*$ induces an equivalence of categories of crystals
\[
\mathscr{C}^{\mathrm{qcoh}}((X/(D,J,I,\gamma))_{m\text{-}\mathrm{crys}}) \xrightarrow{\simeq} \mathscr{C}^{\mathrm{qcoh}}((X_1/(D,J_1,I_1,\gamma_1))_{m\text{-}\mathrm{crys}}).
\]
\end{prop}

\begin{proof}
In the case where $p$ is nilpotent in $D$, the proposition follows from Corollary 2.14 of \cite{Miy15} (The difference of our site with Miyatani's one does not cause any problem because Remark \ref{crysstab} ensures that the categories of crystals are the same). 
In the case where $D$ is a $p$-torsion free $p$-complete ring, the proposition follows from the previous case and the last equivalence of Remark \ref{crysstab}.
\qed
\end{proof}

Next, we relate the functor $\widetilde{\sigma}$ to a functor which is more similar to the functors studied in the previous sections, by using the infinitesimal invariance of the category of crystals on the $m$-crystalline site.

\begin{con}
Let $(D,I,\gamma)$ be a $p$-torsion free $p$-complete PD ring with $p \in I$. Set
\[
J = \text{Ker}(D \twoheadrightarrow D/I \xrightarrow{\phi^m} D/I).
\]
Then $(D,J,I,\gamma)$ is an $m$-PD ring. Assume that there exists a diagram with Cartesian squares
\[
\begin{tikzcd}
\widetilde{X}\arrow[r,"f"]\arrow[dr,phantom, "\Box"] & \text{Spec}(D/I) \\
X\arrow[r]\arrow[u,hookrightarrow]\arrow[dr,phantom, "\Box"] & \text{Spec}(D/J)\arrow[u,hookrightarrow] \\
X'\arrow[r]\arrow[u] & \text{Spec}(D/I), \arrow[u,"(\phi^m)^*"]
\end{tikzcd} 
\]
where $f$ is a smooth and separated map. Then we have functors of sites as above:
\begin{align*}
\pi : (\widetilde{X}/(D,I,I,\gamma))_{m\text{-crys}} &\to (X/(D,J,I,\gamma))_{m\text{-crys}}, \\
\widetilde{\sigma} : (\widetilde{X}/(D,I,I,\gamma))_{m\text{-crys}} &\to (X'/(D,I,\gamma))_{\text{crys}}.
\end{align*}
 We can also define the functor $\sigma$ from $(X/(D,J,I,\gamma))_{m\text{-crys}}$ to $(X'/(D,I,\gamma))_{\text{crys}}$ in the following way:
For an object $(\text{Spec}(E) \gets \text{Spec}(E/J_E) \to X)$ of $(X/D)_{m\text{-crys}}$, we define the object $\sigma(\text{Spec}(E) \gets \text{Spec}(E/J_E) \to X)$ of $(X'/D)_{\text{crys}}$ by \[(\text{Spec}(E) \gets \text{Spec}(E/I_E) \xrightarrow{f} X'),\]  where the right map $f$ is defined as follows:
\[ \mathrm{Spec}(E/I_E) \xrightarrow{g} \mathrm{Spec}(E/J_E) \underset{\mathrm{Spec}(D/J),(\phi ^m)^*}{\times} \mathrm{Spec}(D/I) \to X \underset{\mathrm{Spec}(D/J),(\phi ^m)^*}{\times} \mathrm{Spec}(D/I). \]
Here, the first map $g$ is induced by the map of rings
\[
E/J_E \underset{D/J,\phi^m}{\otimes} D/I \to E/I_E ; \ \ \ \ \  e \otimes d \mapsto \phi^m(e)d.
\]
This defines the functor $\sigma : (X/(D,J,I,\gamma))_{m\text{-crys}} \to (X'/(D,I,\gamma))_{\text{crys}}.$
We can check that $\sigma \circ \pi = \widetilde{\sigma}$. Hence we have a commutative diagram:
\[
\begin{tikzcd}
\mathscr{C}^{\text{qcoh}}((X'/(D,I,\gamma))_{\text{crys}})\arrow[r,"\sigma^*"]\arrow[rd,bend right=10,"\widetilde{\sigma}^*"'] & \mathscr{C}^{\text{qcoh}}((X/(D,J,I,\gamma))_{m\text{-crys}}) \arrow[d,"\simeq"',"\pi^*"] \\
& \mathscr{C}^{\text{qcoh}}((\widetilde{X}/(D,I,I,\gamma))_{m\text{-crys}}).
\end{tikzcd} 
\]
So it is enough to consider the functor $\sigma$ to study the Frobenius descent.
The functor $\sigma$ is much closer to the functors between the sites of level $m$ and level $0$ in the previous sections.
\end{con}

However, the $m$-crystalline site is not suitable enough to apply the site-theoretic argument in the previous sections; for any object $(E,J_E,I_E,\gamma_E)$ of $(X/D)_{m\text{-crys}}$, $p$ is nilpotent in $E$, and the data $(J_E,\gamma_E)$ is not uniquely determined by the pair $(E,I_E)$. But we can overcome this difficulty by defining a variant of the $m$-crystalline site.  Fix a non-negative integer $m$.

\begin{defi}
Let $(D,I,\gamma)$ be a PD ring with $p \in I$. We define \[J_{\text{max}} := \text{Ker}(D \twoheadrightarrow D/I \xrightarrow{\phi^m} D/I).\] Then 
$(J_{\text{max}},I,\gamma)$ is an $m$-PD ideal and for any $m$-PD ideal of the form $(J,I,\gamma)$, we have $J \subseteq J_{\text{max}}$. A \textit{maximal m-PD ring} is an $m$-PD ring $(D,J,I,\gamma)$ satisfying $J = J_{\text{max}}$.
\end{defi}

Note that, for a $p$-torsion free $p$-complete maximal $m$-PD ring $(D,J,I,\gamma)$, the data $(J,\gamma)$ is uniquely determined by the pair $(D,I)$: indeed, the ideal $J$ is uniquely determined by $I$ and the PD structure is uniquely determined by $\gamma_n(x) = \frac{x^n}{n!}$.

Next, we define a variant of the $m$-crystalline site.

\begin{defi}
\label{defmcrysnew}
Let $(D,J,I,\gamma)$ be a $p$-torsion free $p$-complete maximal $m$-PD ring with $p \in I$. Let $X$ be a scheme smooth and separated over $D/J$. We define the {\it new $m$-crystalline site} $(X/D)_{m\text{-crys,new}}$ of $X$ over $D$ as follows. 
Objects are maps $(D,J,I,\gamma) \to (E,J_E,I_E,\gamma_E)$ of $p$-torsion free $p$-complete maximal $m$-PD rings together with a map Spec$(E/J_E) \to X$ over $D/J$ satisfying the following condition:
\begin{enumerate}
\item[(*)] Spec$(E/J_E) \to X$ factors through some affine open Spec$(R) \subseteq X$.
\end{enumerate}
We shall often denote such an object by
\[
(\text{Spf}(E) \gets \text{Spec}(E/J_E) \to X) \in (X/D)_{m\text{-crys,new}}
\]
or $(E,J_E,I_E,\gamma_E)$ if no confusion arises. A morphism 
\[
(\text{Spf}(E') \gets \text{Spec}(E'/J_{E'}) \to X) \to (\text{Spf}(E) \gets \text{Spec}(E/J_E) \to X)
\]
 is a map of maximal $m$-PD rings 
\[
(E,J_E,I_E,\gamma_E) \to (E',J_{E'},I_{E'},\gamma_{E'})
\]
 over $(D,J,I,\gamma)$ such that the induced morphism $\text{Spec}(E'/J_{E'}) \to \text{Spec}(E/J_E)$
 is \linebreak compatible with the maps $\text{Spec}(E'/J_{E'}) \to X$, $\text{Spec}(E/J_E) \to X$. When we denote such an object by $(E,J_E,I_E,\gamma_E)$, we shall write $(E,J_E,I_E,\gamma_E) \to (E',J_{E'},I_{E'},\gamma_{E'})$ (not $(E',J_{E'},I_{E'},\gamma_{E'}) \to (E,J_E,I_E,\gamma_E)$) for a morphism from $(E',J_{E'},I_{E'},\gamma_{E'})$ to $(E,J_E,I_E,\gamma_E)$. 
A map $(E,J_E,I_E,\gamma_E) \to (E',J_{E'},I_{E'},\gamma_{E'})$
 in  $(X/D)_{m\text{-crys,new}}$ is a cover if it is a $p$-completely faithfully flat map and satisfies $I_{E'} = I_EE'$ (but we do not require that $J_{E'} =  J_EE'$). (Note that $I_EE'$ is always closed in $p$-adic topology because $p \in I_EE'$.)

When $m=0$, we denote the site $(X/D)_{m\text{-crys,new}}$ simply by $(X/D)_{\text{crys,new}}$ and call it the {\it new crystalline site}.
\end{defi}

The new $m$-crystalline site is much closer to our notion of sites of higher level in the previous sections. We will prove the equivalence between the category of crystals on the new $m$-crystalline site and that on the new crystalline site by the site-theoretic argument as in the previous sections. We first construct a functor between these sites.

\begin{con}
Let $X'$ be $X \underset{\mathrm{Spec}(D/J),(\phi ^m)^*}{\times} \mathrm{Spec}(D/I)$. Then we have the functor of sites \[
\sigma_{\text{new}} : (X/D)_{m\text{-crys,new}} \to (X'/D)_{\text{crys,new}}
\]
that sends an object $(\text{Spf}(E) \gets \text{Spec}(E/J_E) \to X)$ of $(X/D)_{m\text{-crys,new}}$ to an object $(\text{Spf}(E) \gets \text{Spec}(E/I_E) \xrightarrow{f} X')$ of $(X'/D)_{\text{crys,new}}$, where the right map $f$ is defined as follows:
\[ \mathrm{Spec}(E/I_E) \xrightarrow{g} \mathrm{Spec}(E/J_E) \underset{\mathrm{Spec}(D/J),(\phi ^m)^*}{\times} \mathrm{Spec}(D/I) \to X \underset{\mathrm{Spec}(D/J),(\phi ^m)^*}{\times} \mathrm{Spec}(D/I). \]
Here, the first map $g$ is induced by the map of rings
\[
E/J_E \underset{D/J,\phi^m}{\otimes} D/I \to E/I_E ; \ \ \ \ \  e \otimes d \mapsto \phi^m(e)d.
\]
\end{con}

\begin{prop}
\label{sigmanewcoco}
The functor $\sigma_{\mathrm{new}}$ is cocontinuous.
\end{prop}
\begin{proof}
We can prove the result in the same way as Proposition \ref{pr:3}.
\qed
\end{proof}

By Proposition \ref{sigmanewcoco}, we obtain a morphism of the associated topoi. By abuse of notation, we will denote it by $\sigma_{\mathrm{new}}$. We have the following result.

\begin{prop}
The morphism $\sigma_{\mathrm{new}} : \widetilde{(X/D)}_{m\text{-}\mathrm{crys,new}} \to \widetilde{(X'/D)}_{\mathrm{crys,new}}$ is an equivalence of topoi.
\end{prop}
\begin{proof}
We can prove that the functor of sites $\sigma_{\text{new}} : (X/D)_{m\text{-crys,new}} \to (X'/D)_{\text{crys,new}}$ is fully faithful, continuous and cocontinuous in the same way as the propositions in Section 1, so the result follows from  Proposition \ref{prmcrys4} below and Proposition 4.2.1 of \cite{Oya17} (see also Proposition 9.10 of \cite{Xu19}).
\qed
\end{proof}

We will need the following analogs of Lemma \ref{prislift}.

\begin{lem}
\label{mcryslift}
Let $(E,J_E,I_E,\gamma_E)$ be an $m$-PD ring which is an object of $(X/D)_{m\text{-}\mathrm{crys}}$. Let $\overline{f} : \mathrm{Spec}(\overline{E_i}) \to \mathrm{Spec}(E/J_E)$ be an open immersion. Then there is a  unique open immersion  $f : \mathrm{Spec}(E_i) \to \mathrm{Spec}(E)$ which lifts $\overline{f}$. Moreover, there exists a PD structure $\gamma_{E_i}$ on $I_EE_i$ such that the corresponding map of rings induces a map $(E,J_E,I_E,\gamma_E) \to (E_i,J_EE_i,I_EE_i,\gamma_{E_i})$ in $(X/D)_{m\text{-}\mathrm{crys}}$.
\end{lem}
\begin{proof}
As  $J_E$  is a nil ideal, the existence, the uniqueness and the affinity of the lifting follow from Lemma \ref{lemlift}. We denote the lifting by \[f : \mathrm{Spec}(E_i) \to \mathrm{Spec}(E).\]
As $E_i$ is flat over $E$, by Corollary 3.22 of \cite{BO78}, there exists (uniquely) a PD structure \linebreak $\gamma_{E_i}$ on $I_EE_i$ such that the corresponding map of rings induces a map \[(E,I_E,\gamma_E) \to (E_i,I_EE_i,\gamma_{E_i})\] of PD rings. This induces a map $(E,J_E,I_E,\gamma_E) \to (E_i,J_EE_i,I_EE_i,\gamma_{E_i})$, as required.
\qed
\end{proof}

\begin{lem}
\label{newlift}
Let $(E,J_E,I_E,\gamma_E)$ be a $p$-torsion free $p$-complete maximal $m$-PD ring which is an object of $(X/D)_{m\text{-}\mathrm{crys,new}}$. Let $\overline{f}$ be an open immersion \[\overline{f} : \mathrm{Spec}(\overline{E_i}) \to \mathrm{Spec}(E/I_E)\] or an open immersion \[\overline{f} : \mathrm{Spec}(\overline{E_i}) \to \mathrm{Spec}(E/J_E).\] Then there is a unique open immersion  $f : \mathrm{Spf}(E_i) \to \mathrm{Spf}(E)$ which lifts $\overline{f}$. Moreover, there exists a PD structure $\gamma_{E_i}$ on $I_EE_i$ such that the corresponding map of rings induces a map $(E,J_E,I_E,\gamma_E) \to (E_i,J_{E_i},I_EE_i,\gamma_{E_i})$ in $(X/D)_{m\text{-}\mathrm{crys,new}}$.
\end{lem}
\begin{proof}
As the images of $J_E$ and $I_E$ in $E/pE$ are nil ideals, the existence, the uniqueness and the affinity of the lifting follow from Lemma \ref{lemlift}. We denote the lifting by \[f : \mathrm{Spf}(E_i) \to \mathrm{Spf}(E).\]
Note that $I_E/p^nE$ has naturally the PD structure induced by $\gamma_E$. Then, as \[E/p^nE \to E_i/p^nE_i\] is flat, by Corollary 3.22 of \cite{BO78}, there exists (uniquely) a PD structure on $I_EE_i/p^nE_i$ such that the map $(E/p^nE,I_E/p^nE) \to (E_i/p^nE_i,I_EE_i/p^nE_i)$ is the map of PD rings. Taking inverse limit with respect to $n$, we see that there exists a PD structure \linebreak $\gamma_{E_i}$ on $I_EE_i$ such that the map  \[(E,I_E,\gamma_E) \to (E_i,I_EE_i,\gamma_{E_i})\] is a map of PD rings. Also, the flatness of $E/p^nE \to E_i/p^nE_i$ for all $n$ and the $p$-torsion freeness of $E$ imply the $p$-torsion freeness of $E_i$. Hence the map \[(E,J_E,I_E,\gamma_E) \to (E_i,J_{E_i},I_EE_i,\gamma_{E_i})\] (where $J_{E_i} := \text{Ker}(E_i \twoheadrightarrow E_i/I_EE_i \xrightarrow{\phi^m} E_i/I_EE_i$)) is a map in $(X/D)_{m\text{-}\mathrm{crys,new}}$.
\qed
\end{proof}

\begin{prop}
\label{prmcrys4}
Let $(E',J_{E'},I_{E'},\gamma_{E'})$ be an object in $(X'/D)_{\mathrm{crys,new}}$. Then there exists an object $(E,J_E,I_{E},\gamma_E)$ in $(X/D)_{m\text{-}\mathrm{crys,new}}$ and a cover of the form \[(E',J_{E'},I_{E'},\gamma_{E'}) \to \sigma_{\mathrm{new}}(E,J_E,I_{E},\gamma_E).\]
\end{prop}
\begin{proof}
By Lemma \ref{newlift} and the argument as in the proof of Proposition \ref{pr:4}, we may assume that the structure morphism $\text{Spec}(E'/I_{E'}) \to X'$ factors through an affine open $\text{Spec}(R') \subseteq X'$ such that $R'$ is of the form $R \underset{D/J,\phi^m}{\otimes} D/I$ for some affine open $\text{Spec}(R) \subseteq \nolinebreak X$, and that the corresponding map of rings $D/J \to R$ is standard smooth in the sense of Tag 00T6 of \cite{Sta}.

By Tag 00T7 of \cite{Sta}, there exists a surjection $D[x_1, \dots, x_n]^\wedge \to R$ whose kernel is the ideal $(J, y_1, \dots, y_r)$ such that $\overline{y_1}, \dots, \overline{y_r} \in D/J[x_1, \dots, x_n]$ form a regular sequence. 
Since the image of $J$ in $D/pD$ is a nil ideal, Jacobian criterion for smoothness implies that $\overline{y_1}, \dots, \overline{y_r} \in D/pD[x_1, \dots, x_n]$ form a regular sequence and that the quotient $\widetilde{R} := D[x_1, \dots, x_n]/(p,y_1,\dots,y_r)$ is smooth over $D/pD$. 
By a similar argument to that in the proof of Corollary \ref{cor:reg seq}, we see that the sequence $y^{p^m}_1, \dots, y^{p^m}_r \in D[x_1, \dots, x_n]^\wedge$ is $p$-completely regular relative to $D$. Then we can construct a map $D[x_1, \dots, x_n]^\wedge \to S$ to the $p$-completed PD envelope  with respect to the ideal $(I, y^{p^m}_1, \dots, y^{p^m}_r)$. 
We denote the PD ideal of $S$ by $I_S$ and let $J_S = \text{Ker}(S \twoheadrightarrow S/I_{S} \xrightarrow{\phi^m} S/I_{S})$.  Then we have a map $R \to S/J_S$. Thus, $(S,J_S,I_S,\gamma_S)$ can be regarded as an object of $(X/D)_{m\text{-}\mathrm{crys,new}}$. By definition, $S = D[x_1, \dots, x_n]^\wedge[\frac{y^{kp^m}_1}{k!}, \dots, \frac{y^{kp^m}_r}{k!}]^\wedge_{k \in \mathbb{N}}$ and $S$ is $p$-torsion free.

Let $\widetilde{R}'$ be $\widetilde{R} \underset{D/pD,\phi^m}{\otimes} D/pD$. Let \[D/pD[x_1, \dots, x_n] \to \widetilde{R}'\text{, }\widetilde{R}' \to R'\] be the base  changes of the surjections $D/pD[x_1, \dots, x_n] \to \widetilde{R}$, $\widetilde{R} \to R$ along the map $D/pD \xrightarrow{\phi^m} D/pD$ respectively. Then we have a surjection $D[x_1, \dots, x_n]^\wedge \to R'$ \linebreak determined as the composition of the map $D[x_1, \dots, x_n]^\wedge \to D/pD[x_1, \dots, x_n]$ and the surjection $D/pD[x_1, \dots, x_n] \to \widetilde{R}' \to R'$ constructed above. On the other hand, the kernel of the map $D[x_1, \dots, x_n]^\wedge \to \widetilde{R}'$ is $(p, y_1', \dots, y_r')$, where the $\overline{y_i'}$'s are the \linebreak images of the $\overline{y_i}$'s under the map $D/pD[x_1, \dots, x_n] \to D/pD[x_1, \dots, x_n]$ sending $\sum_{\underline{j}} \alpha_{\underline{j}} x^{\underline{j}}$ to $\sum_{\underline{j}} \alpha^{p^m}_{\underline{j}} x^{\underline{j}}$ (here we write $\underline{j} := (j_1, \dots, j_n)$ for the multi-index).
So the \linebreak kernel of the map $D[x_1, \dots, x_n]^\wedge \to R'$ is $(I, y_1', \dots, y_r')$. As \[\overline{y_1}, \dots, \overline{y_r} \in D/pD[x_1, \dots, x_n]\] form a regular sequence, the sequence $\overline{y_1'}, \dots, \overline{y_r'} \in D/pD[x_1, \dots, x_n]$ is also a regu\nolinebreak lar sequence. 
By a similar argument to that in the proof of Corollary \ref{cor:reg seq}, the sequence \linebreak $y_1', \dots, y_r' \in D[x_1, \dots, x_n]^\wedge$ is $p$-completely regular relative to $D$. Then we can \linebreak construct a map $D[x_1, \dots, x_n]^\wedge \to S'$ to the $p$-completed PD envelope with respect to the ideal $(I, y_1', \dots, y_r')$. 
By definition, $S' = D[x_1, \dots, x_n]^\wedge[\frac{{y'_1}^{k}}{k!}, \dots, \frac{{y'_r}^{k}}{k!}]^\wedge_{k \in \mathbb{N}}$  and $S'$ is $p$-torsion free.

By the relation between $y_i$ and $y_i'$, we have a map of $D$-PD rings $(\phi^m)' : S' \to S$ sending $x_i$ to $x^{p^m}_i$. 
The map $(\phi^m)'$ is $p$-completely faithfully flat: indeed, since $S, S'$ are $p$-torsion free, it suffices to check that $S'/pS' \to S/pS$ is faithfully flat. But this map is the base change of the map $D/pD[x_1, \dots, x_n] \to D/pD[x_1, \dots, x_n]$ sending $x_i$ to $x^{p^m}_i$, which is clearly faithfully flat.

Let $(E',J_{E'},I_{E'},\gamma_{E'}) \in (X'/D)_{\mathrm{crys,new}}$ be as above. We have a map \[f_1 : D[x_1, \dots, x_n]^\wedge \to R' \to E'/I_{E'}.\] As $D[x_1, \dots, x_n]^\wedge$ is the completion of a polynomial ring, one can choose a map \linebreak $f_2 : D[x_1, \dots, x_n]^\wedge \to E'$ lifting $f_1$. 
By the universal property of PD envelope, $f_2$ extends uniquely to a $D$-PD ring map $g : S' \to E'$. 
If we set $h : E' \to E' \widehat{\otimes}_{S'} S$ to be the base change of $(\phi^m)' : S' \to S$ along $g$, then by $p$-complete faithful flatness of $(\phi^m)'$, we see that the same holds true for $h$.

Then we can prove that the map $h$ defines a morphism \[(E',J_{E'},I_{E'},\gamma_{E'}) \to \sigma_{\mathrm{new}}(E' \widehat{\otimes}_{S'} S,J_{E' \widehat{\otimes}_{S'} S},I_{E' \widehat{\otimes}_{S'} S},\gamma_{E' \widehat{\otimes}_{S'} S})\] in $(X'/D)_{\mathrm{crys,new}}$ as in the proof of Proposition \ref{pr:4}, so the proposition follows.
\qed
\end{proof}

$\sigma_{\text{new}}$ also induces equivalences between the categories of crystals:

\begin{defi}
Let 
\begin{align*}
&\mathscr{C}((X/D)_{m\text{-crys,new}})\text{ (resp. }\mathscr{C}^{\mathrm{tors}}((X/D)_{m\text{-crys,new}}), \\
&\mathscr{C}^{p^n\text{-}\mathrm{tors}}((X/D)_{m\text{-crys,new}}),\mathscr{C}^{\mathrm{fp}}((X/D)_{m\text{-crys,new}}))
\end{align*}
 be the category of abelian sheaves $\mathscr{F}$ on $(X/D)_{m\text{-crys,new}}$ such that, for any object $(E,J_E,I_E,\gamma_E)$  in $(X/D)_{m\text{-crys,new}}$,  
\[
\mathscr{F}(E) \in \modd{E}\text{ (resp. }\modtors{E},\modpntors{E},\modfp{E}),
\]
 and for any  morphism \[(E,J_E,I_E,\gamma_E) \to (E_1,J_{E_1},I_{E_1},\gamma_{E_1})\] in $(X/D)_{m\text{-crys,new}}$, $\mathscr{F}(E,J_E,I_E,\gamma_E)  \to \mathscr{F}(E_1,J_{E_1},I_{E_1},\gamma_{E_1})$ is compat\nolinebreak ible with the module structures in the usual sense and 
the canonical map \[\mathscr{F}(E,J_E,I_E,\gamma_E) \widehat{\otimes}_E E_1 \to \mathscr{F}(E_1,J_{E_1},I_{E_1},\gamma_{E_1})\] is an isomorphism of $E_1$-modules. 
\end{defi}

\begin{thm}
\label{thm:sigmanew}
$\sigma_{\mathrm{new}}^*$ induces an equivalence of categories of crystals
\[
\mathscr{C}((X'/D)_{\mathrm{crys,new}}) \to \mathscr{C}((X/D)_{m\text{-}\mathrm{crys,new}}).
\]
The same holds true for $\mathscr{C}^{\mathrm{tors}}, \mathscr{C}^{p^n\text{-}\mathrm{tors}}$ and $\mathscr{C}^{\mathrm{fp}}$.
\end{thm}
\begin{proof}
We can prove the result in the same way as Theorem \ref{th:mcrys}.
\qed
\end{proof}

The following corollary recovers the Frobenius descent, namely, we obtain a site-theoretic proof of the Frobenius descent in our setting.

\begin{cor}
\label{cor:new}
$\sigma^*$ induces an equivalence of categories of crystals
\[
\mathscr{C}^{\mathrm{qcoh}}((X'/D)_{\mathrm{crys}}) \to \mathscr{C}^{\mathrm{qcoh}}((X/D)_{m\text{-}\mathrm{crys}}).
\]
\end{cor}

\begin{proof}
For $n \in \mathbb{N}$, let $\nu_n : (X/D)_{m\text{-crys,new}} \to (X/D)_{m\text{-crys}}$ be the functor of sites that sends an object $(\text{Spf}(E) \gets \text{Spec}(E/J_E) \to X)$ to 
\[
(\text{Spec}(E/p^nE) \gets \text{Spec}((E/p^nE)/(J_E/p^nE)) \to X).
\]
The functor $(X'/D)_{\text{crys,new}} \to (X'/D)_{\text{crys}}$ can be defined in the same way. By abuse of notation, we will denote it also by $\nu_n$. Then one can check that they  are continuous. So we obtain morphisms of topoi
\begin{align*}
&\widetilde{(X/D)}_{m\text{-crys}} \to \widetilde{(X/D)}_{m\text{-crys,new}}, \\
&\widetilde{(X'/D)}_{\text{crys}} \to \widetilde{(X'/D)}_{\text{crys,new}},
\end{align*}
which we denote  by  $\widehat{\nu}_{n}$. Then the $\widehat{\nu}_{n,*}$'s induce functors of categories of crystals
\begin{align*}
&\mathscr{C}^{p^n\text{-tors}}((X/D)_{m\text{-crys}}) \to \mathscr{C}^{p^n\text{-tors}}((X/D)_{m\text{-crys,new}}), \\
&\mathscr{C}^{p^n\text{-tors}}((X'/D)_{\text{crys}}) \to \mathscr{C}^{p^n\text{-tors}}((X'/D)_{\text{crys,new}}).
\end{align*}
Also, we have an equality $\sigma \circ \nu_n = \nu_n \circ \sigma_{\text{new}} $ as morphisms of sites. So we have a commutative diagram
\[
\begin{tikzcd}
\mathscr{C}^{p^n\text{-tors}}((X'/D)_{\text{crys}})\arrow[r,"\sigma^*"]\arrow[d,"\widehat{\nu}_{n,*}"] & \mathscr{C}^{p^n\text{-tors}}((X/D)_{m\text{-crys}})\arrow[d,"\widehat{\nu}_{n,*}"] \\
\mathscr{C}^{p^n\text{-tors}}((X'/D)_{\text{crys,new}})\arrow[r,"\sigma_{\text{new}}^*"] & \mathscr{C}^{p^n\text{-tors}}((X/D)_{m\text{-crys,new}}),
\end{tikzcd} 
\]
where the functors $\sigma^*, \sigma^*_{\text{new}}$ are the pullback functors induced by the morphism of topoi $\sigma, \sigma_{\text{new}}$, which are defined by the cocontinuity of the corresponding functors of sites. As $\sigma^*_{\text{new}}$ is an equivalence by Theorem \ref{thm:sigmanew} and $\widehat{\nu}_{n,*}$'s are equivalences by Lemma \ref{lem:crysnew} below, we conclude that the functor
\[
\sigma^* : \mathscr{C}^{p^n\text{-tors}}((X'/D)_{\text{crys}}) \to \mathscr{C}^{p^n\text{-tors}}((X/D)_{m\text{-crys}})
\]
is an equivalence. By taking the projective limit over $n$ and using the last equivalence in Remark \ref{crysstab}, we obtain the required equivalence
\[
\sigma^* : \mathscr{C}^{\text{qcoh}}((X'/D)_{\text{crys}}) \to \mathscr{C}^{\text{qcoh}}((X/D)_{m\text{-crys}}).
\]
\qed
\end{proof}

\begin{lem}
\label{lem:crysnew}
The functors
\begin{align*}
&\widehat{\nu}_{n,*} : \mathscr{C}^{p^n\text{-}\mathrm{tors}}((X/D)_{m\text{-}\mathrm{crys}}) \to \mathscr{C}^{p^n\text{-}\mathrm{tors}}((X/D)_{m\text{-}\mathrm{crys,new}}) \ (n \in \mathbf{N}), \\
&\widehat{\nu}_{n,*} : \mathscr{C}^{p^n\text{-}\mathrm{tors}}((X'/D)_{\mathrm{crys}}) \to \mathscr{C}^{p^n\text{-}\mathrm{tors}}((X'/D)_{\mathrm{crys,new}}) \ (n \in \mathbf{N})
\end{align*}
in the proof of Corollary \ref{cor:new} are equivalences which are compatible with the natural inclusion functors
\begin{align*}
&\iota_n : \mathscr{C}^{p^n\text{-}\mathrm{tors}}((X/D)_{m\text{-}\mathrm{crys}}) \to \mathscr{C}^{p^{n+1}\text{-}\mathrm{tors}}((X/D)_{m\text{-}\mathrm{crys}}), \\
&\iota_n : \mathscr{C}^{p^n\text{-}\mathrm{tors}}((X/D)_{m\text{-}\mathrm{crys,new}}) \to \mathscr{C}^{p^{n+1}\text{-}\mathrm{tors}}((X/D)_{m\text{-}\mathrm{crys,new}}).
\end{align*}
\end{lem}

For the proof of Lemma \ref{lem:crysnew}, we will need  the notion of stratifications for the new $m$-crystalline site $(X/D)_{m\text{-}\mathrm{crys,new}}$, as in Section 2.

\begin{defi}
\label{defnewstr}
Let $(D,I),J,X$ as above.
\begin{enumerate}
\item We define the category $\strat{m}{(D,I)}$ in the same way as in Definition \ref{defstra}. 
Namely, objects of $\strat{m}{(D,I)}$ are maps $\str : T \to R$ where $T$ is a finite set and Spec$(R) \subseteq X$ is an affine open subscheme such that the induced map $D/J \to R$ is standard smooth in the sense of Tag 00T6 of \cite{Sta}; moreover, $\str$ is a map satisfying the following conditions:
\begin{enumerate}
\item $f_\str : D_\str := D[x_t]^{\wedge}_{t \in T} \to R$ \ \ \ is surjective. \\
$\ \ \ \ \ \ \ \ \ \ \ \ \ \ \ \ \ \ \   x_t \ \ \ \mapsto \str(t)$
\item There exists a sequence $y_1, \dots , y_r \in D[x_t]^{\wedge}_{t \in T}$ such that the kernel of $f_\str$ can be described as the ideal $(J,(y_w)_{w \in W})$ (where $W = \{ 1, \dots , r \}$), and that $\overline{y_1}, \dots , \overline{y_r} \in (D/J)[x_t]_{t \in T}$ form a regular sequence. 
\end{enumerate}
The morphism and the sum $\str \sqcup \str'$ for $\str, \str' \in \strat{m}{(D,I)}$ are defined in the same way as in Definition \ref{defstra}. We shall simply write $\strat{m}{D}$ instead of $\strat{m}{(D,I)}$ if no confusion arises.
\item For $\str : T \to R \in \strat{m}{D}$, we define $S_\str$ in the same way as in the proof of Proposition \ref{prmcrys4} : namely, we define $S_\str = D[x_t]^\wedge[\frac{y^{kp^m}_w}{k!}]^\wedge_{t \in T, w \in W, k \in \mathbb{N}}$ to be the $p$-completed PD envelope of $D[x_t]^\wedge_{t \in T}$ with respect to the ideal $(I,(y^{p^m}_w)_{w \in W})$.
Note that it is independent of the choice of the elements $y_w$ $(w \in W)$. Indeed, the ideal $(I,(y^{p^m}_w)_{w \in W}) = (I,(J, (y_w)_{w \in W})^{(p^m)})$ is independent of the choice. In particular, we see that the construction of $S_\str$ is functorial in $\str \in \strat{m}{D}$.
\item A \textit{stratification} with respect to $\strat{m}{D}$ and  $\mathcal{M}$ (resp. $\mathcal{M}^{p^n\text{-tors}}$) is a pair \[((M_\str)_{\str \in \strat{m}{D}}, (\varphi_{\str \str'})_{\str \to \str'}),\] where $M_\str \in \modd{S_\str}$ (resp. $\mathcal{M}^{p^n\text{-tors}}(S_\str)$) and $\varphi_{\str \str'} : M_\str \widehat{\otimes}_{S_\str} S_{\str'} \xrightarrow{\simeq} M_{\str'}$ is an isomorphism of $S_{\str'}$-modules satisfying the cocycle condition. 
We denote the category of stratifications with respect to $\strat{m}{D}$ and  $\mathcal{M}$ (resp. $\mathcal{M}^{p^n\text{-tors}}$) by $\text{Str}(\strat{m}{D})$ (resp. $\text{Str}^{p^n\text{-tors}}(\strat{m}{D})$).
\end{enumerate}
\end{defi}

We can prove the following proposition in the same way as Proposition \ref{pr:qstra}.

\begin{prop}
There exist equivalences of categories
\begin{align*}
\mathscr{C}((X/D)_{m\text{-}\mathrm{crys,new}}) &\xrightarrow{\simeq} \mathrm{Str}(\strat{m}{D}), \\
\mathscr{C}^{p^n\text{-}\mathrm{tors}}((X/D)_{m\text{-}\mathrm{crys,new}}) &\xrightarrow{\simeq} \mathrm{Str}^{p^n\text{-}\mathrm{tors}}(\strat{m}{D}).
\end{align*}
\end{prop}

Using this, we prove Lemma \ref{lem:crysnew}. \bigskip \\*
\textit{Proof of Lemma \ref{lem:crysnew}}.
 Let $\strat{m}{D}$, $\mathrm{Str}^{p^n\text{-}\mathrm{tors}}(\strat{m}{D})$ be as above. Then the cate\nolinebreak gory $\mathscr{C}^{p^n\text{-}\mathrm{tors}}((X/D)_{m\text{-}\mathrm{crys,new}})$ is equivalent to the category $\mathrm{Str}^{p^n\text{-}\mathrm{tors}}(\strat{m}{D})$. 
It remains to show that there  is a natural  equivalence of categories \[\mathscr{C}^{p^n\text{-}\mathrm{tors}}((X/D)_{m\text{-}\mathrm{crys}}) \xrightarrow{\simeq} \mathrm{Str}^{p^n\text{-}\mathrm{tors}}(\strat{m}{D}).\]

We first define a functor $\mathscr{C}^{p^n\text{-}\mathrm{tors}}((X/D)_{m\text{-}\mathrm{crys}}) \to \mathrm{Str}^{p^n\text{-}\mathrm{tors}}(\strat{m}{D})$. 
For an object $(\str : T \to R) \in \strat{m}{D}$, we can construct the $p$-completed PD envelope $S_\str$ as in Definition \ref{defnewstr}. Then $S_\str/p^nS_\str$ can be regarded as an object of $(X/D)_{m\text{-}\mathrm{crys}}$. Thus, we can define the functor by
\[
\mathscr{F} \mapsto ((\mathscr{F}(S_\str/p^nS_\str))_{\str \in \strat{m}{D}}, (\mathscr{F}(S_\str/p^nS_\str) \otimes_{S_\str} S_{\str'} \xrightarrow{\simeq} \mathscr{F}(S_{\str'}/p^nS_{\str'}))_{\str \to \str'}).
\]

Also, we can define the functor $\mathrm{Str}^{p^n\text{-}\mathrm{tors}}(\strat{m}{D}) \to \mathscr{C}^{p^n\text{-}\mathrm{tors}}((X/D)_{m\text{-}\mathrm{crys}})$  as \linebreak follows. 
Given \[((M_\str)_{\str \in \strat{m}{D}}, (\varphi_{\str \str'})_{\str \to \str'}) \in \mathrm{Str}^{p^n\text{-}\mathrm{tors}}(\strat{m}{D})\] and $(E,J_E,I_E,\gamma_E) \in (X/D)_{m\text{-}\mathrm{crys}}$, we choose an affine open Spec$(R) \subseteq X$ such that Spec$(E/J_E) \to X$ factors through Spec$(R)$. 
As in the proof of Proposition \ref{pr:qstra}, there exists an open cover of $\text{Spec}(R)$ by standard opens $\text{Spec}(R_{g_i})$ such that each $R_{g_i}$ is standard smooth over $D/J$. 
Then the fiber product $\text{Spec}(R_{g_i} \otimes_{R} E/J_E)$ is an affine open of $\text{Spec}(E/J_E)$. We denote the scheme $\text{Spec}(R_{g_i} \otimes_{R} E/J_E)$ by $\text{Spec}(\overline{E_i})$. By Lemma \ref{mcryslift} there is a unique $m$-PD ring $(E_i,J_EE_i, I_EE_i,\gamma_{E_i}) \in (X/D)_{m\text{-}\mathrm{crys}}$ for which the corresponding scheme $\text{Spec}(E_i)$ is an affine open subscheme of $\text{Spec}(E)$ and lifts $\text{Spec}(\overline{E_i})$, namely,  $E_i/J_EE_i = \overline{E_i}$.

First, we define $\mathscr{F}(E_i)$ for such $E_i$. By construction, there exists an object in $\strat{m}{D}$ of the form $\str : T \to R_{g_i}$. Then we can construct a morphism $S_\str \to E_i$ as in the proof of Proposition \ref{pr:qstra}. We define
\[
\mathscr{F}(E_i) := M_\str \otimes_{S_\str} E_i.
\]
We can also define $\mathscr{F}(E)$ for general $(E,J_E,I_E,\gamma_E) \in (X/D)_{m\text{-}\mathrm{crys}}$ as in the proof of Proposition \ref{pr:qstra}.

We can prove that the presheaf $\mathscr{F}$ is well-defined and that it defines an object of $\mathscr{C}^{p^n\text{-}\mathrm{tors}}((X/D)_{m\text{-}\mathrm{crys}})$ as in the proof of Proposition \ref{pr:qstra}. So the functor \[\mathrm{Str}^{p^n\text{-}\mathrm{tors}}(\strat{m}{D}) \to \mathscr{C}^{p^n\text{-}\mathrm{tors}}((X/D)_{m\text{-}\mathrm{crys}})\] can be defined by
\[
((M_\str)_{\str \in \strat{m}{D}}, (\varphi_{\str \str'})_{\str \to \str'}) \mapsto \mathscr{F}.
\]
The two functors we constructed are quasi-inverse to each other. Hence, the category $\mathscr{C}^{p^n\text{-}\mathrm{tors}}((X/D)_{m\text{-}\mathrm{crys}})$ is equivalent to the category $\mathrm{Str}^{p^n\text{-}\mathrm{tors}}(\strat{m}{D})$, as desired.
\qed \bigskip

Finally, we compare the $m$-$q$-crystalline site with our variant of the $m$-crystalline site.  Assume that $(D,I)$ is a $q$-PD pair with $q=1$ in $D$, i.e., a derived $p$-complete $\delta$-pair over $\mathbf{Z}_p$ satisfying the following conditions:

\begin{enumerate}
\item For any $f \in I$, $f^p \in pI$.
\item The pair $(D,(p))$ is a bounded prism, i.e., $D$ is $p$-torsion free. In particular, $D$ is classically $p$-complete.
\item $D/I$ is classically $p$-complete.
\end{enumerate}
By Remark 16.3 of \cite{BS19}, there exists a canonical PD structure $\gamma$ on $I$. Now suppose further that $p \in I$. Let $J = (\phi^m)^{-1}(I)$. Then we have the following result.

\begin{lem}
With the notation above,
\[
J^{(p^m)}+pJ \subseteq I \subseteq J.
\]
Here, $J^{(p^m)}$ denotes the ideal of $D$ generated by $x^{p^m}$ for all elements $x$ of $J$. In particular, $(J,I,\gamma)$ is an $m$-PD ideal.
\end{lem}
\begin{proof}
Set $J_i = (\phi^i)^{-1}(I)$. Then  for all $i$, we have $p \in J_i$. For any $x \in J_i$, we see that $x^p = \phi(x)-p\delta(x) \in J_{i-1}$. So for any $x \in J$, one has $x^{p^m} \in I$. This implies that $J^{(p^m)}+pJ \subseteq I$. On the other hand, for any $x \in I$, we have $\phi^m(x) = p\phi^{m-1}(\frac{\phi(x)}{p}) \in I$. The result follows.
\qed
\end{proof}

In this situation, there is a functor of sites from the $m$-$q$-crystalline site to the new $m$-crystalline site by forgetting the $\delta$-structure on the $q$-PD pair.

\begin{con}
Assume that $(D,I)$ is a $q$-PD pair with $q=1$ in $D$. Suppose that the ideal $I$ contains $p$. Let $X$ be a scheme smooth and separated over $D/J$. Then we have the functor of sites \[\tau : (X/D)_{m\text{-}q\text{-crys}} \to (X/D)_{m\text{-crys,new}}\] that sends an object $(E,I_E)$ of the $m$-$q$-crystalline site $(X/D)_{m\text{-}q\text{-crys}}$ to an object $(E,J_E,I_E,\gamma_E)$ of the new $m$-crystalline site $(X/D)_{m\text{-crys,new}}$, where $J_E = (\phi^m)^{-1}(I_E)$, and $\gamma_E$ is the canonical PD structure on $I_E$ defined by Remark 16.3 of \cite{BS19}.
\end{con}

\begin{prop}
\label{pr:taucon}
Assume that $(D,I)$ is a $q$-PD pair with $q=1$ in $D$. Suppose that the ideal $I$ contains $p$.  Then the functor $\tau$ is continuous.
\end{prop}
\begin{proof}
Note that if $q=1$ in $D$, then the $(p,[p]_q)$-completion is the same as the $p$-completion. So we can prove the result in the same way as Proposition \ref{pr:2}.
\qed
\end{proof}

So we obtain a morphism of topoi
\[
\widetilde{(X/D)}_{m\text{-crys,new}} \to \widetilde{(X/D)}_{m\text{-}q\text{-crys}}.
\]
We denote it by $\widehat{\tau}$. Then the following holds true.

\begin{prop}
\label{prtaucrys}
Assume that $(D,I)$ is a $q$-PD pair with $q=1$ in $D$. Suppose that the ideal $I$ contains $p$. Then $\widehat{\tau}_*$ induces an equivalence of categories of crystals
\[
\mathscr{C}((X/D)_{m\text{-}\mathrm{crys,new}}) \xrightarrow{\simeq} \mathscr{C}((X/D)_{m\text{-}q\text{-}\mathrm{crys}}).
\]
\end{prop}
\begin{proof}
Let $\stra{m}{D}, \mathrm{Str}(\stra{m}{D})$ be the categories in Definition \ref{defstra}. Then the category $\mathscr{C}((X/D)_{m\text{-}q\text{-}\mathrm{crys}})$ is equivalent to the category $\mathrm{Str}(\stra{m}{D})$. 
It remains to show that there is a natural equivalence of categories \[\mathscr{C}((X/D)_{m\text{-}\mathrm{crys,new}}) \xrightarrow{\simeq} \mathrm{Str}(\stra{m}{D}).\]

We first define a functor \[\mathscr{C}((X/D)_{m\text{-}\mathrm{crys,new}}) \to \mathrm{Str}(\stra{m}{D}).\] For an object $(\str : T \to R) \in \stra{m}{D}$, we can construct the $q$-PD envelope $S_\str$ as in Definition \ref{defstra}. We may also regard $S_\str$ as an object of $(X/D)_{m\text{-}\mathrm{crys,new}}$. Thus, we can define the  functor by
\[
\mathscr{F} \mapsto ((\mathscr{F}(S_\str))_{\str \in \stra{m}{D}}, (\mathscr{F}(S_\str) \widehat{\otimes}_{S_\str} S_{\str'} \xrightarrow{\simeq} \mathscr{F}(S_{\str'}))_{\str \to \str'}).
\]

Also, we can define the functor $\mathrm{Str}(\stra{m}{D}) \to \mathscr{C}((X/D)_{m\text{-}\mathrm{crys,new}})$ as follows. 
Given $((M_\str)_{\str \in \stra{m}{D}}, (\varphi_{\str \str'})_{\str \to \str'}) \in \mathrm{Str}(\stra{m}{D})$ and $(E,J_E,I_E,\gamma_E) \in (X/D)_{m\text{-}\mathrm{crys,new}}$, we choose an affine open Spec$(R) \subseteq X$ such that Spec$(E/J_E) \to X$ factors through Spec$(R)$. As in the proof of Proposition \ref{pr:qstra}, there exists an open cover of $\text{Spec}(R)$ by standard opens $\text{Spec}(R_{g_i})$ such that each $R_{g_i}$ is standard smooth over $D/J$. 
Then the fiber product $\text{Spec}(R_{g_i} \otimes_{R} E/J_E)$ is an affine open of $\text{Spec}(E/J_E)$. We denote the scheme $\text{Spec}(R_{g_i} \otimes_{R} E/J_E)$ by $\text{Spec}(\overline{E_i})$. By Lemma \ref{newlift} there is a unique maximal $m$-PD ring $(E_i,J_{E_i}, I_EE_i,\gamma_{E_i}) \in (X/D)_{m\text{-}\mathrm{crys,new}}$ for which the corresponding formal scheme $\text{Spf}(E_i)$ is an affine open formal subscheme of $\text{Spf}(E)$ and lifts $\text{Spec}(\overline{E_i})$, namely, $E_i/J_EE_i = \overline{E_i}$.

First we define $\mathscr{F}(E_i)$ for such $E_i$. By construction, there exists an object in $\stra{m}{D}$ of the form $\str : T \to R_{g_i}$ . As $D\{ x_t \}^\wedge_{t \in T}$ is the completion of a polynomial ring by Lemma 2.11 of \cite{BS19}, one can choose a map $f_1 : D\{ x_t \}^\wedge_{t \in T} \to R_{g_i} \to E_i/J_{E_i}$ and a map $f_2 : D\{ x_t \}^\wedge_{t \in T} \to E_i$ lifting $f_1$. 
On the other hand,  we have 
\[S_\str = D\{x_t, \frac{\phi^{m+1}(y_w)}{p}\}^\wedge_{t \in T, w \in W}
\] by Lemma 16.10 of \cite{BS19}. By assumption, this is the $p$-completed PD envelope of $D\{ x_t \}^\wedge_{t \in T}$ with respect to the ideal $(I,(\phi^m(y_w))_{w \in W}) = (I,(y^{p^m}_w)_{w \in W})$. So $f_2$ extends uniquely to a $D$-PD ring map $f_3 : S_\str \to E_i$. We define
\[
\mathscr{F}(E_i) :=  M_\str \widehat{\otimes}_{S_\str} E_i \in \modd{E_i}.
\]
We can also define $\mathscr{F}(E)$ for general $(E,J_E,I_E,\gamma_E) \in (X/D)_{m\text{-}\mathrm{crys,new}}$ as in the proof of Proposition \ref{pr:qstra}.

We can prove that the presheaf $\mathscr{F}$ is well-defined and that it defines an object of $\mathscr{C}((X/D)_{m\text{-}\mathrm{crys,new}})$ as in the proof of Proposition \ref{pr:qstra}. So the  functor
\[
\mathrm{Str}(\stra{m}{D}) \to \mathscr{C}((X/D)_{m\text{-}\mathrm{crys,new}})
\]
can be defined by
\[
((M_\str)_{\str \in \stra{m}{D}}, (\varphi_{\str \str'})_{\str \to \str'}) \mapsto \mathscr{F}.
\]
The two functors we constructed are quasi-inverse to each other. Hence, the category $\mathscr{C}((X/D)_{m\text{-}\mathrm{crys,new}})$ is equivalent to the category $\mathrm{Str}(\stra{m}{D})$, as desired.
\qed
\end{proof}

We have a commutative diagram:
\[
\begin{tikzcd}
(X/D)_{m\text{-}q\text{-crys}} \arrow[d,"\rho"]\arrow[r,"\tau"] & (X/D)_{m\text{-crys,new}}\arrow[d,"\sigma_{\text{new}}"] \\
(X'/D)_{q\text{-crys}}\arrow[r,"\tau"] &  (X'/D)_{\text{crys,new}}.
\end{tikzcd} 
\]
On the other hand, we have an inclusion $\mathscr{C}((X/D)_{m\text{-}q\text{-crys}}) \subseteq \mathscr{C}_\prism((X/D)_{m\text{-}q\text{-crys}})$. So the above square induces the following diagram:
\[
\begin{tikzcd}
\mathscr{C}_\prism((X/D)_{m\text{-}q\text{-crys}}) \arrow[r,hookleftarrow]&\mathscr{C}((X/D)_{m\text{-}q\text{-crys}})  & \mathscr{C}((X/D)_{m\text{-crys,new}}) \arrow[l,"\widehat{\tau}_*","\simeq"']\\
\mathscr{C}_\prism((X'/D)_{q\text{-crys}}) \arrow[r,hookleftarrow]\arrow[u,"\simeq","\text{C}^*"']&\mathscr{C}((X'/D)_{q\text{-crys}}) \arrow[u,"\simeq"] &  \mathscr{C}((X'/D)_{\text{crys,new}}) \arrow[l,"\widehat{\tau}_*","\simeq"']\arrow[u,"\simeq","\sigma_{\text{new}}^*"'].
\end{tikzcd} 
\]
In this way the equivalence between the category of crystals on the $m$-$q$-crystalline site and that on the usual $q$-crystalline site in Section 2 is compatible with the Frobenius descent.

\section{Relation to the results of Xu, Gros-Le Stum-Quir\'{o}s and Morrow-Tsuji}

In this section, we discuss relations between our equivalences in Sections 1, 2 and the results of Xu \cite{Xu19}, Gros-Le Stum-Quir\'{o}s \cite{GSQ20a}\cite{GSQ20b}\cite{GSQ18} and Morrow-Tsuji \cite{MT20}.

First, we establish a relation between our results and the results of Xu. Let $k$ be a perfect field of characteristic $p$, and let $W$ be the Witt ring of $k$. We consider the diagram:
\[
\begin{tikzcd}
\text{Spf}(W) \arrow[r,hookleftarrow] & \text{Spec}(W/pW)=\text{Spec}(k) \arrow[dr,phantom, "\Box"] & X\arrow[l, "f"']\\
\text{Spf}(W) \arrow[u, "\phi^*"]\arrow[r,hookleftarrow] & \text{Spec}(k) \arrow[u, "\phi^*"] & X',\arrow[l]\arrow[u]
\end{tikzcd} 
\]
where the right square is the Cartesian diagram, $\phi^*$ is the morphism induced by the lift $\phi : W \to W$ of Frobenius, and $f$ is a smooth and separated map.

Let us now briefly recall some notations and results in \cite{Xu19}. In \cite{Xu19}, Xu defined the category $\mathscr{E}'$ (resp. $\underline{\mathscr{E}}$) as follows: Objects are diagrams 
\[(\mathcal{T} \gets T \to U) \text{ (resp. }(\mathcal{T} \gets \underline{T} \to U)),\] where $\mathcal{T}$ is a flat $p$-adic formal $W$-scheme, $T \to \mathcal{T}$ (resp. $\underline{T} \to \mathcal{T}$) is the closed immersion defined by the ideal $p\mathcal{O}_{\mathcal{T}}$ (resp. the ideal Ker$(\mathcal{O}_{\mathcal{T}} \to \mathcal{O}_{\mathcal{T}}/p\mathcal{O}_{\mathcal{T}} \xrightarrow{\phi} \mathcal{O}_{\mathcal{T}}/p\mathcal{O}_{\mathcal{T}})$ with $\phi$ the Frobenius) 
and $T \to U$ is an affine morphism over $k$ to an open subscheme $U$ of $X'$ (resp. $X$). The notion of morphism is the obvious one.
We endow $\mathscr{E}', \underline{\mathscr{E}}$ with the topology induced by the fppf covers of $\mathcal{T}$ (see 7.13 of \cite{Xu19}) and denote these sites by $\mathscr{E}'_{\rm fppf}, \underline{\mathscr{E}}_{\rm fppf}$ respectively (Xu also considers  the topology induced by the Zariski topology on $\mathcal{T}$ in 7.9 of \cite{Xu19} that we will not consider here).

We define the categories of crystals by
\[
\mathscr{C}^{\text{tors}}(\mathscr{E}'_{\rm fppf}) = \displaystyle\bigcup_{n}\mathscr{C}^{p^n\text{-tors}}(\mathscr{E}'_{\rm fppf}), \ \ \mathscr{C}^{\text{tors}}( \underline{\mathscr{E}}_{\rm fppf}) = \displaystyle\bigcup_{n}\mathscr{C}^{p^n\text{-tors}}( \underline{\mathscr{E}}_{\rm fppf}),
\]
where $\mathscr{C}^{p^n\text{-tors}}(\mathscr{E}'_{\rm fppf}), \mathscr{C}^{p^n\text{-tors}}( \underline{\mathscr{E}}_{\rm fppf})$ are the categories of $p^n$-torsion quasi-coherent crystals on $\mathscr{E}'_{\rm fppf}, \underline{\mathscr{E}}_{\rm fppf}$ respectively. 
(The categories $\mathscr{C}^{p^n\text{-tors}}(\mathscr{E}'_{\rm fppf}), \mathscr{C}^{p^n\text{-tors}}( \underline{\mathscr{E}}_{\rm fppf})$ are denoted by $\mathscr{C}^{\text{qcoh}}_{\text{fppf}}(\mathcal{O}_{\mathscr{E}' , n}), \mathscr{C}^{\text{qcoh}}_{\text{fppf}}(\mathcal{O}_{\underline{\mathscr{E}} , n})$ respectively in \cite{Xu19}.)

In 9.1 of \cite{Xu19}, Xu defined a functor $\rho : \underline{\mathscr{E}} \to \mathscr{E}'$ that sends $(\mathcal{T} \gets \underline{T} \xrightarrow{f} U)$ to $(\mathcal{T} \gets T \to U')$, where the right map $f: T \to U'$ is defined as follows:
\[
T \xrightarrow{g} \underline{T} \underset{\text{Spec}(k),\phi^*}{\times} \text{Spec}(k) \to U \underset{\text{Spec}(k),\phi^*}{\times} \text{Spec}(k) =: U'.
\]
Here the first map $g$ is induced by the map of sheaves of rings
\[
\mathcal{O}_{\underline{T}} \underset{k,\phi}{\otimes} k \to \mathcal{O}_T, \ \ t \otimes a \mapsto \phi(t)a.
\]
Then, in Theorem 9.2 of \cite{Xu19}, he proved that the functor $\rho$ induces an equivalence of topoi
\[
\rho : \widetilde{\underline{\mathscr{E}}}_{\rm fppf} \xrightarrow{\simeq} \widetilde{\mathscr{E}}'_{\rm fppf},
\]
and in Theorem 9.12 of \cite{Xu19}, he proved that $\rho$ induces an equivalence between the categories of crystals
\[
\rho^* : \mathscr{C}^{p^n\text{-tors}}(\mathscr{E}'_{\rm fppf}) \xrightarrow{\simeq} \mathscr{C}^{p^n\text{-tors}}( \underline{\mathscr{E}}_{\rm fppf}).
\]

We relate our result in Section 1 to his result. For a pair $(E,pE)$ over $(W,pW)$,  we define
\[
\underline{E/pE} = E/\text{Ker}(E \to E/pE \xrightarrow{\phi} E/pE)
\]
(this is the ring-theoretic version of $\underline{T}$ above). Note that if $(E,I_E) = (E,pE)$ is a $\delta$-pair over $(W,pW)$, then 
\[
E/\text{Ker}(E \to E/pE \xrightarrow{\phi} E/pE) = E/\text{Ker}(E \xrightarrow{\phi} E \to E/pE) = E/J_E,
\] where $J_E$ is the ideal considered in the level 1-prismatic site, i.e., $J_E = \phi^{-1}(I_E)$. Then we have a commutative diagram:
\[
\begin{tikzcd}
(X/W)_{1-\prism}  \arrow[r, "\beta_1"]\arrow[d, "\rho'"]& \underline{\mathscr{E}}_{\rm fppf} \arrow[d, "\rho"]\\
(X'/W)_{\prism} \arrow[r, "\beta_0"]& \mathscr{E}'_{\rm fppf},
\end{tikzcd}
\]
where $\rho$ is the functor defined by Xu, $\rho'$ is the functor $\rho$ defined in Section 1 and $\beta_0,\beta_1$ are the functors defined as:
\begin{align*}
&\beta_0 : (\text{Spf}(E) \hookleftarrow \text{Spec}(E/pE) \to X' ) \mapsto (\text{Spf}(E) \hookleftarrow \text{Spec}(E/pE) \to X' ), \\
&\beta_1 : (\text{Spf}(E) \hookleftarrow \text{Spec}(E/J_E) \to X ) \mapsto (\text{Spf}(E) \hookleftarrow \text{Spec}(E/J_E = \underline{E/pE}) \to X ).
\end{align*}
It is easy to see that the functors $\beta_0, \beta_1$ are continuous. So the functors above induce a commutative diagram:
\[
\begin{tikzcd}
\mathscr{C}^{p^n\text{-tors}}((X/W)_{1-\prism}) & \mathscr{C}^{p^n\text{-tors}}( \underline{\mathscr{E}}_{\rm fppf})\arrow[l,"\widehat{\beta}_{1,*}"] \\
\mathscr{C}^{p^n\text{-tors}}((X'/W)_{\prism})\arrow[u,"\simeq","\rho'^*"'] & \mathscr{C}^{p^n\text{-tors}}(\mathscr{E}'_{\rm fppf}). \arrow[l,"\widehat{\beta}_{0,*}"]\arrow[u,"\simeq","\rho^*"'] 
\end{tikzcd} 
\]
So our equivalence $\rho'^*$ is compatible with Xu's equivalence $\rho^*$.

If $X/k$ lifts to $\widetilde{X}/W$ and the relative Frobenius of $X$ lifts to $\widetilde{X}$, then we have the following commutative diagram: 
{\fontsize{9pt}{10pt}\selectfont
\[
\begin{tikzcd}
\text{MIC}(\widetilde{X}/W)^{\text{qn}} & \mathscr{C}^{\text{tors}}((X/W)_{\text{crys}}) \arrow[l,"\simeq","\psi"'] & \mathscr{C}^{\text{tors}}((X/W)_{1-\prism}) \arrow[l,"\simeq","\lambda"'] & \mathscr{C}^{\text{tors}}( \underline{\mathscr{E}}_{\rm fppf})\arrow[l,"\widehat{\beta}_{1,*}"] \\
p\text{-MIC}(\widetilde{X}'/W)^{\text{qn}} \arrow[u,"\simeq","\varphi"'] && \mathscr{C}^{\text{tors}}((X'/W)_{\prism})\arrow[ll,dashed,"\simeq","\varpi"']\arrow[u,"\simeq","\rho'^*"'] & \mathscr{C}^{\text{tors}}(\mathscr{E}'_{\rm fppf}). \arrow[l,"\widehat{\beta}_{0,*}"]\arrow[u,"\simeq","\rho^*"'] 
\end{tikzcd} 
\]
}

\noindent
Here, $\text{MIC}(\widetilde{X}/W)$ (resp. $p\text{-MIC}(\widetilde{X}'/W)$) denote the category of $p$-power torsion quasi-coherent $\mathscr{O}_X$-modules with quasi-nilpotent  integrable connection (resp. $p$-connection) relative to $W$. The functor $\lambda$ is the composition:
{\fontsize{9pt}{10pt}\selectfont
\[
\mathscr{C}^{\mathrm{tors}}((X/W)_{1-\prism}) \to \mathscr{C}^{\mathrm{tors}}((X/W)_{q\text{-}\mathrm{crys}})  \to \mathscr{C}^{\mathrm{tors}}((X/W)_{\mathrm{crys,new}}) \to  \mathscr{C}^{\mathrm{tors}}((X/W)_{\text{crys}}),
\]
}
where the first functor is the equivalence
\[
\alpha^* : \mathscr{C}^{\mathrm{tors}}((X/W)_{1-\prism}) \xrightarrow{\simeq} \mathscr{C}^{\mathrm{tors}}((X/W)_{q\text{-}\mathrm{crys}})
\]
in Theorem \ref{th:mm1}, the second functor is the inverse of the functor (note that $(W, pW)$ is a $q$-PD pair with $q=1$ in $W$)
\[
\widehat{\tau}_* : \mathscr{C}^{\mathrm{tors}}((X/W)_{\mathrm{crys,new}}) \xrightarrow{\simeq} \mathscr{C}^{\mathrm{tors}}((X/W)_{q\text{-}\mathrm{crys}})
\]
in Proposition \ref{prtaucrys} and the third functor is the inverse of the functor 
\[
\widehat{\nu}_* : \mathscr{C}^{\mathrm{tors}}((X/W)_{\text{crys}}) \xrightarrow{\simeq} \mathscr{C}^{\mathrm{tors}}((X/W)_{\text{crys,new}})
\]
 in Lemma \ref{lem:crysnew}. The functor $\varphi$ is the equivalence constructed in Proposition 2.5 of \cite{Shi12}. The functor $\psi$ is a natural equivalence in the theory of crystalline site. By composing the equivalences in the above commutative diagram, we obtain an \linebreak equivalence $\varpi$ from the category of crystals on the prismatic site to the category of modules with integrable $p$-connection. We note that the equivalence $\varpi$ has been established in more general situation by Ogus.

\bigskip

Next, we establish a relation between the functors we constructed and the twisted Simpson correspondence by Gros-Le Stum-Quir\'{o}s. Let $(D,I)$ be a $q$-PD pair with $I = \phi^{-1}([p]_qD)$. Assume that there exists a pushout diagram
\[
\begin{tikzcd}
D \arrow[r]\arrow[d,"\phi"] & D[x]^\wedge \arrow[r,"f"] & A\arrow[d] \\
D \arrow[rr] && A'\arrow[ul, phantom, "\lrcorner", very near start],
\end{tikzcd}
\]
where the map $f$ is $(p,I)$-completely \'{e}tale. We give a $\delta$-structure on $D[x]^\wedge$ by $\delta(x) = 0$.  By Lemma 2.18 of \cite{BS19}, The ring $A$ admits a unique $\delta$-structure compatible with the one on $D[x]^\wedge$. Set 
\begin{align*}
\overline{A} &= A \underset{D}{\otimes} D/I, \\
\overline{A}' &= A' \underset{D}{\otimes} D/[p]_qD.
\end{align*}
Then by Theorem 4.8 and Proposition 6.9 of \cite{GSQ20b}, we have the following \linebreak commutative diagram:
\[
\begin{tikzcd}
\mathscr{C}_{\prism}((\overline{A}/D)_{1-\prism}) \arrow[r,"\simeq","\widehat{\alpha}_*"'] & \mathscr{C}_{\prism}((\overline{A}/D)_{q\text{-crys}}) \arrow[r,"\text{G}"'] & \widehat{\text{Strat}}^{(0)}_q(A/D) \\
\mathscr{C}_{\prism}((\overline{A}'/D)_{\prism}) \arrow[u,"\simeq","\rho^*"']\arrow[rr,"\text{H}"'] && \widehat{\text{Strat}}^{(-1)}_q(A'/D), \arrow[u,"\simeq","\text{F}^*"']
\end{tikzcd}
\]
where $\widehat{\text{Strat}}^{(0)}_q(A/D)$ (resp. $\widehat{\text{Strat}}^{(-1)}_q(A'/D)$) is the category of twisted hyper-stratified $A$-modules of level 0 (resp. $A'$-modules of level $(-1)$) defined in Definition 3.9 of \cite{GSQ20b}, $\widehat{\alpha}_*$ and $\rho^*$ are the functors we constructed,  G and H  are the functors that appeared in Proposition 6.9 of \cite{GSQ20b}, and $\text{F}^*$ is an equivalence in Theorem 4.8 of \cite{GSQ20b}. 
(More precisely, it is not clear to us which category of modules they used in their definition of $\widehat{\text{Strat}}^{(0)}_q(A/D)$ and $\widehat{\text{Strat}}^{(-1)}_q(A'/D)$. We guess that it would be reasonable to use the category $\mathcal{M}_\prism(A)$ and $\mathcal{M}_\prism(A')$ respectively. So the categories of twisted hyper-stratified modules would be related to our equivalences of categories of crystals with notation $\mathscr{C}_\prism$.) 

Thus, we see that our equivalences of categories fit into the diagram in  Proposition 6.9 of \cite{GSQ20b}. In particular, our argument gives a direct proof of the equivalence $\widehat{\alpha}_* \circ \rho^*$, which they plan to prove indirectly in forthcoming work by showing that G and H are equivalences (See Remark 2 after Proposition 6.9 of \cite{GSQ20b}).
 Moreover, our proof of the equivalence $\widehat{\alpha}_* \circ \rho^*$, which is in the style of \cite{Oya17}, \cite{Xu19}, answers `the hope' in Remark 3 after Proposition 6.9 of \cite{GSQ20b} to some extent.

\bigskip

Finally, we establish a relation between our results and the results of Morrow-Tsuji in \cite{MT20}. Let us briefly recall some  notations and results in \cite{MT20}. 
Let $\mathcal{O}$ be a ring of integers of a characteristic 0 perfectoid field containing all $p$-power roots of unity, and let $A_{\text{inf}} := W(\mathcal{O}^\flat), \epsilon := (1,\zeta_p,\zeta_{p^2},\dots) \in \mathcal{O}^\flat,\mu := [\epsilon] -1, \xi := \frac{\mu}{\phi^{-1}(\mu)}, \tilde{\xi} := \phi(\xi),$ where $\phi$ is the Frobenius on $A_{\text{inf}}$ ($\phi$ is denoted by $\varphi$ in \cite{MT20}).  Then the pair ($A_{\text{inf}},(\tilde{\xi})$) is a bounded prism.

Let $R$ be a smooth $\mathcal{O}$-algebra with a formally \'{e}tale map $\mathcal{O} [ T^{\pm1}_1,\dots,T^{\pm1}_d ]^\wedge \to R$ called a framing. 
We can use the framing to define the ring $R_{\infty}$ with the action \[ \text{Gal}(R_{\infty}/R) =: \Gamma \cong \mathbb{Z}^d_p,\] and the rings $A^\Box := A^\Box_{\rm inf}(R)$, $A^\Box_{\infty} := A_{\rm inf}(R_{\infty})$. For an $A_{\rm inf}$-algebra $B$, we set \[B^{(1)} := A_{\rm inf} \otimes_{\phi,A_{\rm inf}} B.\] We denote the relative Frobenius $B^{(1)} \to B$ by $F$. 
We use the same notation for the base change of an $\mathcal{O}$-algebra along the Frobenius because we can regard any $\mathcal{O}$-algebra  as an $A_{\rm inf}$-algebra.

Let $\text{Rep}^\mu_\Gamma(A^\Box)$ (resp. $\text{Rep}^\mu_\Gamma(A^{\Box(1)})$) be the category of generalized representations \linebreak of $\Gamma$ over $A^\Box$ (resp. $A^{\Box(1)}$) which is trivial modulo $\mu$, and let $\text{qMIC}(A^{\Box})$ (resp. $\text{qHIG}(A^{\Box(1)})$) be the category of finite projective $A^\Box$-modules (resp. $A^{\Box(1)}$-modules) with flat $q$-connection (resp. flat $q$-Higgs field). 
We shall write $\text{qMIC}(A^{\Box})^{\rm qnilp}$  (resp. \linebreak $\text{qHIG}(A^{\Box(1)})^{\rm qnilp}$) for the full subcategory of $\text{qMIC}(A^{\Box})$ \hspace{-3pt} (resp. $\text{qHIG}(A^{\Box(1)})$)  consisting \linebreak of $(p,\mu)$-adically quasi-nilpotent objects (this notation is not used in \cite{MT20}).

Then, in Section 3 of \cite{MT20}, they considered the following commutative diagram
\[
\begin{tikzcd}
\mathscr{C}^{\text{fp}}((\text{Spf}(R^{(1)})/A_{\text{inf}})_{\prism}) \arrow[r,"\text{ev}_{A^{\Box(1)}}"]\arrow[rd,"\text{ev}^\phi_{A^{\Box(1)}}"'] & \text{Rep}^\mu_\Gamma(A^{\Box(1)}) \arrow[r,"\simeq"]\arrow[d,"- \otimes_{A^{\Box(1)},F} A^\Box"] & \text{qHIG}(A^{\Box(1)}) \arrow[d] \\
&\text{Rep}^\mu_\Gamma(A^\Box) \arrow[r,"\simeq"] & \text{qMIC}(A^{\Box}),
\end{tikzcd}
\]
where $\text{ev}_{A^{\Box(1)}}$ is the functor defined by evaluation on \[(\text{Spf}(A^{\Box(1)}) \gets \text{Spf}(R^{(1)}) \xrightarrow{=} \text{Spf}(R^{(1)})),\] and $\text{ev}^\phi_{A^{\Box(1)}}$ is the composition of the functor $\text{ev}_{A^{\Box(1)}}$ with $- \otimes_{A^{\Box(1)},F} A^\Box$. Moreover, the arrows with $\simeq$ are equivalences.

By our result in Section 1, we see that the above diagram extends to the  following diagram:
\[
\begin{tikzcd}
\mathscr{C}^{\text{fp}}((\text{Spf}(R^{(1)})/A_{\text{inf}})_{\prism}) \arrow[r,"\text{ev}_{A^{\Box(1)}}"]\arrow[d,"\simeq"] & \text{Rep}^\mu_\Gamma(A^{\Box(1)}) \arrow[r,"\simeq"]\arrow[d,"- \otimes_{A^{\Box(1)},F} A^\Box"] & \text{qHIG}(A^{\Box(1)}) \arrow[d] \\
\mathscr{C}^{\text{fp}}((\text{Spf}(R)/A_{\text{inf}})_{1-\prism}) \arrow[r,"\text{ev}_{A^{\Box}}"] &\text{Rep}^\mu_\Gamma(A^\Box) \arrow[r,"\simeq"] & \text{qMIC}(A^{\Box}),
\end{tikzcd}
\]
where $\text{ev}_{A^{\Box}}$ is the functor defined by evaluation on $(\text{Spf}(A^{\Box}) \gets \text{Spf}(R) \xrightarrow{=} \text{Spf}(R))$.

In Theorem 3.2 of \cite{MT20}, they proved that the functors $\text{ev}_{A^{\Box(1)}},\text{ev}^\phi_{A^{\Box(1)}}$ are fully faithful and  induce equivalences
\begin{align*}
&\mathscr{C}^{\text{fp}}((\text{Spf}(R^{(1)})/A_{\text{inf}})_{\prism}) \xrightarrow{\simeq} \text{qHIG}(A^{\Box(1)})^{\rm qnilp}, \\
&\mathscr{C}^{\text{fp}}((\text{Spf}(R^{(1)})/A_{\text{inf}})_{\prism})  \xrightarrow{\simeq} \text{qMIC}(A^{\Box})^{\rm qnilp}.
\end{align*}
Thus, we see that the functor $\text{ev}_{A^{\Box}}$ is also fully faithful and  induces an equivalence
\[
\mathscr{C}^{\text{fp}}((\text{Spf}(R)/A_{\text{inf}})_{1-\prism}) \xrightarrow{\simeq}  \text{qMIC}(A^{\Box})^{\rm qnilp}.
\]

\phantomsection
\bibliographystyle{alpha}
\bibliography{myrefs}
\addcontentsline{toc}{section}{References}
\nocite{*}

\end{document}